\documentclass[11pt,]{article} 
\usepackage[margin=1in]{geometry} 
\usepackage{setspace}
\usepackage{amsmath}
\usepackage{amsthm}
\usepackage{amssymb}
\usepackage{dsfont}
\usepackage[mathscr]{eucal}
\usepackage{bm}
\usepackage{mathtools}
\usepackage{graphicx}
\usepackage{subcaption}
\usepackage{float}
\usepackage[inline]{enumitem}
\usepackage{url} 
\IfFileExists{microtype.sty}{\usepackage{microtype}}{}
\usepackage{relsize}
\usepackage[dvipsnames]{xcolor}
\usepackage{hyperref}
  \hypersetup{ 
      unicode=false,          
      pdftoolbar=true,        
      pdfmenubar=true,        
      pdffitwindow=false,     
      pdfstartview={FitH},    
      pdftitle={My title},    
      pdfauthor={Author},     
      pdfsubject={Subject},   
      pdfcreator={Creator},   
      pdfproducer={Producer}, 
      pdfkeywords={keyword1} {key2} {key3}, 
      pdfnewwindow=true,      
      colorlinks=true,        
      linkcolor=BrickRed,     
      citecolor=blue,         
      filecolor=magenta,      
      urlcolor=blue,          
  } 
\usepackage{booktabs} 
\usepackage{array}
\usepackage[]{cleveref}
\crefrangeformat{assumption}{assumptions~#3#1#4--#5#2#6}
\Crefrangeformat{assumption}{Assumptions~#3#1#4--#5#2#6}
\crefname{assumption}{Assumption}{Assumptions}
\Crefname{assumption}{Assumption}{Assumptions}
\usepackage{expl3}
\usepackage{color-edits}
  \definecolor{authRed}{HTML}{C0392B}
  \definecolor{authBlue}{HTML}{2980B9}
  \definecolor{authGreen}{HTML}{27AE60}
  \definecolor{authPurple}{HTML}{8E44AD}
  \definecolor{authOrange}{HTML}{D35400}
  \addauthor[CY]{cy}{authBlue}
  \addauthor[JB]{jb}{authRed}
  \addauthor[AA]{aa}{authPurple}
  \addauthor[to]{to}{authOrange}
  \addauthor[ym]{ym}{authGreen}
\usepackage[round,authoryear]{natbib}

\theoremstyle{definition}
\newtheorem{theorem}{Theorem}[section]

\newtheorem{proposition}{Proposition}[section]
\newtheorem{lemma}{Lemma}[section]
\newtheorem*{lemma*}{Lemma}
\newtheorem{assumption}{Assumption}[section]
\newtheorem{definition}{Definition}[section]

\newtheorem{remark}{Remark}[section]
\numberwithin{equation}{section}

\newcommand{\E}{\mathbb E}

\newcommand{\1}{\mathbf 1}
\newcommand{\Var}{\mathrm{Var}}
\newcommand{\Aspace}{\mathscr A}
\newcommand{\calF}{\mathcal F}
\newcommand{\calD}{\mathcal D}

\newcommand{\norm}[1]{\left\lVert #1\right\rVert}

\newcommand{\dotmu}{\dot\mu}
\newcommand{\AJ}{\textsf{AJ}}
\newcommand{\AJCov}{\textsf{AJ-Cov}}
\newcommand{\KM}{\textsf{KM}}
\newcommand{\KMCov}{\textsf{KM-Cov}}
\newcommand{\os}{\textsf{os}}
\newcommand{\cf}{\textsf{cf}}
\newcommand{\plug}{\textsf{pi}}
\newcommand{\Cox}{\textsf{Cox}}

\newcommand{\PathVar}{\operatorname{pvar}}

\newcommand{\alive}{a_{\mathrm{surv}}}
\newcommand{\ezero}{e_0}
\newcommand{\Pev}{\mathbb P_{\mathrm{ev},m}}
\newcommand{\Ptr}{\mathbb P_{\mathrm{tr},n}}
\newcommand{\Pn}{\mathbb P_n}
\newcommand{\cJ}{\mathcal J}

\newcommand{\TV}{\mathrm{TV}}
\newcommand{\PFS}{\mathrm{PFS}}
\newcommand{\OS}{\mathrm{OS}}

\newcommand*{\indep}{\mathrel{\text{{\rotatebox[origin=c]{90}{\resizebox{2.00ex}{1.55ex}{$\vDash$}}}}}}

\ExplSyntaxOn
  \NewDocumentCommand{\invsc}{m}
    {
      \tl_set:Nn \l_tmpa_tl { #1 }
      \regex_replace_all:nnN { ([A-Z]) }
        { \c{textsc}\cB\{ \c{MakeLowercase}\cB\{ \1 \cE\} \cE\} } \l_tmpa_tl
      \l_tmpa_tl
    }
\ExplSyntaxOff
\let\oldblacklozenge\blacklozenge
\renewcommand{\blacklozenge}{{%
  \mathord{\mathchoice%
    {\vcenter{\hbox{\scalebox{0.8}{$\displaystyle\oldblacklozenge$}}}}%
    {\vcenter{\hbox{\scalebox{0.8}{$\textstyle\oldblacklozenge$}}}}%
    {\vcenter{\hbox{\scalebox{0.8}{$\scriptstyle\oldblacklozenge$}}}}%
    {\vcenter{\hbox{\scalebox{0.8}{$\scriptscriptstyle\oldblacklozenge$}}}}%
  }%
}}
\let\oldlozenge\lozenge
\renewcommand{\lozenge}{{%
  \mathord{\mathchoice%
    {\vcenter{\hbox{\scalebox{0.8}{$\displaystyle\oldlozenge$}}}}%
    {\vcenter{\hbox{\scalebox{0.8}{$\textstyle\oldlozenge$}}}}%
    {\vcenter{\hbox{\scalebox{0.8}{$\scriptstyle\oldlozenge$}}}}%
    {\vcenter{\hbox{\scalebox{0.8}{$\scriptscriptstyle\oldlozenge$}}}}%
  }%
}}

  \title{\larger[1]{\textbf{Surrogate-powered Causal Inference \\ with Censored Outcomes}}} 
    \author{\smaller[1]
      {  Yaroslav Mukhin\thanks{
           We thank
           Francesca Molinari,
           Whitney Newey,
           Jana Mareckova,
           Christophe Gaillac,
           Alexander Belloni,
           James Robins,
           Helene Rytgaard,
           Sijia Li,
           Michael Lechner,
           Torben Martinussen,
           Thomas Gerds,
           Thomas Scheike,
           Mats Stensrud,
           Jaap Abbring,
           Bryan Graham,
           Caleb Miles
           for encouragement and helpful comments. All errors are our own.
         }\;\;\footnote{Corresponding author. Center for Data Science for Enterprise and Society, Departmetns of Economics and Statistics, Cornell University, Ithaca NY; \texttt{ym545@cornell.edu}. }
         \quad Tereza Oprea\footnote{Center for Applied Mathematics, Cornell University, Ithaca NY; \texttt{teo28@cornell.edu}. }
         \quad Arielle Anderer\footnote{SC Johnson College of Business, Cornell University, Ithaca, NY; \texttt{aea68@cornell.edu}. }
         \quad Christina Lee Yu\footnote{ School of Operations Research and Information Engineering, Cornell University, Ithaca NY; \texttt{cleeyu@cornell.edu}. }
         \quad Jelena Bradic\footnote{ Department of Statistics and Data Science, Center for Data Science for Enterprise and Society, Cornell University, Ithaca, NY; \texttt{jelena.bradic@cornell.edu}. }
      }
    }
\date{\smaller[1]{\texttt{Sep 27, 2026}}}

\begin{document}
\maketitle
\begin{spacing}{1.15}
\begin{abstract} 
Clinical trials with survival endpoints lose information when participants are censored before death is observed.
We develop target-preserving estimators that use posttreatment disease history, such as recurrence or progression, to recover information lost to censoring for marginal survival and restricted mean survival effects.
The difficulty is that the intermediate event is downstream of treatment: naive adjustment can change the causal estimand, and the useful information enters only through the observed coarsening.
We derive observed-data influence functions with and without recurrence history and obtain an exact gain identity.
The identity shows that efficiency improvement is driven by the censoring hazard, the split of the alive risk set into recurrence states, and the residual-survival separation between those states.
In the no-covariate illness--death model, the Aalen--Johansen estimator realizes the recurrence-augmented efficient score after standardization to the marginal target.
With covariates, correctly specified Cox--Breslow transition hazards provide a root-\(N\) plug-in benchmark, while a hazard-induced one-step estimator gives rate robustness and, under primitive transition and censoring learner rates, canonical inference with a second-order product remainder.
A semi-synthetic metastatic breast cancer study calibrated from digitized progression-free survival and overall survival curves illustrates the gain identity.
The framework applies broadly to censored time-to-event studies with informative intermediate histories.
\end{abstract} 

\newpage

\section{Introduction}\label{sec:introduction}%

Cancer clinical trials are expensive because survival outcomes take time.
A day of an oncology trial has been estimated to cost an average of \$0.54 million.
Yet reviews report that 30--40\% of enrolled patients are censored before their survival outcome is observed
\citep{smith2024new,rosen2020censored,vitale2026censoring,moore2018estimated}.
A censored patient is not uninformative: we know that the patient survived up to the censoring time.
But the trial does not observe the terminal event that the study was designed to compare.
This loss of information matters because the sample size required to detect a treatment effect scales with the efficiency bound for the target estimand.
More information per enrolled patient means more power, shorter studies, and faster evidence.
The same statistical problem appears well beyond oncology, whenever causal effects are defined through censored event times, including public-health interventions, labor-market spells, startup lifetimes, and other duration outcomes
\citep{hernan00,honore2006bounds,honore2018new,deaner2024causal,hellmann02,alvarez2024decomposing}.

Causal inference with a survival endpoint is a compound missing-data problem.
Under the potential-outcomes framework, each unit has potential survival times $T^0,T^1$, but treatment assignment reveals at most one of them
\citep{neyman23,rubin05,rubin76,ding18}.
Survival data add a second coarsening: even the factual event time may be hidden by right censoring.
If $C$ denotes the censoring time, the observed data contain $\widetilde T=T\wedge C$ and $\Delta=I(T\le C)$ rather than necessarily observing $T$ itself
\citep{andersen12}.
Thus a survival trial has two statistical tasks.
The first is identification: the analyst needs assumptions on treatment assignment and censoring that connect the observed law to the marginal potential-outcome estimand.
The second is precision: after identification, the analyst wants to recover as much of the information lost to censoring as possible.
This paper is primarily about the second task, while also showing how richer disease history can make the censoring condition more plausible.

The guiding principle is simple: auxiliary data improve efficiency only when they recover information hidden by sampling or censoring.
If $T_i^a$ is fully observed, measuring auxiliary variables on each unit does not improve first-order efficiency for estimating a functional of the marginal distribution $P_{T^a}$.
This holds regardless of the strength of its association with $T^a_i$.
The outcome itself is already the relevant contribution to the marginal parameter.
Gains become possible when some contributions are missing or coarsened.
Then auxiliary variables improve precision by refining the conditional mean of the missing residual contribution.
This information-recovery view underlies augmented estimators for missing and censored data 
\citep{robins92,robins94,robins95,murray1996nonparametric}.

Covariate adjustment is the familiar application of this idea.
To estimate an arm-specific parameter $\theta^a$, units assigned to the opposite arm have $T_i^a$ missing.
But their baseline covariates $X_i$ are observed.
Augmenting the arm-specific estimator with $X_i$ can recover part of the missing counterfactual information and improve precision
\citep{hahn98,robins95,li2023estimating}.
This is safe for total-effect estimands because $X_i$ is measured before treatment and is therefore not a mediator.
The same $X$-based information helps estimate both $\theta^0$ and $\theta^1$.
In contrast, the observed outcome under the opposite treatment, $T_i^{1-a}$, cannot be used as ordinary auxiliary information for $\theta^a$: it is never observed jointly with $T_i^a$, and using it would effectively change the problem into estimating a treatment contrast.
The limitation in survival studies is that baseline covariates are frozen at time zero.
Long-run survival often depends on post-randomization disease events, leaving substantial conditional variation in $T^a$ even after rich baseline adjustment.

This motivates using a posttreatment intermediate event time $R^a$.
In oncology, $T^a$ and $R^a$ may be time to death and time to progression.
Knowing whether progression has occurred before a censoring time can strongly change the conditional residual-survival probability for the event $\{T^a>\tau\}$.
Our goal is to use this later disease-state information to recover power lost to censoring while preserving the marginal survival estimand.
The target remains, for example,

\begingroup
\setstretch{1}
\setlength{\jot}{0pt}
\begin{align}\label{eq:parameters}
  \theta^a_\tau=P(T^a>\tau),
  \qquad
  \mu^a_\tau=E(T^a\wedge \tau)=\int_0^\tau P(T^a>t)\,dt .
\end{align}
\endgroup
The intermediate event is not promoted to the endpoint.
It is used only as observed-data information about the terminal event.

A posttreatment variable cannot be used like a baseline covariate without changing the causal question.
The counterfactual recurrence time $R^{1-A}$ is missing by definition, and the factual recurrence time is itself observed only until death or censoring.
Thus recurrence cannot be used to borrow information across treatment arms in the same way as pretreatment $X$.
Moreover, including recurrence as a time-dependent covariate in a treatment-effect Cox model generally changes the estimand: since recurrence is affected by treatment, conditioning on it targets a conditional or direct-effect-like quantity rather than the marginal total effect on survival.
This target drift is entangled with the noncollapsibility and limited causal interpretation of hazard-ratio estimands 
\citep{hernan10,luo25}.
Our construction avoids this problem.
It uses recurrence armwise, inside the observed-data score, and then standardizes back to the marginal survival or restricted-mean estimand.

There is also an identification opportunity.
Causal survival analysis requires an assumption on the censoring mechanism.
If dropout depends on latent survival beyond the observed history used in the analysis, a marginal treatment-effect estimator can be inconsistent.
Recording posttreatment disease events expands the observed history: instead of requiring censoring to be independent of latent death given only baseline variables and being alive, one can condition on the current recurrence state.
This formulation makes explicit which disease information must be observed for the censoring condition to be plausible, while allowing for the possibility of unmeasured dropout. In oncology, where censoring bias can distort estimated treatment effects, this distinction is practically important 
\citep{rosen2020censored,vitale2026censoring}.
We focus on efficiency, but the same representation clarifies how recurrence history can relax the censoring assumption.

Our contribution is threefold.
First, we show that recurrence history has no automatic efficiency value for marginal survival inference: without coarsening or model restrictions, it cannot improve first-order efficiency.
Second, we show that right censoring changes this conclusion.
We derive the observed-data influence functions (IFs) with and without recurrence history and obtain an exact efficiency-gain identity.
The gain is large when censoring hazard is high, recurrence splits the alive risk set into well-populated states, those states have different conditional residual-survival means, and sufficient subjects remain alive at the relevant times.
Third, we develop estimators that attain the resulting efficiency bounds.
In the no-covariate illness--death model, the Aalen--Johansen (AJ) estimator realizes the recurrence-augmented efficient score after standardization to the marginal survival or restricted mean survival target.
With covariates, correctly specified Cox--Breslow transition hazards give a root-$n$ plug-in benchmark, while our one-step residual-augmented estimator accommodates flexible conditional-hazard learners.
The one-step correction has two roles: it gives a rate-robust mean when censoring is correctly learned, and it yields canonical root-$N$ inference when the transition and censoring errors satisfy the usual product-rate condition.

\subsection{Related literature}\label{subsec:related-literature}

Semiparametric missing-data theory shows how inverse weighting, augmentation, and observed-data projection recover information lost to coarsening
\citep{robins92,robins94,robins95}.
Our setting is different from a routine augmented inverse-probability-weighting problem.
The auxiliary information is a stopped posttreatment illness--death history: recurrence lies downstream of treatment, can be precluded by death, and is observed only until censoring.
The main step is therefore to project the full-data survival score $I(T>\tau)-\theta$ onto the observed filtration generated by the censored transition processes.
The task is distinct from augmenting with a fully observed covariate regression.
This projection yields an inverse-censoring-weighted martingale integral; recurrence changes the efficient influence function by replacing the alive-only residual-survival mean with the state-specific means.  
The resulting gain is therefore a structural consequence of censoring, state occupancy, and conditional-mean separation, not a generic augmentation term.

Our work is distinct from surrogate endpoint validation.
Classical surrogate endpoint methods ask whether an intermediate endpoint can validate, replace, or stand in for the clinical endpoint 
\citep{prentice1989surrogate,fleming96,buyse2000validation}.
Surrogate-index and limited-outcome methods use short-term outcomes to learn long-term effects under surrogacy, comparability, or missing-primary-outcome assumptions \citep{athey25,kallus25}.  
We do not make recurrence the endpoint and do not estimate a surrogate-index effect.
The target remains marginal survival or restricted mean survival.
Recurrence enters only armwise, inside the observed-data score, and is then integrated out.
The goal is target-preserving information recovery, not surrogate validation.

The distinction from covariate adjustment is equally important.
Baseline covariates can improve precision because they are observed before treatment and can be used in both arms without changing the causal estimand
\citep{hahn98,robins95}.
Semi-supervised learning, prediction-powered inference and its refinements have the same target-preservation logic: auxiliary predictions are useful only after a correction restores validity for the original target
\citep{zhang2022high,angelopoulos2023ppi++,zrnic24crossppi}.
Our correction takes the form of a martingale residual adjustment induced by the fitted illness--death transition law.
This construction differs from calibration of black-box predictions on unlabeled data and is essential because recurrence is a posttreatment variable.
Naively conditioning on it in a treatment-effect model can change a marginal total-effect question into a conditional or direct-effect-like question.
Furthermore, in survival settings, hazard ratios are noncollapsible and sensitive to conditioning choices
\citep{hernan10,martinussen2022causality,luo25}.

\section{Problem setup}\label{sec:problem-setup}%

Let $A\in\{0,1\}$ be treatment and let $X$ be baseline covariates.
For arm $a$, let $(T^a,R^a,C^a)$ denote the potential death time, recurrence time, and censoring time, following the potential-outcomes notation of \citet{neyman23} and \citet{rubin05}.
The semi-competing-risk convention is $R^a<T^a$ if recurrence occurs before death, while $R^a=\tau_0$ denotes no recurrence before death or the study horizon.
The factual variables are $(T,R,C)=(T^A,R^A,C^A)$, and the observed coarsening is

\begingroup
\setstretch{1}
\setlength{\jot}{0pt}
\begin{align} \label{eq:C-coarsening-map}
  (\tilde T,\Delta)&=(T\wedge C,\,I(T\le C)),
  \qquad
  (\tilde R,\Gamma)=(R\wedge \tilde T,\,I(R\le \tilde T)).
\end{align}
\endgroup
For each unit, the observed data are represented by the six-tuple $ O=(\tilde T,\Delta,\tilde R,\Gamma,A,X) $.
The theory below is developed armwise; treatment effects are obtained by applying the same construction in each randomized or ignorable treatment arm and taking contrasts.

\begin{assumption} \label{ass:time-reg}
For each arm $a$, $(T^a,R^a,C^a)$ is supported on $[0,\tau_0]^3$ and is jointly absolutely continuous on $[0,\tau_0)^3$.
The death time $T^a$ is continuous, so $P(T^a=\tau_0)=0$, while point masses at $\tau_0$ are allowed for $R^a$ and $C^a$.
\end{assumption}

For an event time $U$, write $S_U(t)=P(U>t)$, $F_U=1-S_U$, and $\lambda_U=f_U/S_U$ when a density exists.
We write $S\equiv S_T$, $G\equiv S_C$, $\lambda\equiv\lambda_T$, and $\lambda_{01}\equiv\lambda_R$ in a fixed arm.
Throughout, $\theta$ and $\mu_\tau$ denote a generic single-arm survival probability and restricted mean survival time (RMST), \eqref{eq:parameters}, $0<\tau<\tau_0$.
Censoring coarsens the endpoint without eliminating all prognostic information carried by the subject.
At the censoring time, the observed history still contains both the partial death history and the intermediate disease state.
The latter is inferentially relevant only through its effect on the conditional residual-survival mean at that time.
Hence recurrence can improve efficiency only when censoring occurs with nonnegligible probability in regions where the current disease state predicts subsequent survival.
We formalize this censoring-supported information-recovery principle and characterize when recurrence recovers information about the survival endpoint lost to censoring.

\begin{assumption} \label{ass:A-ignore}
Consistency holds, $(T,R,C)=(T^A,R^A,C^A)$ almost surely. 
Moreover, $A\indep(T^0,T^1,R^0,R^1,C^0,C^1)\mid X$, with $e\le P(A=1\mid X)\le1-e$ a.s. for some $e>0$.
\end{assumption}

\begin{assumption} \label{ass:C-ignore}
For each arm $a$, $C^a\indep (T^a,R^a)\mid X$. 
\end{assumption}

Together with Assumptions~\ref{ass:time-reg}  -- \ref{ass:A-ignore},
Assumption \ref{ass:C-ignore} identifies the recurrence and death transition hazards from the coarsened observations,
as $S^a(\tau) =E\exp\left\{-\int_0^\tau \lambda^a(t\mid X)dt\right\}$, and $ \lambda^a(t\mid X)dt =P\{t\le \tilde T<t+dt\mid A=a,\,\tilde T\ge t,\,X\} $.
Identification reduces the causal problem to an armwise censored survival problem.

\section{Surrogate-powered efficiency} \label{sec:efficiency}%

Under an unrestricted joint law, $R$ provides no first-order information for a marginal functional of fully observed $T$.

\begin{theorem}\label{thm:no-gain}
Suppose $(T,R)$ is fully observed, $P_{R\mid T}$ is unrestricted,
and $\psi(P)=\psi(P_T)$ is a scalar pathwise differentiable parameter with marginal IF $\varphi_T$.
Then
$ \varphi_{\mathrm{joint}}(T,R)=\varphi_T(T). $
\end{theorem}
With censoring, $R$ can improve efficiency by refining the projection of the unobserved residual-survival contribution onto the history available at censoring.
The semiparametric definitions and proof are deferred to Appendix~\ref{app:full-data-efficiency};
see also the general efficiency frameworks of
\citep{ibragimov1981statistical,newey1990semiparametric,vandervaart98}.
The theorem separates the full-data and coarsened-data problems: any efficiency gain derived below arises from coarsening, not from imposing a surrogate-validity restriction on the joint law.

\subsection{Efficient influence functions}%

Fix a treatment arm and suppress $X$.
We derive observed-data scores using death history alone and using recurrence history as well.
Recurrence changes the information about subsequent death without itself being the endpoint.
We therefore track whether a patient is alive before recurrence, alive after recurrence, or dead through the illness--death process

\begingroup
\setstretch{1}
\setlength{\jot}{0pt}
\begin{align} \label{eq:S-state}
  \mathcal S(t) =0\,I(t<R\wedge T)+1\,I(R\le t<T)+2\,I(T\le t), \qquad \mathcal S(0)=0
\end{align}
\endgroup
\citep{aalen78,andersen12}.
The state space is ordered by the transitions $0\to1$, $0\to2$, and $1\to2$, where state $0$ denotes alive and recurrence-free, state $1$ denotes alive after recurrence, and state $2$ denotes death.
Thus the transition $0\to1$ updates the observed alive-state history, whereas the transitions $0\to2$ and $1\to2$ determine the death endpoint.
Censoring acts as a coarsening mechanism on this event history: it stops observation of the current alive state rather than defining an additional clinical state.

Let $\cJ=\{01,02,12\}$.
For $\ell=(j,k)\in\cJ$, write $j(\ell)=j$ and $k(\ell)=k$.
In the full data, the alive risk set at time $t$ is partitioned into those who have not yet recurred, $Y_0^*(t)=I(t\le R,\,t\le T)$, and those who have recurred, $Y_1^*(t)=I(R<t\le T)$.
Hence $Y_T^*(t)=Y_0^*(t)+Y_1^*(t)$.
The full-data transition counts distinguish recurrence before death,
death before recurrence, and death after recurrence, respectively:
$N_{01}^*(t) =I(R\le t,\,R<T)$,
$N_{02}^*(t) =I(T\le t,\,T<R)$,
$N_{12}^*(t) =I(R<T\le t)$.
The full-data recurrence risk indicator and event count are
$Y_R^*(t)=Y_0^*(t)$ and $N_R^*(t)=N_{01}^*(t)$.
Death can occur from either alive state, so
$N_T^*(t)=N_{02}^*(t)+N_{12}^*(t)$.
Finally, $N_C^*(t)=I(C\le t)$ counts censoring.

For $0\le t\le \tau$, define the augmented and marginal event-history filtrations

\begingroup
\setstretch{1}
\setlength{\jot}{0pt}
\begin{align} \label{eq:F-filtrations}
  \mathcal F^\blacklozenge(t)
  =
  \sigma\{N_T^*(s),N_R^*(s):s\le t\},
  \qquad
  \mathcal F^\bullet(t)
  =
  \sigma\{N_T^*(s):s\le t\}.
\end{align}
\endgroup
The filtration $\mathcal F^\blacklozenge(t)$ records both the recurrence and death histories, whereas $\mathcal F^\bullet(t)$ records only the death history.
Under Assumption~\ref{ass:C-ignore}, adding censoring history to the conditioning information leaves the residual-survival means below unchanged. 
For $\theta=P(T>\tau)$, recall the full-data IF $\phi^*(T)=I(T>\tau)-\theta$.
The recurrence information relevant for the marginal endpoint is the projection contrast

\begingroup
\setstretch{1}
\setlength{\jot}{0pt}
\begin{align} \label{eq:IF-contrast}
  \Delta_\theta(t)
  =
  M_\theta^\blacklozenge(t)-M_\theta^\bullet(t),
  \quad   M_\theta^\blacklozenge(t)
  =
  E\{\phi^*(T)\mid \mathcal F^\blacklozenge(t)\}, \  M_\theta^\bullet(t)
  =
  E\{\phi^*(T)\mid \mathcal F^\bullet(t)\}.  
\end{align}
\endgroup
The contrast vanishes when recurrence adds no information about residual survival beyond death history.
Under censoring, any efficiency gain depends on its value at censoring times.

\begin{assumption}  \label{ass:Markov-state}
The current state suffices for residual survival:
 for   $0\le t\le \tau$ and $j=0,1$,
$
  E \{I(T>\tau)\mid \mathcal F^\blacklozenge(t) \}
  =
  v_0(t)Y_0^*(t+)+v_1(t)Y_1^*(t+), $
with
$ v_j(t)=P\{T>\tau\mid \mathcal S(t)=j\}.
$
\end{assumption}

Under Assumption~\ref{ass:Markov-state}, recurrence affects residual
survival through the current state. Suppose the transitions
$0\to1$, $0\to2$, and $1\to2$ have state-dependent hazards
$\lambda_{01}$, $\lambda_{02}$, and $\lambda_{12}$, respectively.
For illness--death transitions $jk\in \cJ$, define the $\mathcal F^\blacklozenge$-martingales $ M^\blacklozenge_{jk}(t) =N^*_{jk}(t)-\int_0^tY_j^*(u)\lambda_{jk}(u)\,du. $
The death count $N_T^*=N^*_{02}+N^*_{12}$ has
$\mathcal F^\bullet$-martingale $ M_T^\bullet(t) =N_T^*(t)-\int_0^tY_T^*(u)\lambda_T(u)\,du, $
where $\lambda_T$ is its marginal death hazard.
The state-specific residual-survival probabilities satisfy

\begingroup
\setstretch{1}
\setlength{\jot}{0pt}
\begin{align*}%
  \frac{d}{dt}v_1(t)=\lambda_{12}(t)v_1(t),
 \qquad \mbox{and} \qquad
  \frac{d}{dt}v_0(t)
  =
  \lambda_{01}(t)\{v_0(t)-v_1(t)\}
  +
  \lambda_{02}(t)v_0(t),
\end{align*}%
\endgroup
with the terminal conditions $v_0(\tau)=v_1(\tau)=1$ and $v_2(t)=0$.
The transition hazards are otherwise unrestricted.
Under Assumption~\ref{ass:Markov-state}, and with $ v(t)=P\{T>\tau\mid T>t\} $, 

\begingroup
\setstretch{1}
\setlength{\jot}{0pt}
\begin{equation}
  \label{eq:M-blacklozenge}
  M_\theta^\blacklozenge(t)
 =
  \sum_{jk\in\cJ}\int_0^t\{v_k(s)-v_j(s)\}\,
  dM_{jk}^\blacklozenge(s)
  =
  \sum_{j=0}^1 v_j(t)Y_j^*(t+)-\theta,
\end{equation}
\begin{equation}
  \label{eq:M-bullet}
  M_\theta^\bullet(t)
  =
  \int_0^t\{0-v(s)\}\,dM_T^\bullet(s)
 =
  v(t)Y_T^*(t+)-\theta.
\end{equation}
\endgroup

The marginal analysis assigns $v(t)$ to all survivors,
whereas the recurrence-augmented analysis assigns $v_0(t)$ to recurrence-free survivors and $v_1(t)$ to post-recurrence survivors.
Their difference is

\begingroup
\setstretch{1}
\setlength{\jot}{0pt}
\begin{align}\label{eq:IF-contrast2}
\Delta_\theta(t)
=\{v_0(t)-v(t)\}Y_0^*(t+)
 +\{v_1(t)-v(t)\}Y_1^*(t+).
\end{align}
\endgroup
A gain is possible where both alive states have positive probability,
their residual-survival means differ, and censoring can occur.
The martingale property of $\Delta_\theta$ is not immediate from \eqref{eq:IF-contrast}--\eqref{eq:IF-contrast2}; the integral representations \eqref{eq:M-blacklozenge} and \eqref{eq:M-bullet} establish it.

The observed processes and filtrations are stopped at $\tilde T=T\wedge C$:

\begingroup
\setstretch{1}
\setlength{\jot}{0pt}
\begin{align}%
  \label{eq:observed-sa}
  \mathcal F^\lozenge(t)=\mathcal F^\blacklozenge(t\wedge\tilde T)
  \qquad \text{and} \qquad
  \mathcal F^\circ(t)=\mathcal F^\bullet(t\wedge\tilde T) .
\end{align}%
\endgroup
The observed at-risk and events processes are:
$ Y(t)=Y_T^*(t)Y_C^*(t), $
$ Y_j(t)=Y_j^*(t)Y_C^*(t),$
$ N_T(t)=N_T^*(t\wedge\tilde T), $
$ N_{jk}(t)=N_{jk}^*(t\wedge\tilde T). $
If death is observed, the endpoint is known.
If censoring occurs first, the observed-data IF inserts a conditional mean.
The $\mathcal F^\circ$-score inserts the alive-only conditional mean;
the $\mathcal F^\lozenge$-score inserts the state-specific conditional mean.
 
\begin{assumption} 
\label{ass:eif-gain-primitives}
Assumptions~\ref{ass:time-reg},
\ref{ass:C-ignore}, and~\ref{ass:Markov-state} hold, $ \inf_{0\le t\le \tau}G(t)>0, $ and   $M_{jk}^\blacklozenge$, $jk\in\cJ$,  $M_T^\bullet$, and the censoring martingale $M_C$ are square integrable on $[0,\tau]$.
\end{assumption}

Under Assumptions~\ref{ass:time-reg} and~\ref{ass:C-ignore}, censoring is orthogonal to the event martingales.
In particular,
$
  \langle M_C,M_{jk}^\blacklozenge\rangle(t)=0,
  \  jk\in\cJ,
$
and
$
  \langle M_C,M_T^\bullet\rangle(t)=0,
  \  0\le t\le\tau .
$
This orthogonality is used below to separate censoring variation from recurrence and death variation in the observed-data efficient scores.

\begin{theorem}
\label{thm:eif-main}
Under Assumption~\ref{ass:eif-gain-primitives}, the influence functions for $\theta=P\{T>\tau\}$ are

\begingroup
\setstretch{1}
\setlength{\jot}{0pt}
\begin{align}
  \label{eq:IF-star}
  \phi^*(T)
  &=I(T>\tau)-\theta
  =\int_0^\tau dM^\bullet_\theta(t)
  =\int_0^\tau dM^\blacklozenge_\theta(t),
  \\
  \label{eq:IF-Delta}
  \phi^\Delta(\tilde T,\Delta)
  &=\frac{\Delta}{G(\tilde T)}\{\phi^*(\tilde T)-\phi_0\}+\phi_0,
  \\
  \label{eq:IF-circ}
  \phi^\circ(\tilde T,\Delta)
  &=\int_0^\tau \frac{Y_C^*(t)}{G(t)}\,dM^\bullet_\theta(t)
  =\frac{\Delta_{C,\tau}}{G(\tilde T\wedge\tau)}\phi^*(\tilde T)
    +\int_0^{\tilde T\wedge\tau}M^\bullet_\theta(t-)\frac{dM_C^\bullet(t)}{G(t)},
  \\
  \label{eq:IF-lorenge}
  \phi^\lozenge(\tilde T,\Delta,\tilde R,\Gamma)
  &=\int_0^\tau \frac{Y_C^*(t)}{G(t)}\,dM^\blacklozenge_\theta(t)
  =\frac{\Delta_{C,\tau}}{G(\tilde T\wedge\tau)}\phi^*(\tilde T)
    +\int_0^{\tilde T\wedge\tau}M^\blacklozenge_\theta(t-)\frac{dM_C^\blacklozenge(t)}{G(t)},
\end{align}
\endgroup
where $\Delta_{C,\tau}=I\{C>T\wedge\tau\}$ and $\phi_0=E\{\phi^*(T)\omega(T)\}$.
\end{theorem}

The superscripts $*,\Delta,\circ,\lozenge$ denote, respectively, the full, missing-outcome, marginal-censored, and recurrence-augmented-censored data.
The theorem separates three sources of variation.
In the full-data experiment, recurrence does not alter the IF for the marginal survival probability:
$ \phi^*(T) $ is obtained under both the marginal death filtration and the recurrence-augmented filtration.
Second, right censoring replaces the unavailable full-data contribution after the censoring time by a predictable projection on the observed history.
The marginal censored-data IF $\phi^\circ$ is the usual right-censoring IF written in martingale form.
Third, the recurrence-augmented IF $\phi^\lozenge$ has the same inverse-censoring-weighted complete contribution as $\phi^\circ$. 
The only change is in the augmentation term: the marginal residual-survival process $M_\theta^\bullet(t-)$ is replaced by the finer projection $M_\theta^\blacklozenge(t-)$, which conditions on the current illness--death state.
Hence recurrence can improve efficiency only by improving the prediction of the missing residual contribution at censoring times.
In particular, there is no gain without censoring, and there is no gain when $ M_\theta^\blacklozenge(t)=M_\theta^\bullet(t) $ on the time region where censoring occurs.
Under the current-state formulation, this condition is equivalent to the absence of a relevant state-specific residual-survival contrast between recurrence-free and post-recurrence survivors.
We state and interpret that comparison in \Cref{subsec:gain-snapshot}.

\subsection{No-covariate AJ estimation}%
\label{subsec:estimation}

The recurrence-augmented score suggests a natural estimator: estimate the transition law of the observed illness--death process and standardize the resulting state-occupation probabilities to the marginal survival or RMST target.
In the no-covariate model, this construction coincides with the classical AJ estimator 
\citep{aalen78,andersen12}.
The estimator learns the transition dynamics $0\to1$, $0\to2$, and $1\to2$, and then averages over the induced paths to recover the marginal death-time functional.
It can therefore be viewed as the empirical counterpart of the recurrence-augmented efficient score.

Let $A_{jk}(t)=\int_0^t\lambda_{jk}(u)\,du$ and define the cumulative-hazard matrix and transition matrix

\begingroup
\setstretch{1}
\setlength{\jot}{0pt}
\begin{align} \label{eq:matrix-product-integral}
  d\mathbf A(t)=
  \displaystyle
  \begin{psmallmatrix}
  \displaystyle
  -dA_{01}(t)-dA_{02}(t) &
  \displaystyle
  dA_{01}(t) &
  \displaystyle
  dA_{02}(t)\\
  \displaystyle
  0 &
  \displaystyle
  -dA_{12}(t) &
  \displaystyle
  dA_{12}(t)\\
  \displaystyle
  0 &
  \displaystyle
  0 &
  \displaystyle
  0
  \end{psmallmatrix},
  \quad
  \mathbf P(s,t) = \prod_{(s,t]} \big\{I+d\mathbf A(u) \big\},
\end{align}
\endgroup
with all increments evaluated at $t$.
Since $\mathcal S(0)=0$, the state-occupation probabilities are $P_j(t)=P_{0j}(0,t)$, and the survival and RMST estimands are

\begingroup
\setstretch{1}
\setlength{\jot}{0pt}
\begin{align*}%
  \theta(t)=P_0(t)+P_1(t),
  \qquad
  \mu_\tau=\int_0^\tau\theta(t)\,dt.
\end{align*}%
\endgroup
For $u\le t$, write $v_0(u,t)=P_{00}(u,t)+P_{01}(u,t)$ and $v_1(u,t)=P_{11}(u,t)$ for survival to $t$ conditional on the state at $u$.
Their integrals $V_j(u,t)=\int_u^t v_j(u,s)\,ds$ give the corresponding expected survival time over $[u,t]$.

For $n$ subjects in a fixed treatment arm,
let $N_{jk,i}$ and $Y_{j,i}$ denote subject $i$'s observed transition count and risk indicator,
and write $\bar N_{jk}=n^{-1}\sum_{i=1}^nN_{jk,i}$ and $\bar Y_j=n^{-1}\sum_{i=1}^nY_{j,i}$.
For $jk\in\cJ$, the Nelson--Aalen (NA) estimator and the AJ transition matrix $(\widehat P_{jk}(s,t))$ are

\begingroup
\setstretch{1}
\setlength{\jot}{0pt}
\begin{align*}
  \widehat A_{jk}^{\textsf{AJ}}(t) =\int_0^t\frac{I\{\bar Y_j(u)>0\}}{\bar Y_j(u)} \,d\bar N_{jk}(u),
  \qquad
  \widehat{\mathbf P}(s,t) =\prod_{(s,t]}\{I+d\widehat{\mathbf A}^{\textsf{AJ}}(u)\},
\end{align*}
\endgroup
with the ratio defined as zero when $\bar Y_j(u)=0$.
Replacing $A_{jk}$ by $\widehat A_{jk}^{\textsf{AJ}}$ in $\mathbf A$ defines $\widehat{\mathbf A}^{\textsf{AJ}}$.
Since $\mathcal S(0)=0$, the AJ survival and RMST estimators are

\begingroup
\setstretch{1}
\setlength{\jot}{0pt}
\begin{equation}\label{eq:AJ}
\widehat\theta(t)
=\widehat P_{00}(0,t)+\widehat P_{01}(0,t),
\qquad
\widehat\mu_\tau^{\textsf{AJ}}
=\int_0^\tau\widehat\theta(t)\,dt.
\end{equation}
\endgroup

\begin{assumption}
\label{ass:aj-main-integrated}
Assumption~\ref{ass:eif-gain-primitives} holds.
For $\ell=jk\in\cJ$, $M_{\ell,i}(t)=N_{\ell,i}(t)-\int_0^tY_{j,i}(u)\,dA_{\ell}(u)$ is square integrable.
Write $q_j(u)=E\{Y_j(u)\}=P_j(u-)G(u-)$, $j=0,1$.
\begin{enumerate}[label=\textnormal{(AJ\arabic*)},leftmargin=2em]
\item\label{ajmain:support}
  With $\mathcal T_j=\{u\in(0,\tau]:q_j(u)>0\}$, the infimum of $q_j$ on every compact subset of $\mathcal T_j$ is positive,
  and $\int_0^\tau I\{q_{j(\ell)}(u)=0\}\,dA_\ell(u)=0$.
\item\label{ajmain:inverse}
  For each $\ell$,
\[
   V_\ell^\textsf{AJ}(\tau)= \int_{(0,\tau]}\frac{I\{q_{j(\ell)}(u)>0\}}{q_{j(\ell)}(u)}\,dA_\ell(u)<\infty.
\]
\end{enumerate}
\end{assumption}

These conditions impose no uniform lower bound on the state-specific risk probabilities;
in particular, $q_1$ may vanish near time zero because recurrence has not yet occurred.
Without uniform positivity, the NA proof requires additional control.
We localize to $\{q_j\ge\varepsilon\}$, where standard linearization applies, and bound the remaining martingale contribution through
$\int_{\{0<q_j(u)<\varepsilon\}}q_j(u)^{-1}\,dA_\ell(u)$,
$j=j(\ell)$.
This integral vanishes as $\varepsilon\downarrow0$ by \ref{ajmain:inverse},
which also ensures that the contribution from empty empirical risk sets is negligible.
Condition~\ref{ajmain:support} ensures that $dA_\ell$ assigns no mass to $\{q_{j(\ell)}=0\}$.

Integrating the recurrence-augmented survival EIF in \eqref{eq:IF-lorenge} over the time horizon gives the EIF for RMST,
denoted by $\phi_\mu^\lozenge$: 

\begingroup
\setstretch{1}
\setlength{\jot}{0pt}
\begin{equation} \label{eq:main-nocov-rmst-eif}
  \phi_\mu^\lozenge(O)=\sum_{\ell\in\cJ}\int_0^\tau H_\ell^\lozenge(u)dM_\ell(u),
\end{equation}
where
\begin{equation} \label{eq:main-nocov-rmst-H}
  H_{01}^\lozenge(u)=\frac{V_1(u,\tau)-V_0(u,\tau)}{G(u-)},
  \quad
  H_{02}^\lozenge(u)=-\frac{V_0(u,\tau)}{G(u-)},
  \quad
  H_{12}^\lozenge(u)=-\frac{V_1(u,\tau)}{G(u-)}.
\end{equation}
\endgroup
Each integrand is the change in expected remaining-survival time at the corresponding transition,
importance-weighted by $G(u-)^{-1}$ to account for censoring \citep{robins94}.
Recurrence changes this expectation from $V_0(u,\tau)$ to
$V_1(u,\tau)$, whereas death from state $j$ reduces
$V_j(u,\tau)$ to zero.

\begin{theorem} \label{thm:aj-integrated-main}
Under Assumption~\ref{ass:aj-main-integrated}, the AJ RMST estimator satisfies $\widehat\mu^\textsf{AJ}_\tau-\mu_\tau = (\mathbb P_n-P)\phi_\mu^\lozenge + o_p(n^{-1/2})$. 
If $0<\sigma_\mu^2=\Var\{\phi_\mu^\lozenge(O)\}<\infty$, then $ \sqrt n(\widehat\mu^\textsf{AJ}_\tau-\mu_\tau)\rightsquigarrow N(0,\sigma_\mu^2), $ and the AJ estimator attains the semiparametric efficiency bound.
\end{theorem}

Recurrence history changes the observed-data efficient score.
It replaces the marginal residual-survival mean by state-specific residual-survival means for recurrence-free and post-recurrence survivors.
Thus, in this model, the AJ estimator is the empirical implementation of the optimal observed-data projection.
The comparison with Kaplan--Meier (KM) is therefore a comparison between two projections: one that conditions only on the death history, and one that also uses the disease-state history at the censoring times where the efficiency bound depends on it. 

\begin{remark}%
The Markovian trajectories part of \Cref{ass:Markov-state} is a semiparametric restriction but is not the source of the gain described here.
The IFs of \Cref{thm:eif-main} and the gain identity \eqref{eq:gain} below do not require Markovian transitions.
With non-Markovian data, the AJ remains consistent  and more efficient than the KM, but is not semiparametrically efficient 
\citep{datta2001validity, munch2023targeted}.
\end{remark}%

\subsection{When is the gain largest?} \label{subsec:gain-snapshot}%

We now return to the variance comparison.
Under Assumption~\ref{ass:eif-gain-primitives},

\begingroup
\setstretch{1}
\setlength{\jot}{0pt}
\begin{align}%
  \label{eq:gain}
  \Var(\phi^\circ) - \Var(\phi^\lozenge)
  =\int_0^\tau \frac{\lambda_C(t)S_T(t)}{G(t)}
  \Var\big\{M^\blacklozenge_\theta(t-)-M^\bullet_\theta(t-)\mid T\ge t\big\}\,dt.
\end{align}%
\endgroup
The efficiency-gain identity in \eqref{eq:gain} gives the design rule for when the augmented AJ estimator should improve most over the marginal KM estimator.
Under \Cref{ass:Markov-state}  we can write,

\begingroup
\setstretch{1}
\setlength{\jot}{0pt}
\begin{align} \label{eq:M-theta-variance}%
  \Var\big\{M^\blacklozenge_\theta(t-)-M^\bullet_\theta(t-)\mid T\ge t\big\}
  =\frac{P_{00}(0,t)P_{01}(0,t)}{S_T(t)^2}\{v_0(t)-v_1(t)\}^2.
\end{align}%
\endgroup
Thus recurrence helps precisely when censoring occurs at times when recurrence has split the alive population and the two alive states have different residual survival.
The gain is largest when the following four conditions hold concurrently: 
\begin{enumerate}[label=\textnormal{(\roman*)},leftmargin=2em]
\item censoring is appreciable at times that matter for the target, so that $\lambda_C(t)/G(t)$ is large;
\item recurrence has already split the alive population into two non-negligible groups, so that both $P_{00}(0,t)/S_T(t)$ and $P_{01}(0,t)/S_T(t)$ are substantial;
\item the two alive states have sharply different residual survival, so that $|v_1(t)-v_0(t)|$, or $|V_1(t,\tau)-V_0(t,\tau)|$ for RMST, is large;
\item enough subjects remain alive, so that the factor $S_T(t)$ has not already collapsed.
\end{enumerate}
Thus, weak censoring, rare recurrence, very late recurrence, non-prognostic recurrence, or near-complete mortality all force the gain to be small.

\begin{figure}[H]%
  \centering
  \includegraphics[width=0.92\linewidth]{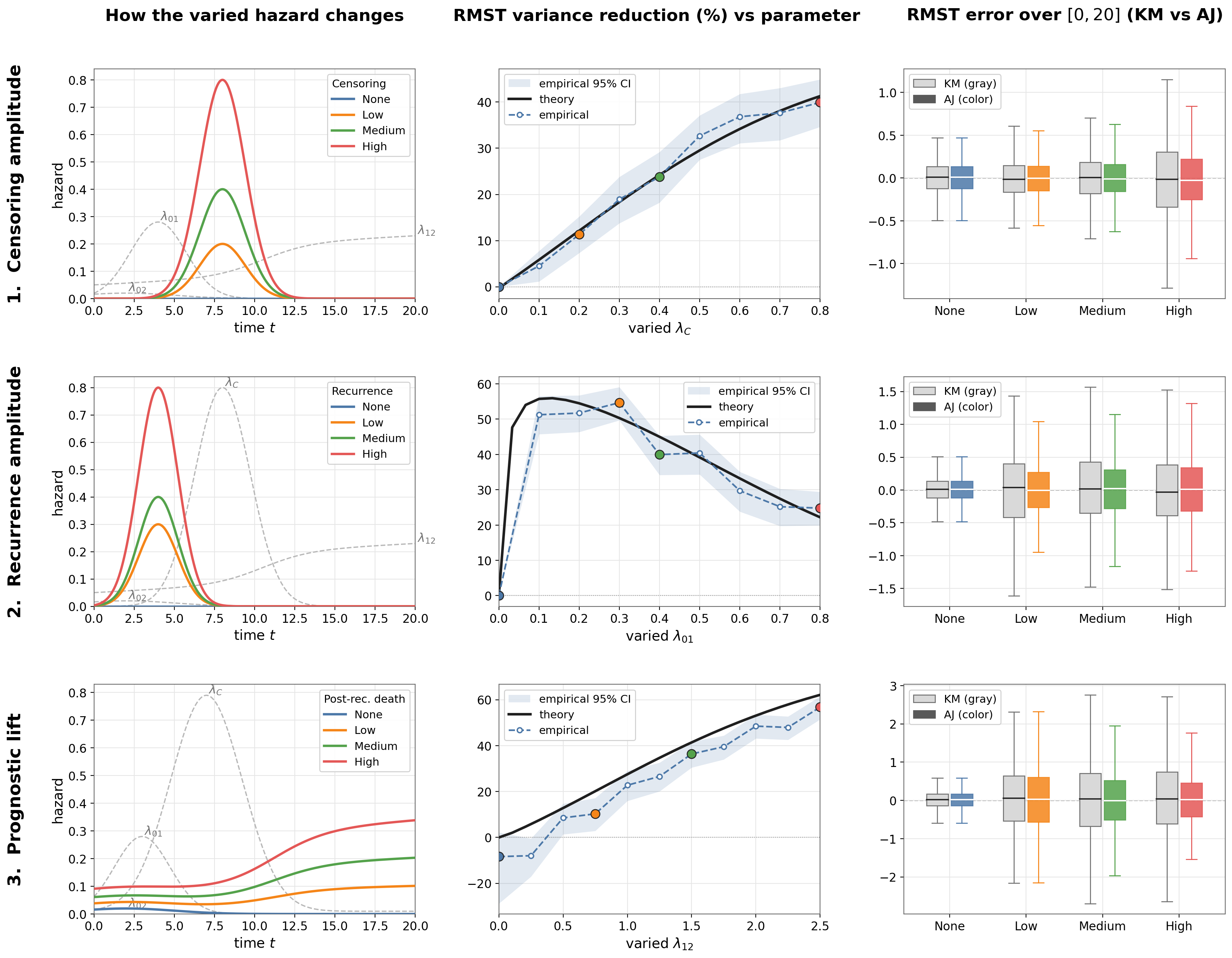}
  \caption{
  \footnotesize   Comparing the variance of two estimators: KM uses only $(\tilde T,\Delta)$, while AJ adds $(\tilde R,\Gamma)$.
    Left col: data designs varying censoring hazard (top), recurrence hazard (mid), post-recurrence hazard (bottom).
    Center col: AJ asymptotic variance reduction over KM across varying data designs.
    Right col: KM vs AJ fixed-$n$ variance across varying data designs.
  }
  \label{fig:gain-snapshot-combined}
\end{figure}%

We illustrate the gain identity \eqref{eq:gain} with synthetic data experiments.
We generate $800$ samples according to a parametrized illness-death law:
\begin{align*}%
  \lambda_C(t)    = A\,e^{-(t-8)^2/4},
  \qquad
  \lambda_{01}(t) = B\,e^{-(t-4)^2/3},
  \qquad
  \lambda_{02}(t) = 0.02\,e^{-(t-2)^2/18},
  \\
  \lambda_{12}(t)  = \lambda_{02}(t) + \kappa\, b(t),
  \qquad
  b(t) = \bigl(0.03 + 0.0005t\bigr) \big\{ 1 +  {2.4} \{1+e^{-(t-11)/1.7}\}^{-1} \big\}
\end{align*}%
on a finite time horizon $\tau=20$, with $n=1000$ subjects, and compute the $\mathcal{F}^\lozenge$-measurable AJ and the $\mathcal{F}^\circ$-measurable KM estimators of the RMST.
We report empirical variances via boxplots in Fig. \ref{fig:gain-snapshot-combined}, third column.
We also compare the empirical AJ-vs-KM variance reduction, expressed as a percentage of the KM variance,
against our theoretical prediction \ref{eq:M-theta-variance}, computed numerically from the input $\lambda$'s.
See Fig. \ref{fig:gain-snapshot-combined}, second column.
Simulation details are provided in \Cref{app:no-X-simulations}.

To illustrate the impact of censoring, recurrence, and post-recurrence prognostic lift on variance reduction,
we vary the corresponding input hazards, see Fig. \ref{fig:gain-snapshot-combined}, first column.

We find:
(i) The gain increases with the strength of censoring (row one, column two), but so does the variance of both the KM and the AJ estimators (row one, column three): the stronger the censoring, the harder the estimation problem becomes.
(ii) As recurrence increases from zero, subjects begin transitioning into state $1$, and the gain rises sharply; as recurrence increases further, we observe a slow decay, since subjects in state $1$ are much more likely to transition onward to state $2$, decreasing the occupancies for both states $0$ and $1$.
(iii) While in theory the efficiency gain is always positive, in practice, if the surrogate is not predictive of the true outcome, the AJ estimator can be noisier than KM, as we observe in the empirical curve of the second column.
This, however, is not the case as the surrogate becomes even a little predictive, and we observe a linear proportional relationship between the prognostic lift and the efficiency gain.

\section{Covariate-adjusted estimation}%
\label{sec:covariates}

The no-covariate result is clean because the transition law is a finite-dimensional-in-state, nonparametric-in-time object that AJ estimates directly.
Continuous covariates change the problem.
Now the analyst must learn how the whole transition system varies with $X$.
This is not ordinary baseline covariate adjustment for a scalar mean; the nuisance is a conditional illness--death law.
With finite well-populated strata, one may fit AJ within each stratum and standardize \citep{aalen78}.
With continuous or multi-dimensional $X$, however, the object is the conditional transition law $A_0(\cdot\mid X)$.

\subsection{Cox transitional hazard estimator}%

The Cox--Breslow setting is used as a benchmark rather than as the general method \citep{cox72,breslow72,andersen12}.
Its role is to illustrate the best-case plug-in regime: when each transition model is correctly specified and finite-dimensional,
the conditional transition law is estimated at a sufficiently fast rate.
Thus ordinary standardization is already first-order valid.
For each transition $\ell=(j,k)\in\cJ$, suppose

\begingroup
\setstretch{1}
\setlength{\jot}{0pt}
\begin{equation}
\label{eq:cox-main-transition-model}
  E\{dN_\ell(u)\mid\mathcal F_{u-},X\}
  =Y_j(u)\exp(\beta_{\ell,0}^\top X)d\Lambda_{0,\ell,0}(u).
\end{equation}
Let $s_{\ell,0}(u)=E\{Y_j(u)\exp(\beta_{\ell,0}^\top X)\}$.
\endgroup

\begin{assumption}
\label{ass:cox-main-primitives}
The training data are i.i.d., Assumption~\ref{ass:eif-gain-primitives}   holds conditionally on $X$.
The model \eqref{eq:cox-main-transition-model} is correctly specified.
In addition: for each $\ell\in\cJ$,
\begin{enumerate}[label=\textnormal{(CX\arabic*)},leftmargin=2em]
  \item\label{cxmain:bounded}  For $K < \infty$,  $\|X\|\le K$ almost surely, and $\|\beta_{\ell,0}\|<R_\ell<\infty$.
  \item The baseline cumulative hazard is supported on the predictable source support and has finite inverse-risk mass:

  \begingroup
  \setstretch{1}
  \setlength{\jot}{0pt}
  \begin{equation} \label{eq:cox-main-source-support}
    \int I\{s_{\ell,0}(u)=0\}d\Lambda_{0,\ell,0}(u)=0,
    \qquad
    \int\frac{I\{s_{\ell,0}(u)>0\}}{s_{\ell,0}(u)}d\Lambda_{0,\ell,0}(u)<\infty.
  \end{equation}
  \endgroup
  \item\label{cxmain:information} The Cox information matrix is nonsingular, i.e., for $\bar X_\ell(\beta,u) = {s_\ell^{(1)}(\beta,u)} / {s_\ell^{(0)}(\beta,u)}.$

  \begingroup
  \setstretch{1}
  \setlength{\jot}{0pt}
  \begin{align*}%
    \lambda_{\min}\left\{\int V_\ell(\beta_{\ell,0},u)s_{\ell,0}(u)d\Lambda_{0,\ell,0}(u)\right\}>0 ,
    \quad
    V_\ell(\beta,u) = \frac{s_\ell^{(2)}(\beta,u)}{s_\ell^{(0)}(\beta,u)} - \bar X_\ell(\beta,u)\bar X_\ell(\beta,u)^\top .
  \end{align*}%
  \endgroup
\end{enumerate}
\end{assumption}

The source-support formulation is important in the illness--death setting.
The post-recurrence risk set is empty early in follow-up, so requiring a uniform lower bound for that risk set would rule out the natural process.
The condition instead asks only that the baseline hazard put mass where the corresponding source state has population risk and that the inverse-risk variation be finite.

Let $\widehat\beta_\ell$ be the partial-likelihood estimator and $\widehat\Lambda_{0,\ell}$ the predictable-support Breslow estimator.
The fitted conditional hazards are $d\widehat A_{\ell,n}(u\mid x)=\exp(\widehat\beta_\ell^\top x)d\widehat\Lambda_{0,\ell}(u)$, $\ell\in\cJ$.
For fixed $x$, the state probabilities are computed by the product-integral recursion associated with the fitted
transition hazards.
At a jump time $u$,

\begingroup
\setstretch{1}
\setlength{\jot}{0pt}
\begin{align*}%
  \widehat P_0(u\mid x)
  &=
  \widehat P_0(u-\mid x)
  \big\{1-\Delta\widehat A_{01,n}(u\mid x)-\Delta\widehat A_{02,n}(u\mid x)\big\},
  \\[1em]
  \widehat P_1(u\mid x)
  &=
  \widehat P_1(u-\mid x)
  \big\{1-\Delta\widehat A_{12,n}(u\mid x)\big\}
  +
  \widehat P_0(u-\mid x)\Delta\widehat A_{01,n}(u\mid x).
\end{align*}%
\endgroup
Let $I_1,I_2$ be a partition of $N=n+m$ i.i.d.\ observations, $n=|I_1|$, and $m=|I_2|$.
Let $\widehat A_n$ be fitted on $I_1$, and let $I_2$ be an independent evaluation sample.
The Cox plug-in RMST is

\begingroup
\setstretch{1}
\setlength{\jot}{0pt}
\begin{align} \label{eq:Cox-plug-in}%
   \widehat\mu_{n,m}^{\Cox,\plug}(\tau)
  =
  m^{-1}\sum_{i=1}^m
  \mu_{\widehat A_n}(\tau\mid X_i^{\rm ev}), \qquad  \mu_{\widehat A_n}(\tau\mid x)
  =
  \int_0^\tau
  \Big\{ \widehat P_0(t\mid x)+\widehat P_1(t\mid x) \Big\}\,dt .
\end{align}%
\endgroup 

\begin{proposition} \label{prop:cox-main-regime}
Under Assumption~\ref{ass:cox-main-primitives},
$ \text{for each } \ell\in\cJ $,
$ \|\widehat\beta_\ell-\beta_{\ell,0}\|=O_p(n^{-1/2}), $\;
$ \sup_{t\le\tau} \big|\widehat\Lambda_{0,\ell}(t)-\Lambda_{0,\ell,0}(t)\big| =O_p(n^{-1/2}), $\;
and\; $ \widehat\mu_{n,m}^{\Cox,\plug}(\tau)-\mu_\tau =O_p(m^{-1/2}+n^{-1/2}). $
\end{proposition}

\subsection{Flexible conditional hazards: why one-step estimator matters}%

The plug-in estimator directly averages the fitted conditional RMST, hence is generally sensitive to first-order errors in the transition hazards.
The one-step estimator adds a correction based on observed-minus-fitted transition increments.
Each increment is weighted by the fitted change in remaining RMST at that transition and adjusted for censoring.
We use the same estimated transition law to construct the fitted RMST, the transition residuals, and their weights.

Estimate the conditional transition law and censoring survival by $\widehat A_n$ and $\widehat G_n$ using a training sample of size $n$.
Let $P_Xg=\int g\,dP_X$, $Pf=\int f\,dP$ and let $\mathbb P_{{\rm ev},m}f=m^{-1}\sum_{i=1}^m f(O_i)$ denote, respectively,
averaging over the population law of $X$, of $O$, and an independent evaluation sample.
Let $\widehat P_{jk,n}(u,t\mid x)$ be the transition probabilities obtained from $\widehat A_n(u\mid x)$ by product integration \eqref{eq:matrix-product-integral}.
Define the fitted residual RMST by

\begingroup
\setstretch{1}
\setlength{\jot}{0pt}
\[
  \widehat V_{j,n}(u,\tau\mid x)
  =\int_u^\tau\sum_{r=0}^1
    \widehat P_{jr,n}(u,t\mid x)\,dt,
  \qquad j=0,1,
\]
\endgroup
and set $\widehat V_{2,n}=0$.
Since all subjects enter in state $0$, write
$\widehat\mu_n(x):=\mu_{\widehat A_n}(\tau\mid x)
=\widehat V_{0,n}(0,\tau\mid x)$.
For $\ell\in\cJ$, define
\[
  \widehat H_{\ell,n}(u,x)
  =
  \frac{\widehat V_{k(\ell),n}(u,\tau\mid x)
        -\widehat V_{j(\ell),n}(u,\tau\mid x)}
       {\widehat G_n(u-\mid x)},
\]
where $\widehat G_n(u-\mid x)>0$ on the estimation interval.
Each weight is the fitted change in remaining RMST at transition
$\ell$, adjusted for censoring.

The one-step score adds weighted transition residuals to the
fitted RMST:

\begingroup
\setstretch{1}
\setlength{\jot}{0pt}
\begin{equation}
\label{eq:main-psihat-def}
  \widehat\psi_n(O)
  =
  \widehat\mu_n(X)
  +\sum_{\ell\in\cJ}\int_0^\tau
    \widehat H_{\ell,n}(u,X)
    \big\{dN_\ell(u)
    -Y_{j(\ell)}(u)\,d\widehat A_{\ell,n}(u\mid X)\big\}.
\end{equation}
\endgroup
The plug-in and one-step estimators are therefore

\begingroup
\setstretch{1}
\setlength{\jot}{0pt}
\begin{align}
\label{eq:os-est-integrated-main}
  \widehat\mu_{n,m}^{\plug}(\tau)
  &=\mathbb P_{{\rm ev},m}\{\widehat\mu_n(X)\},
  \qquad
  \widehat\mu_{n,m}^{\os}(\tau)
  =\mathbb P_{{\rm ev},m}\{\widehat\psi_n(O)\}.
\end{align}
\endgroup

The theorem below first controls the population bias $P\widehat\psi_n-\mu_\tau$, conditional on the training sample,
under bounded transition fits and consistent censoring estimation.
It then gives the plug-in convergence rate and, under stronger nuisance-rate conditions,
the efficient expansion of the one-step estimator with a second-order transition--censoring product remainder.

The actual implementation uses cross-fitting.
Let $I_1,\ldots,I_K$ be a partition of $N$ i.i.d.\ observations, with $K$ fixed, $m_k=|I_k|$, and $N=\sum_{k=1}^K m_k$.
For each fold $k$, fit $\widehat A^{(-k)}$ and $\widehat G^{(-k)}$ on the observations outside $I_k$, and evaluate the fitted plug-in and one-step summands only on observations in $I_k$.
For $i\in I_k$, let $\widehat\psi_{-k}(O_i)$ denote \eqref{eq:main-psihat-def} with $(\widehat A_n,\widehat G_n)$ replaced by $(\widehat A^{(-k)},\widehat G^{(-k)})$.  
The cross-fitted estimators are
\begin{align} \label{eq:crossfit-plugin-os-main}
  \widehat\mu_{N}^{\cf,\plug}(\tau)
  &=
  \frac1N\sum_{k=1}^K\sum_{i\in I_k} \mu_{\widehat A^{(-k)}}(\tau\mid X_i),
  \qquad
  \widehat\mu_{N}^{\cf,\os}(\tau)
  =
  \frac1N\sum_{k=1}^K\sum_{i\in I_k} \widehat\psi_{-k}(O_i).
\end{align}

Let $\mathscr A_\tau$ be the class of admissible finite-variation cumulative-hazard triplets on $[0,\tau]$, with nondecreasing off-diagonal transition coordinates and row increments at most one.
For $A,B\in\mathscr A_\tau$, let $H=B-A$, and let $d\mathbf H$ be the corresponding signed transition-matrix increment.
Define

\begingroup
\setstretch{1}
\setlength{\jot}{0pt}
\begin{align*}
  \kappa_{\rm PI}(A,B)&=\sup_{0\le s\le t\le\tau}\|P_B(s,t)-P_A(s,t)\|_\infty,
  \\
  \omega_{\rm PI}(A,B)&=\sup_{0\le s\le t\le\tau}\left\|P_B(s,t)-P_A(s,t)-\int_{(s,t]}P_A(s,u-)d\mathbf H(u)P_A(u,t)\right\|_\infty.
\end{align*}
\endgroup
With a sample partition $I_1,\ldots,I_K$, fixed $K$, and $m_k=|I_k|$,
assume there are constants $0<\underline\pi\le\overline\pi<1$ such that $\underline\pi\le m_k/N\le\overline\pi$ for all large $N$ and all folds $k$.
For each fold $k$, let $\mathcal D_{-k}$ be the observations outside $I_k$, and let $\widehat A^{(-k)}$ and $\widehat G^{(-k)}$ be $\mathcal D_{-k}$-measurable estimators.
For $M<\infty$, write
$ \mathscr A_\tau(M)=\{A\in\mathscr A_\tau: |A|_\tau\le M\}, $
with $ |A|_\tau=\sum_{\ell\in\cJ}A_\ell(\tau). $
For each fold define

\begingroup
\setstretch{1}
\setlength{\jot}{0pt}
\[
  \kappa_{-k}(x)=\kappa_{\rm PI} \big\{A_0(\cdot\mid x),\,\widehat A^{(-k)}(\cdot\mid x)\big\},
  \qquad
  \omega_{-k}(x)=\omega_{\rm PI} \big\{ A_0(\cdot\mid x),\,\widehat A^{(-k)}(\cdot\mid x) \big\}.
\]
\endgroup

For the sharper local comparison, also write
$
  D_{k,\ell}(u\mid x)
  =
  \widehat A_{\ell}^{(-k)}(u\mid x)-A_{\ell,0}(u\mid x),
$  and $
  \delta_{A,-k}(x)
  =
  \max_{\ell\in\cJ}\sup_{0\le u\le\tau}|D_{k,\ell}(u\mid x)|.
$
Define the signed censoring-survival-ratio error

\begingroup
\setstretch{1}
\setlength{\jot}{0pt}
\begin{align*}
  e_{G,-k}(u,x)
  =
  1-\frac{G(u-\mid x)}{\widehat G^{(-k)}(u-\mid x)},
  \qquad
  \eta_{-k}(x)=\sup_{0\le u\le\tau}|e_{G,-k}(u,x)|.
\end{align*}
\endgroup
For an admissible transition law $A$, put

\begingroup
\setstretch{1}
\setlength{\jot}{0pt}
\[
  W_{\ell,A}(u,x)
  =
  V^A_{k(\ell)}(u,\tau\mid x)-V^A_{j(\ell)}(u,\tau\mid x),
  \qquad
  L_{\ell,0}(u,x)
  =
  P^{A_0}_{j(\ell)}(u-\mid x)W_{\ell,A_0}(u,x).
\]
\endgroup

\begin{assumption}
\label{ass:main-transition-martingale-linearization}
For a deterministic sequence $a_N\downarrow0$, suppose that for every fold $k$, relative to the training-sample calendar-time filtration $\mathbb F_{-k}=\{\mathcal F_{-k}(u):0\le u\le\tau\}$,
\begin{align*}
  D_{k,\ell}(u\mid x)
  =
  U_{k,\ell}(u\mid x)+R_{k,\ell}(u\mid x),
  \qquad \ell\in\cJ,
\end{align*}
where $U_{k,\ell}(\cdot\mid x)$ is a square-integrable $\mathbb F_{-k}$-martingale and the stochastic Fubini interchanges below are valid.
The process $e_{G,-k}(u,x)$ is left-continuous and $\mathbb F_{-k}$-predictable.
For every bounded $\mathbb F_{-k}$-predictable multiplier $z$, write
$
\|z\|_{2,\infty}
=
\left[
P_X\left\{\sup_{0\le u\le\tau}|z(u,X)|^2\right\}
\right]^{1/2},
$
and define
\begin{align*}
  \mathcal U_{k,z}(t)
  &=
  \sum_{\ell\in\cJ}
  P_X\int_0^t
  L_{\ell,0}(u,X)z(u,X)\,dU_{k,\ell}(u\mid X),
\\
  \mathcal R_{k,z}
  &=
  \sum_{\ell\in\cJ}
  P_X\int_0^\tau
  L_{\ell,0}(u,X)z(u,X)\,dR_{k,\ell}(u\mid X).
\end{align*}
There are random variables $C_{k,N}=O_p(1)$, uniformly over the fixed number of folds, such that
$
  \langle\mathcal U_{k,z}\rangle(\tau)
  \le C_{k,N}a_N^2\|z\|_{2,\infty}^2,
$ and $ 
  |\mathcal R_{k,z}|
  \le C_{k,N}a_N\|z\|_{2,\infty}.
$
\end{assumption}

Assumption~\ref{ass:main-transition-martingale-linearization} is the conditional analogue of the weighted Nelson--Aalen linearization.
Taking $z=e_{G,-k}$, using $\|e_{G,-k}\|_{2,\infty}^2=P_X\{\eta_{-k}(X)^2\}$, and applying Lenglart's inequality gives

\begingroup
\setstretch{1}
\setlength{\jot}{0pt}
\begin{align} \label{eq:main-AG-transform-rate}
  \max_{1\le k\le K}
  \left|
  \sum_{\ell\in\cJ}P_X\int_0^\tau
  L_{\ell,0}(u,X)e_{G,-k}(u,X)\,dD_{k,\ell}(u\mid X)
  \right|
  =O_p(a_Nc_N)
\end{align}
\endgroup
whenever
$\max_k P_X\{\eta_{-k}(X)^2\}=O_p(c_N^2)$.

\begin{theorem}  \label{thm:plugin-onestep-integrated-summary}
Assume conditionally independent censoring.
Suppose that, for some constants $g_0>0$ and $M<\infty$, the event
\begin{align*}%
  \mathcal E_N=
  \bigcap_{k=1}^K
  \left\{
  \begin{array}{l}
    A_0(\cdot\mid x) \text{ and } \widehat A^{(-k)}(\cdot\mid x) \in\mathscr A_\tau(M)\quad \text{for }P_X\text{-a.e. }x,
    \\[2mm]
    G(\cdot\mid x) \text{ and } \widehat G^{(-k)}(\cdot\mid x) \text{ are survival functions},
    \\[1mm]
    \displaystyle\inf_{0\le u\le\tau,\, x}G(u-\mid x)\ge g_0,
    \qquad
    \inf_{0\le u\le\tau,\, x}\widehat G^{(-k)}(u-\mid x)\ge g_0/2
  \end{array}
\right\}
\end{align*}%
satisfies $P(\mathcal E_N)\to1$, where the infima over $x$ are understood $P_X$-essentially.
Suppose that, for a deterministic sequence $c_N\downarrow0$,

\begingroup
\setstretch{1}
\setlength{\jot}{0pt}
\begin{align} \label{eq:main-censoring-rate}
  \max_{1\le k\le K}P_X\{\eta_{-k}(X)^2\}=O_p(c_N^2).
\end{align}
\endgroup
Let $\widehat\mu_N^{\cf,\plug}$ and $\widehat\mu_N^{\cf,\os}$ be the
cross-fitted estimators in \eqref{eq:crossfit-plugin-os-main}, with hazard-induced fold-specific one-step weights.
Then the following statements hold.

\begin{enumerate}[label=\textnormal{(\alph*)},leftmargin=2em]
\item \textit{Global one-step robustness.}
The hazard-induced cross-fitted one-step estimator satisfies

\begingroup
\setstretch{1}
\setlength{\jot}{0pt}
\begin{equation} \label{eq:main-os-rate-with-Ghat}
  \widehat\mu_N^{\cf,\os}(\tau)-\mu_\tau
  =
  O_p(N^{-1/2}+c_N).
\end{equation}
\endgroup
This conclusion uses no convergence of the transition learner.

\item \textit{Plug-in rate.}
Suppose, in addition, that for deterministic sequences $a_N,b_N\downarrow0$,

\begingroup
\setstretch{1}
\setlength{\jot}{0pt}
\begin{align} \label{eq:main-plugin-precise-rates}
  \max_{1\le k\le K}P_X\{\delta_{A,-k}(X)^2\}=O_p(a_N^2),
  \qquad
  \max_{1\le k\le K}P_X\{\omega_{-k}(X)\}=O_p(b_N).
\end{align}
\endgroup
Then 
the cross-fitted plug-in estimator satisfies
\begingroup
\setstretch{1}
\setlength{\jot}{0pt}
\begin{align}
  \label{eq:main-plugin-rate}
  \widehat\mu_N^{\cf,\plug}(\tau)-\mu_\tau
  =
  O_p(N^{-1/2}+a_N+b_N).
\end{align}
\endgroup

\item \textit{Local second-order one-step rate.}
Suppose \eqref{eq:main-plugin-precise-rates} and Assumption~\ref{ass:main-transition-martingale-linearization} hold.
Let

\begingroup
\setstretch{1}
\setlength{\jot}{0pt}
\begin{align*}%
  \psi_0(O)=
  \mu_{A_0}(\tau\mid X)
  +
  \sum_{\ell\in\cJ}\int_0^\tau H_{\ell,0}(u,X)dM_{\ell,0}(u),
  \qquad
  \phi^\lozenge(O)=\psi_0(O)-\mu_\tau.
\end{align*}%
\endgroup
Then

\begingroup
\setstretch{1}
\setlength{\jot}{0pt}
\begin{align} \label{eq:main-os-canonical-expansion}
  \widehat\mu_N^{\cf,\os}(\tau)-\mu_\tau
  =
  (\mathbb P_N-P)\phi^\lozenge
  +O_p(a_Nc_N)
  +O_p\{N^{-1/2}(a_N+c_N)\},
\end{align}
\endgroup
and with it

\begingroup
\setstretch{1}
\setlength{\jot}{0pt}
\begin{equation} \label{eq:main-os-canonical-rate}
  \widehat\mu_N^{\cf,\os}(\tau)-\mu_\tau
  =
  O_p(N^{-1/2}+a_Nc_N).
\end{equation}
\endgroup
\end{enumerate}
\end{theorem}

Part (a) is the global robustness statement: bounded admissibility and coherence give a valid one-step mean at rate $N^{-1/2}+c_N$, even when the transition learner is inconsistent.  
Parts (b) and (c) are local statements.
The plug-in retains the first-order transition error $a_N$, whereas the coherent one-step correction removes that term and leaves the mixed transition--censoring remainder $a_Nc_N$.    If the nuisance product-rate condition
$
  a_Nc_N=o(N^{-1/2})
$
holds,
then the canonical expansion
$
\widehat\mu_N^{\cf,\os}(\tau)-\mu_\tau
=
(\mathbb P_N-P)\phi^\lozenge+o_p(N^{-1/2})
$
holds and the one-step reaches the semiparametric efficiency bound.

In particular, if $a_N\asymp c_N\asymp r_N$ and
$b_N=O(r_N^2)$, then

\begingroup
\setstretch{1}
\setlength{\jot}{0pt}
\[
\widehat\mu_N^{\cf,\plug}(\tau)-\mu_\tau
=O_p(N^{-1/2}+r_N),
\qquad
\widehat\mu_N^{\cf,\os}(\tau)-\mu_\tau
=O_p(N^{-1/2}+r_N^2).
\]
\endgroup
For the standard integrated nonparametric rate $r_N=N^{-\beta/(2\beta+d)}$, the one-step remainder is $N^{-2\beta/(2\beta+d)}$, and canonical root-$N$ inference follows when $\beta>d/2$.
More generally, if the transition and censoring learners have rates $N^{-\alpha_A}$ and $N^{-\alpha_G}$, respectively, the product-rate condition is $\alpha_A+\alpha_G>1/2$.

\section{Numerical Work}%

We contrast finite-$N$ performance of eight RMST estimators, obtained by crossing three binary design choices:
First, either death-only data or the death-and-recurrence data \eqref{eq:observed-sa} are used;
marked \textsf{KM} and \textsf{AJ}.
Second, either transition hazards are unadjusted for $X$, as in the no-$X$-AJ \eqref{eq:AJ},
or fit conditional on $X$ with Cox \eqref{eq:cox-main-transition-model} or a flexible learner;
marked ``\textsf{Cov}.''
Third, either the fitted transition law is plugged-in \eqref{eq:Cox-plug-in},
or corrected by one-step augmentation \eqref{eq:os-est-integrated-main};
marked ``\textsf{pi}'' or ``\textsf{os}.''
Crossing recurrence (\textsf{KM/AJ}), covariates (\textsf{$\emptyset$/Cov}), and debiasing (\textsf{pi/os}) yields

\begingroup
\setstretch{1}
\setlength{\jot}{0pt}
\begin{align} \label{eq:eight-estimators}%
  \big\{
  \widehat\mu^{\KM},\quad
  \widehat\mu^{\AJ},\quad
  \widehat\mu^{\KMCov,\plug},\quad
  \widehat\mu^{\AJCov,\plug},\quad
  \widehat\mu^{\KM,\os},\quad
  \widehat\mu^{\AJ,\os},\quad
  \widehat\mu^{\KMCov,\os},\quad
  \widehat\mu^{\AJCov,\os}
  \big\}.
\end{align}%
\endgroup
Each matched pair differs in exactly one feature and isolates the value of recurrence, covariates, or the one-step correction.
\Cref{app:simulations}
provides complete experimental details.

\subsection{Synthetic simulation studies} \label{sec:simulations}%

\begin{figure}[t]%
  \centering
  \includegraphics[width=\linewidth]{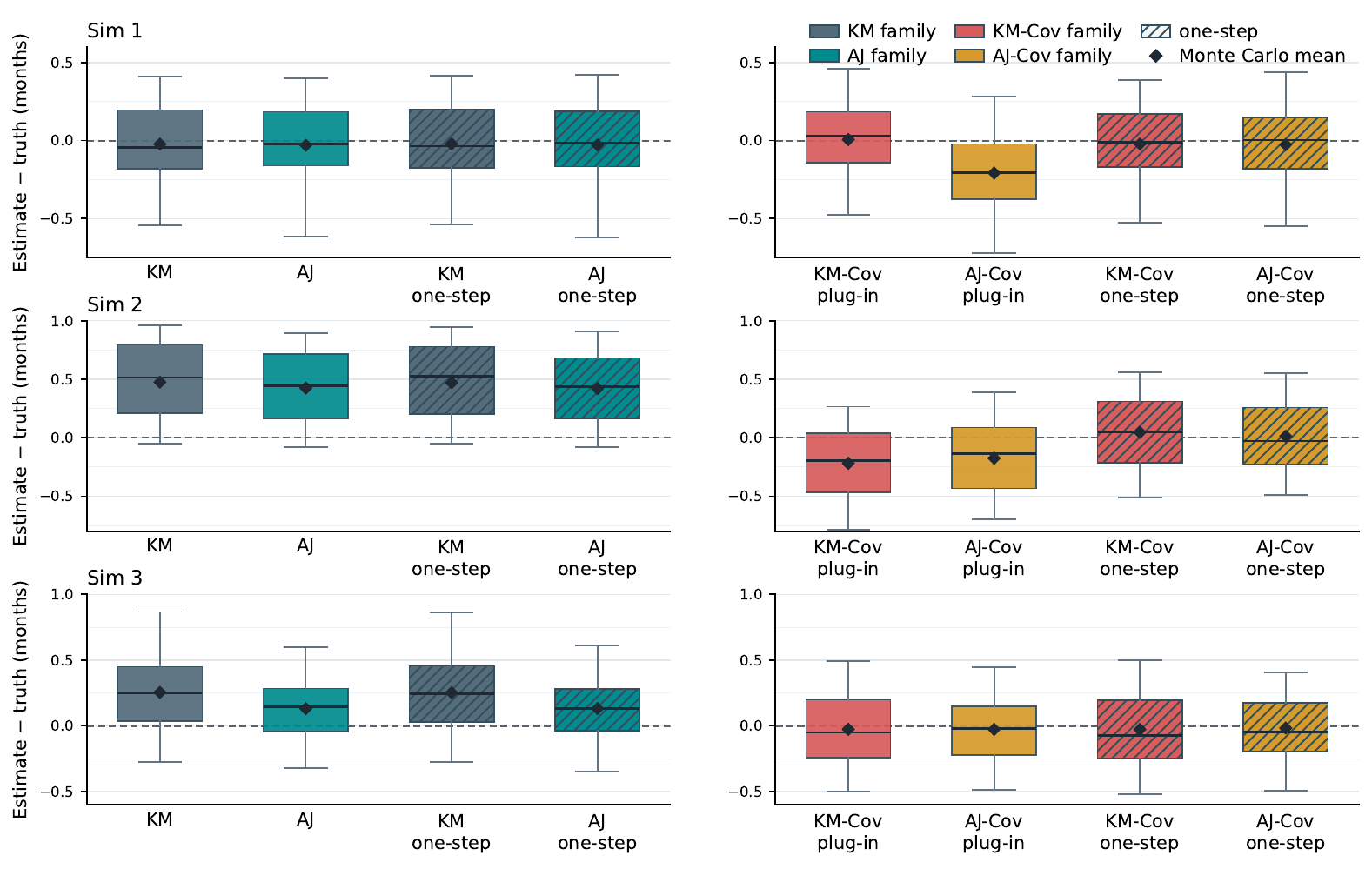}
  \caption{Fixed-$N$ estimator distributions; eight RMST estimators; three experiments.
  }
  \label{fig:sim-boxplots}
\bigskip
\end{figure}%

Three simulation designs vary, in turn, misspecification in the post-recurrence hazard,
the flexibility of the nuisance learner, and the amount and timing of recurrence information.
Figure~\ref{fig:sim-boxplots}
reports the $N=800$ replicate error distributions for the strongest cell of each design:
each box is one scenario, sample-size, and estimator cell,
hatching distinguishes one-step from plug-in estimators, and color distinguishes the KM and AJ families.

We find:
(i) $\widehat\mu^{\AJCov,\plug}$ is inconsistent under misspecified $\widehat A_{12}$, but $\widehat\mu^{\AJCov,\os}$ restores consistency, top row.
(ii) $\widehat\mu^{\AJCov,\os}$ removes finite sample bias of $\widehat\mu^{\AJCov,\plug}$ in truly nonlinear design with nuisances learned by gradient-boosted Poisson-loss regression, middle row.
(iii) the efficiency gain from recurrence is lowered by covariate adjustment and when censoring arrives ahead of recurrence,
bottom row.

\Cref{app:sim-covariate-scenarios} details the three mechanisms and their scaling on $N$, \Cref{fig:sim-contrasts}.
Read together, the three designs separate the two questions that the population gain identity of \Cref{subsec:gain-snapshot} conflates:
Simulations 1 and 2 show that the one-step correction protects the target under a wrong or merely flexible nuisance model;
Simulation 3 shows that, even with a correctly specified model, the AJ-versus-KM gain tracks how strongly recurrence has split the risk set and separated the state-specific
residual-survival means, eroding its finite-sample benefit once censoring hides recurrence before it can be observed.

\subsection{Application to semi-synthetic clinical trials data}\label{sec:application}%

We use metastatic breast cancer (MBC) clinical-trial curves to build semi-synthetic experiments to contrast the performance of eight estimators \eqref{eq:eight-estimators}.
The data come from a repository of 1,865 drug-therapy trials assembled by \citet{silberholz2019}.
Our setting requires a two-arm randomized trial with armwise curves for the intermediate (progression/PFS) and terminal (death/OS) events.
We conducted a quality-control audit to verify that the repository representation of each trial was sufficiently consistent with the corresponding source publication for our analysis.
This yields 40 comparisons we can simulate from.
Table~\ref{tab:rct_summary} in the Appendix summarizes key features of the patients and their outcomes.

From each trial, we calibrate a latent arm-specific event-time law to the published censoring-adjusted survival curves.
Because the public files do not contain the underlying patient-level censoring process, censoring is generated from a separately specified design.
The resulting application is semi-synthetic: published trial curves determine the event-time marginals, and an explicit generative model determines the joint illness--death path, patient covariates, and censoring mechanism.

\subsubsection{Semi-synthetic data-generation}%

We run two sets of experiments.
In both we simulate the illness--death process \eqref{eq:S-state} with transitions $\cJ$.
The estimand throughout is the arm-wise RMST with $\tau=24$ months, and the treatment effect is $\mu^1-\mu^0$.
Full details are given in Appendix~\ref{app:datagen}.

We first fix the marginals at prespecified smooth Weibull PFS and OS curves.
The distributions for these curves are chosen to resemble the MBC trials in scale but not fitted specifically to any of them. 
We calibrate a single illness--death process such that its induced marginal curves match these two Weibull distributions.
Holding both marginals smooth and fixed allows us to isolate the effect of the coupling mechanism alone.

We then apply the coupling with the marginals from an actual trial from the processed digitized PFS and OS curves.
For clarity of exposition we present one comparison from \cite{conte1996chemotherapy} in detail.
We select this example because it displays the surrogate paradox: at $\tau = 24$ months, treatment improves PFS by 1.322 months on average, while worsening OS by 0.749 months.
Later, in Section~\ref{sec:allTrials}, we further examine all 40 comparisons.
We take only the arm-specific digitized survival values as trial inputs.

The pair of curves alone cannot determine the process that generated it. 
PFS counts both recurrence and death without recurrence without distinction. 
We therefore construct a family of joint laws for recurrence and death, all of which reproduce the PFS and OS marginals, and present the results across that family. 
Let $\pi_j$ denote the share of recurrence-free exits in interval $j$ allocated to recurrence; choosing $\pi_j$ fixes the remaining transition flows, so the compatible laws form a one-dimensional family in each interval. 
We evaluate five couplings sweeping the feasible range, \emph{Lower}, \emph{Quarter}, \emph{Midpoint}, \emph{Three-quarter}, and \emph{Upper}. 
In the trial-anchored experiment, we evaluate two additional couplings \emph{Smooth} and \emph{Constant projection}.
The couplings differ in how many patients occupy the post-recurrence state, when they enter it, and how quickly they then die.

Each simulated patient carries a baseline covariate $X\sim\mathcal N(0,1)$, drawn independently of the treatment.
It enters all three transition hazards through nonlinear, time-varying log-risk terms,
with per-interval intercepts calibrated so that averaging over $X$ still reproduces the same PFS and OS marginals.

Censoring is imposed after the latent paths are drawn.
We draw censoring times independently of recurrence, death, and the baseline covariate given the treatment arm.
Each arm follows a single two-point law, censoring at a fixed landmark with a probability calibrated so that the pooled-arm expected chance of censoring before death is $50\%$.

\subsubsection{Results and Main takeaways}%

\begin{figure}[t]
  \centering
  \includegraphics[width=\linewidth]{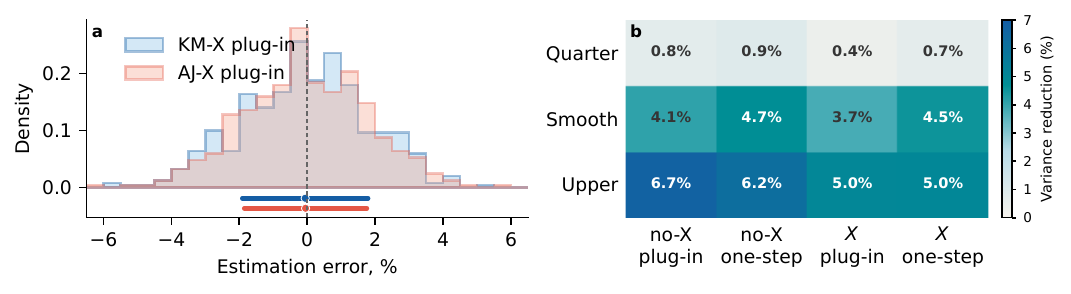}
  \caption{
    CONTE0796-anchored experiment.
    (a): distributions of no-$X$ KM and AJ plug-in estimators under \emph{Upper}.
    (b): variance reduction for estimator-by-coupling combinations.
  }
  \label{fig:main-mc}
\end{figure}

Figure~\ref{fig:main-mc} shows fixed-$N$ comparisons for the trial-anchored couplings.
Panel~(a) enlarges a single AJ-vs-KM pair.
The two error distributions are centered but differ in spread, a $6.7\%$ variance reduction.
Panel~(b) reports the reduction in coupling-by-estimator combinations.
All twelve are positive.
The estimators preserve the same ordering in their gains: under \emph{Upper}, the no-$X$ estimators improve by $6.7\%$ and $6.2\%$, compared with $5.0\%$ for their covariate-adjusted counterparts.
The smaller adjusted gain is consistent with the baseline covariate already capturing part of the relevant prognostic information.
Collectively, these results show that observing and incorporating recurrence increases efficiency, but the gain depends on the joint law connecting recurrence and death.
This cannot simply be determined from effect sizes and PFS and OS curves.
Importantly, the estimator exploits recurrence without requiring the joint law to be known in advance: in the worst case the benefit is small, and in the best case it is substantial.

These gains translate directly into trial size.
A variance reduction of $\mathbb{V}$ means a Kaplan--Meier analysis requires $1/(1-\mathbb{V})$ times as many patients to reach the same precision.
Under the most informative construction, where $\mathbb{V}\approx5\%$, a trial that would otherwise enroll $600$ patients reaches the same precision with roughly $570$.
At \$50{,}000 per participant \citep{moore2018estimated}, that is a saving of approximately \$1.5 million.
The saving comes from information already collected in the course of the trial.

One might naturally ask whether individual-level and trial-level correlation used to validate surrogate endpoints
\citep{prentice1989surrogate, buyse2000validation} would tell us in advance the efficiency gain with recurrence.
However, they cannot, because the gain also depends on when recurrence occurs relative to when follow-up ends.
These measures were developed to assess whether a surrogate can substitute for the true endpoint.
Our question is different: overall survival remains the estimand and recurrence enters only as auxiliary information.

\subsubsection{Impact Across Trials}\label{sec:allTrials}%

\begin{figure}[t]%
  \centering
  \begin{minipage}[c]{0.49\textwidth}
    \centering
    \includegraphics[angle=270,width=\linewidth,keepaspectratio]{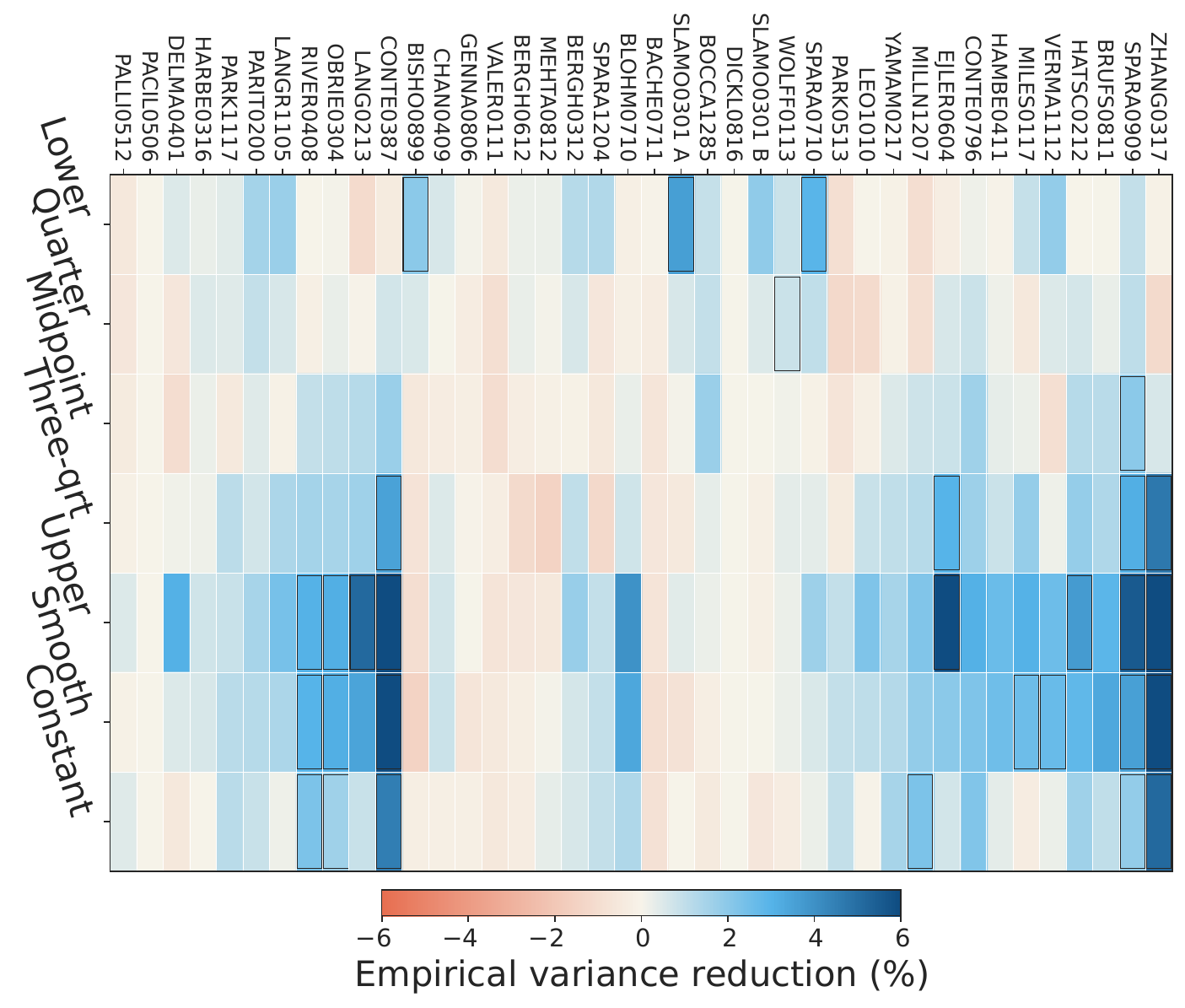}

  \end{minipage}%
  \hfill
  \begin{minipage}[c]{0.49\textwidth}
    \centering
    \includegraphics[width=\linewidth,]{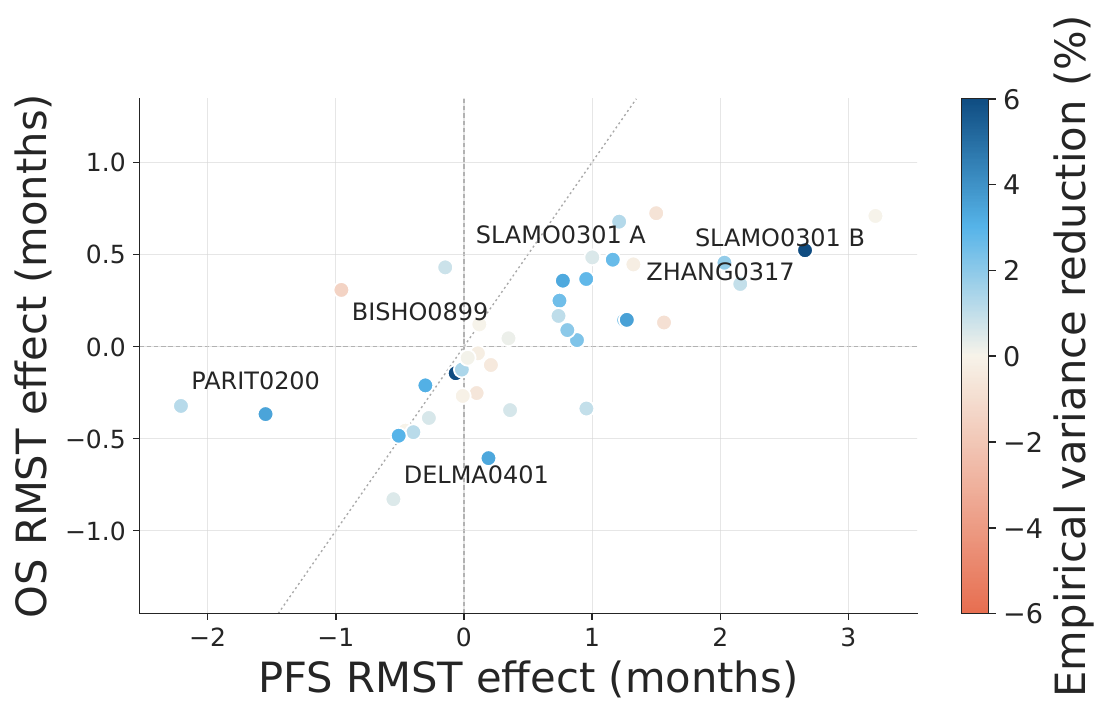}
  \end{minipage}
  \caption{
    (a):
      variance reduction with recurrence (\%),
      for 40 trials (columns) under each seven transition couplings (rows).
    (b):
      trial's own 12-month PFS (x-axis) and OS (y-axis) treatment effects (months),
      and variance reduction with recurrence (\%, color),
      for 40 trials under \emph{Smooth} coupling.
  }
  \label{fig:all-40-trials}
  \bigskip
\end{figure}%

To assess whether the CONTE0796 pattern generalizes, we repeat the same construction across all 40 eligible comparisons,
holding the estimators and the seven transition couplings fixed and varying only the trial.
Figure~\ref{fig:all-40-trials} 
reports the resulting empirical variance gain for every trial-by-coupling combination, along with the target-vs-surrogate treatment effects comparison for each trial.

\section{Discussion}%
Survival data are important targets for causal inference and pose challenges for identification and efficiency.
Surrogate outcomes carry no first-order efficiency value with complete data,
but recover information lost to censoring and relax the ignorability assumption by refining the predictable residual-survival projection.
We derive the information efficiency bound and the exact gap between the marginal KM and the augmented AJ scores.
In the no-covariate setting, this leads to the natural and efficient standardized AJ estimator.
In the setting with covariates, we show that a cross-fitted hazard-induced one-step estimator achieves rate-robustness and canonical root-$N$ inference.
Interesting follow-up questions are how to select surrogate events toward optimal efficiency and how to employ surrogate events in the design of experiments.

\clearpage
\appendix

\section*{Technical appendices}
\addcontentsline{toc}{section}{Technical appendices}

\tableofcontents

\section{Notation and relation to the main text}\label{app:notation2}

We use the data, coarsening map, and estimands of Section~\ref{sec:problem-setup},
with coarsening map \eqref{eq:C-coarsening-map} and estimands \eqref{eq:parameters}.
The illness--death states, transition indices, risk sets, and event processes are those of Section~\ref{sec:efficiency};
the event-history and observed filtrations are given in \eqref{eq:F-filtrations} and~\eqref{eq:observed-sa}.
The projection martingales and their contrast are defined in \eqref{eq:IF-contrast} and~\eqref{eq:IF-contrast2}.

For no-covariate estimation, we use the product integral, AJ estimator, and
RMST influence function in \eqref{eq:matrix-product-integral}, \eqref{eq:AJ},
and \eqref{eq:main-nocov-rmst-eif}--\eqref{eq:main-nocov-rmst-H}.
The conditional-hazard, sample-splitting, and cross-fitting notation is that
of Section~\ref{sec:covariates}, with estimators given in
\eqref{eq:os-est-integrated-main} and~\eqref{eq:crossfit-plugin-os-main}.
The moduli $\kappa_{\rm PI},\omega_{\rm PI}$ and the foldwise errors
$D_{k,\ell},\delta_{A,-k},e_{G,-k},\eta_{-k}$ retain their definitions there.
Additional norm conventions are introduced in
Section~\ref{sec:setup-notation}; single-split sampling and empirical-mean
notation are given at the beginning of Section~\ref{app:covariate-proofs}.

The proofs below use the assumptions of the corresponding main-text
results. Section~\ref{sec:cox-breslow-details} distinguishes the raw fitted Cox hazards
from their row-stabilized version.

The order of the proofs follows Sections~\ref{sec:efficiency}
and~\ref{sec:covariates}. Supplementary statements are placed with their
proofs, and the numerical details are collected in Section~\ref{app:simulations}.

\subsection{Variables and potential outcomes}

\begin{itemize}[leftmargin=0em,itemsep=-0.1em]
  \item $A \in \{0, 1\}$: Treatment indicator
  \item $X \in \mathbb{R}^d$: Baseline covariate vector with support $\mathcal{X}$.
  \item $\tau_0 > 0$: Event support bound ($T^a, R^a, C^a \in [0, \tau_0]$).
  \item $\tau \in (0, \tau_0)$: Evaluation horizon for $\theta$ and $\mu$.
  \item $(T^a, R^a, C^a)$: Potential death, recurrence/progression, censoring time.
  \item $(T, R, C)$: Factual event times under consistency: $(T, R, C) = (T^A, R^A, C^A)$.
  \item $\tilde T = T \wedge C$: Observed event time.
  \item $\Delta = I(T \le C)$: Death indicator.
  \item $\tilde R = R \wedge \tilde T = R \wedge T \wedge C$: Observed intermediate (recurrence/progression) event time.
  \item $\Gamma = I(R \le \tilde T) = I(R \le T \wedge C)$: Intermediate event indicator.
  \item $O = (\tilde T, \Delta, \tilde R, \Gamma, A, X)$: Full observed-data tuple for an individual unit.
  \item $P, P_X$: Population law, population covariate law.
  \item $\mathbb{P}_n, \mathbb{P}_N$: Empirical measure over a sample of size $n$ or $N$, i.e., $\mathbb{P}_n f = n^{-1}\sum_{i=1}^n f(O_i)$.
\end{itemize}
\subsection{Target estimands and marginal distributions}

\begin{itemize}[leftmargin=0em,itemsep=-0.1em]
  \item $\theta^a_\tau = P(T^a > \tau) = S^a(\tau)$: Arm-specific marginal survival probability at horizon $\tau$.
  \item $\mu^a_\tau = E(T^a \wedge \tau) = \int_0^\tau P(T^a > t)\,dt$: Restricted mean survival time (RMST) up to horizon $\tau$.
  \item $S(t) \equiv S_T(t) = P(T > t)$: Marginal survival function of the terminal event $T$.
  \item $G(t) \equiv S_C(t) = P(C > t)$: Marginal survival function of the censoring time $C$.
  \item $G(t \mid X) = P(C > t \mid X)$ Conditional survival function of $C$.
  \item $F_U(t) = 1 - S_U(t)$: Cumulative distribution function for a generic random variable $U$.
  \item $\lambda_T(t), \lambda_C(t), \lambda_R(t)$: Marginal hazard rates for death, censoring, and recurrence.
  \item $v(t) = P(T > \tau \mid T > t)$: Residual-survival probability to $\tau$ given survival past $t$.
\end{itemize}
\subsection{Illness--death process and risk sets}

\begin{itemize}[leftmargin=0em,itemsep=-0.1em]
  \item $\mathcal{S}(t) \in \{0, 1, 2\}$: Illness--death state process at calendar/study time $t$:
\[
  \mathcal{S}(t) = 0 \cdot I(t < R \wedge T) + 1 \cdot I(R \le t < T) + 2 \cdot I(T \le t), \qquad \mathcal{S}(0) = 0,
\]
where $0$ is \textit{alive and recurrence-free}, $1$ is \textit{alive post-recurrence}, and $2$ is \textit{dead} (absorbing).
\item $\cJ = \{01, 02, 12\}$: The set of admissible state transitions ($0\to1$, $0\to2$, $1\to2$).
\item $\ell = (j, k) \in \cJ$: Generic transition with source state $j = j(\ell)$ and destination state $k = k(\ell)$.
\item $\lambda_{jk}(t)$: Instantaneous transition hazard from state $j$ to state $k$ at time $t$.
\item $\mathbf{A}(t)$: Cumulative transition hazard matrix, $A_{jk}(t) = \int_0^t \lambda_{jk}(u)\,du$ and $A_{jj}(t) = -\sum_{k \ne j} A_{jk}(t)$.
\item $P_{jk}(s, t) = P(\mathcal{S}(t) = k \mid \mathcal{S}(s) = j)$: Transition probability.
\item $p_j = P(\mathcal{S}(0) = j)$: Initial state probability distribution, $p_0 = 1, p_1 = p_2 = 0$.
\item $P_0(t) = p_0 P_{00}(0, t)$: State-$0$ population occupancy probability at time $t$.
\item $P_1(t) = p_0 P_{01}(0, t) + p_1 P_{11}(0, t)$: State-$1$ population occupancy probability at time $t$.
\item $v_j(t) = P(T > \tau \mid \mathcal{S}(t) = j)$: State-specific residual survival probability to $\tau$.
  \item $v_j(u, t) = P_{j0}(u, t) + P_{j1}(u, t)$: State-specific residual survival probability.
  \item $V_j(u, \tau) = \int_u^\tau v_j(u, s)\,ds$: State-specific residual RMST over $[u, \tau]$.
\end{itemize}

\subsection{Counting processes, risk processes, and filtrations}

\begin{itemize}[leftmargin=0em,itemsep=-0.1em]
  \item Full-data at-risk processes:
  \[
Y_0^*(t) = I(t \le R, t \le T), \quad Y_1^*(t) = I(R < t \le T), \quad Y_T^*(t) = Y_0^*(t) + Y_1^*(t) = I(t \le T).
\]
  \item Full-data transition counting processes:
  \[
N_{01}^*(t) = I(R \le t, R < T), \quad N_{02}^*(t) = I(T \le t, T < R), \quad N_{12}^*(t) = I(R < T \le t).
\]
  \item Observed at-risk and counting processes, stopped at $\tilde T = T \wedge C$:
  \[
Y_j(t) = Y_j^*(t) I(C \ge t), \;\; Y(t) = Y_T^*(t) I(C \ge t), \;\; N_{jk}(t) = N_{jk}^*(t \wedge \tilde T), \;\; N_T(t) = N_T^*(t \wedge \tilde T).
\]
  \item Full data and observed event filtrations:
  \begin{itemize}[label=,]
    \item $\mathcal{F}^\bullet(t) = \sigma\{N_T^*(s) : s \le t\}$: Marginal death filtration.
    \item $\mathcal{F}^\blacklozenge(t) = \sigma\{N_T^*(s), N_R^*(s) : s \le t\}$: Augmented illness--death event filtration.
    \item $\mathcal{F}^\circ(t) = \mathcal{F}^\bullet(t \wedge \tilde T)$: Observed marginal filtration.
    \item $\mathcal{F}^\lozenge(t) = \mathcal{F}^\blacklozenge(t \wedge \tilde T)$: Observed recurrence-augmented filtration.
  \end{itemize}
  \item Full data and observed transition data martingales:
  \begin{itemize}[label=,]
    \item $M_{jk}^\blacklozenge(t) = N_{jk}^*(t) - \int_0^t Y_j^*(s)\lambda_{jk}(s)\,ds$: Full-data transition martingale.
    \item $M_T^\bullet(t) = N_T^*(t) - \int_0^t Y_T^*(s)\lambda_T(s)\,ds$: Full-data marginal death martingale.
    \item $M_\ell(t) = N_\ell(t) - \int_0^t Y_{j(\ell)}(s)\,dA_\ell(s)$: Observed transition martingale for $\ell \in \cJ$.
    \item $M_C(t) = I(C \le t, C \le T) - \int_0^t I(\tilde T \ge s)\lambda_C(s)\,ds$: Censoring martingale.
  \end{itemize}
  \item Estimand information projection martingales and contrast:
    \begin{align*}
      M_\theta^\blacklozenge(t) &= E\{I(T > \tau) - \theta \mid \mathcal{F}^\blacklozenge(t)\}, \qquad M_\theta^\bullet(t) = E\{I(T > \tau) - \theta \mid \mathcal{F}^\bullet(t)\},
      \\
      \Delta_\theta(t) &= M_\theta^\blacklozenge(t) - M_\theta^\bullet(t) = \{v_0(t) - v(t)\}Y_0^*(t+) + \{v_1(t) - v(t)\}Y_1^*(t+).
    \end{align*}
\end{itemize}
\subsection{Influence functions and estimators}

\begin{itemize}[leftmargin=0em,itemsep=-0.1em]
  \item $\phi^*(T) = I(T > \tau) - \theta$: Full-data IF for $\theta = P(T > \tau)$.
  \item $\phi^\Delta(\tilde T, \Delta)$: Missing-outcome IF.
  \item $\phi^\circ(\tilde T, \Delta)$: Death observed-data EIF.
  \item $\phi^\lozenge(\tilde T, \Delta, \tilde R, \Gamma)$: Death and recurrence observed-data EIF.
  \item $\phi_\mu^\lozenge(O) = \sum_{\ell \in \cJ}\int_0^\tau H_\ell^\lozenge(u)\,dM_\ell(u)$: EIF for RMST in the no-covariate model, where:
    \[
      H_{01}^\lozenge(u) = \frac{V_1(u, \tau) - V_0(u, \tau)}{G(u-)}, \quad H_{02}^\lozenge(u) = -\frac{V_0(u, \tau)}{G(u-)}, \quad H_{12}^\lozenge(u) = -\frac{V_1(u, \tau)}{G(u-)}.
    \]
  \item $\widehat A_{jk}^{\AJ}(t) = \int_0^t \frac{I\{\bar Y_j(u) > 0\}}{\bar Y_j(u)}\,d\bar N_{jk}(u)$: Nelson--Aalen cumulative hazard estimator.
  \item $\widehat\theta(t), \widehat\mu^{\AJ}_\tau$: Aalen--Johansen product-integral survival and RMST estimators.
  \item $\widehat\beta_\ell, \widehat\Lambda_{0,\ell}$: Cox coefficient and Breslow baseline hazard estimator.
  \item $\widehat\mu_{n,m}^{\Cox,\plug}(\tau)$: Cox--Breslow plug-in RMST estimator.
  \item $\widehat\mu_{n,m}^{\plug}(\tau), \widehat\mu_N^{\rm cf,plug}(\tau)$: Generic conditional hazard plug-in and cross-fitted plug-in estimators.
  \item $\widehat\mu_{n,m}^{\os}(\tau), \widehat\mu_N^{\rm cf,os}(\tau)$: Hazard-induced one-step and cross-fitted one-step estimators.
  \item $\widehat H_{\ell,n}(u, X) = \frac{\widehat V_{k(\ell),n}(u, \tau \mid X) - \widehat V_{j(\ell),n}(u, \tau \mid X)}{\widehat G_n(u- \mid X)}$: Hazard-induced one-step residual weights.
\end{itemize}

\subsection{Cross-fitting, Error metrics, and convergence rates}

\begin{itemize}[leftmargin=0em,itemsep=-0.1em]
  \item $I_1, \dots, I_K$: Fixed $K$-fold partition of the $N$-sample indices, $m_k = |I_k|$, $N = \sum_{k=1}^K m_k$.
  \item $\mathcal{D}_{-k}$: Training sample omitting fold $k$.
  \item $\mathscr{A}_\tau$: Class of admissible finite-variation cumulative hazard triplets on $[0, \tau]$;
  \item $\mathscr{A}_\tau(M)$ bounds total variation by $M$.
  \item $\kappa_{\rm PI}(A, B) = \sup_{0 \le s \le t \le \tau}\|P_B(s, t) - P_A(s, t)\|_\infty$: Uniform product-integral discrepancy metric.
  \item $\omega_{\rm PI}(A, B)$: Second-order Duhamel linearization remainder of the product integral.
  \item $D_{k,\ell}(u \mid x) = \widehat A_\ell^{(-k)}(u \mid x) - A_{\ell,0}(u \mid x)$: Conditional cumulative hazard error in fold $k$.
  \item $\delta_{A,-k}(x) = \max_{\ell \in \cJ} \sup_{0 \le u \le \tau}|D_{k,\ell}(u \mid x)|$: Uniform hazard estimation error.
  \item $e_{G,-k}(u, x) = 1 - \frac{G(u- \mid x)}{\widehat G^{(-k)}(u- \mid x)}$: Relative censoring estimation error.
  \item $\eta_{-k}(x) = \sup_{0 \le u \le \tau}|e_{G,-k}(u, x)|$: Uniform relative censoring estimation error.
  \item $a_N, b_N, c_N$: Deterministic convergence rate sequences such that
    \[
      P_X\{\delta_{A,-k}(X)^2\} = O_p(a_N^2), \qquad
          P_X\{\omega_{-k}(X)\} = O_p(b_N), \qquad
          P_X\{\eta_{-k}(X)^2\} = O_p(c_N^2).
    \]
\end{itemize}

\section{Proofs for surrogate-powered efficiency}
\label{app:efficiency-proofs}
We follow the sequence of results in Section~\ref{sec:efficiency}:
full-data efficiency, observed-data influence functions, AJ estimation,
and the efficiency gain.

\subsection{Efficiency with full data}\label{app:full-data-efficiency}
This subsection supplies the semiparametric definitions and proof deferred in
Section~\ref{sec:efficiency}. We use the joint and marginal models for
fully observed $(T,R)$ in Theorem~\ref{thm:no-gain}.

\begin{definition} 
A function $\phi\in L^2_0(P)$ is an influence function (IF) for the parameter $\psi(P)$ at $P$,
relative to a model $\mathcal{M}\equiv\{ P \}$, if:
for every regular one-dimensional submodel
$\{  P^{(\epsilon)} \}_{\epsilon\in(-\eta,\eta)}\subset\mathcal{M}$ with a score function
$s(t,r)\equiv\partial_\epsilon\log f^{(\epsilon)}(t,r)\big|_{\epsilon=0} \in L^2_0(P)$,
the pathwise derivative of $\psi$ at $P$ is given by the inner product of the score with $\phi$:
\begin{align*}
  \partial_\epsilon \psi(P^{(\epsilon)})\big|_{\epsilon=0}
  =
  E \{\phi(T,R) s(T,R)\} , \qquad\forall s\in\mathcal{T}_P(\mathcal{M}).
\end{align*}
An IF $\phi\equiv\phi_{P}(t,r)$ is determined jointly by $\psi$ and $\mathcal{M}$, and is an $L^2_0(P)$ function that varies with $P\in\mathcal{M}$.
Geometrically, $\phi_P$ is the $L^2(P)$ gradient of the functional $\psi$ on $\mathcal{M}$.

\end{definition}

When the model $\mathcal{M}$ is restricted and its tangent space $\mathcal{T}_P(\mathcal{M})$ is a proper subset of $L^2_0(P)$,
the IF is not unique and efficiency is not a trivial question.
The setting of censored data considered here leads to a restricted tangent space for the observed data,
even when the full-data distribution is nonparametric.

\begin{definition} 
The semiparametric efficiency bound for $\psi(P)$ relative to $\mathcal{M}$ is
$\inf_{\phi \in \textsf{IF}(\psi,P;\mathcal{M})} E \phi^2(T,R),$
where $\textsf{IF}(\psi,P;\mathcal{M})$ is the set of all influence functions.
The IF that achieves the bound is called the efficient influence function (EIF).
Geometrically, the EIF is the $L^2(P)$ projection of any IF onto $\mathcal{T}_P(\mathcal{M})$.
\end{definition}

Theorem~\ref{thm:no-gain} identifies the joint-model influence function with its marginal counterpart.
More explicitly, when $P_{R\mid T}$ is unrestricted and $\psi(P)=\psi(P_T)$,
the sets of influence-functions in the joint model $\mathcal M$ and its reduced marginal model $\mathcal M_0$ for $P_T$ coincide as functions of $T$.
Their efficiency bounds therefore coincide.
Hence an unrestricted auxiliary variable supplies no gain in the fully observed experiment.

\begin{proof}[Proof of Theorem~\ref{thm:no-gain}]
We use regular parametric submodels in a dominated model.
Let $\mathcal{T}_R$ be the nuisance tangent space corresponding to perturbations of $P_{R|T}$.
Since the conditional law $P_{R|T}$ is unrestricted,
$$
\mathcal{T}_R = \big\{ s(t,r)\in L^2(P) \ ;\ P[s(T,R)\,|\,T]=0 \text{ a.s. } \big\} .
$$
The parameter $\psi(P)$ is invariant under perturbations in $\mathcal{T}_R$.
Its EIF $\phi_P(t,r)$ therefore satisfies $E[\phi_P(T,R) s(T,R)] = 0$ for all $s \in \mathcal{T}_R$,
and hence $\phi_P\in\mathcal{T}^\perp_R$.
The orthogonal complement of $\mathcal{T}_R$ in $L^2_0(P_{T,R})$ consists of the mean-zero functions of the marginal variable with law $P_T$:
$$
\mathcal{T}^\perp_R = \big\{ g(T)\in L^2(P_T)\ ;\ Eg(T) = 0 \big\} .
$$
Consequently, the EIF in $\mathcal{M}$ for $P_{T,R}$ agrees with the EIF in the reduced model $\mathcal{M}_0$ for $P_T$:
$\phi_{P_T}(T)\equiv \phi_{P_{T,R}}(T,R)$ a.s..
The efficiency bounds for $\psi(P_{T,R})$ in $\mathcal{M}$ and $\mathcal{M}_0$ thus coincide.
\end{proof}

\subsection{Efficient influence functions}\label{app:eif-proofs}
The calculations in this subsection establish the representations in Theorem~\ref{thm:eif-main}.
We first record the score-martingale representation used in the censored-data calculations.

\begin{proof}[Direct verification of \Cref{eq:M-bullet}]
For $t<\tau$, use $v(t)=S(\tau)/S(t)$ and
$\mathrm{d}\{1/S(u)\}=\lambda_T(u)\,\mathrm{d}u/S(u)$ in the compensator decomposition to obtain
\begin{align}
  \label{eq:M-theta-bullet}
    \int_0^t-v(u)\,\mathrm{d}M_T^\bullet(u)
    &=-\frac{S(\tau)}{S(T\wedge t)}N_T^*(t)
    +S(\tau)\left\{\frac1{S(T\wedge t)}-1\right\}\notag
    \\
    &=v(t)Y^*(t+)-\theta
    =E\{Y^*_{\tau+}\mid\mathcal F_t^\bullet\}
      -E\{Y^*_{\tau+}\mid\mathcal F_0^\bullet\}
    =M_\theta^\bullet(t).
  \end{align}
\end{proof}

Next, we show that every score in the tangent space $L^2(P,\mathcal{F}^\bullet)$ has a martingale representation,
and is isometric to a hazard-score function.

\begin{lemma}  \label{lem:martingale-representation}
We ignore censoring $C$ in this lemma, and assume that $\mathcal{F}^\bullet$ is generated only by $N^*_T(t)$.
The tangent space $\mathcal{T}_P\equiv\big\{ g\in L^2(P,\mathcal{F}^\bullet) \ ;\ \mathbb{E}g(T)=0 \big\}$
of density score functions has the following characterization.
Define
\[
\mathcal{T}^M_P=
  \bigg\{
    \int_0^\infty h(t) \;\mathrm{d} M^\bullet_T(t)
    \ ;\
    \mathbb{E}h^2(T)=\int_0^\infty h^2(t) S(t) \lambda(t) \;\mathrm{d}t <\infty
  \bigg\} .
\]
Then $\mathcal{T}^M_P=\{g(T)\ ;\ g\in \mathcal{T}_P\}$.
In particular, the full-data IF of $\theta(P)$ can be expressed as:
\begin{equation} \label{eq:IF-M-star}
  \phi^\bullet(T)
  = I(T>\tau) - \theta
  = \int_0^\infty \big\{ I(s>\tau)-v(s)\big\} \;\mathrm{d} M^\bullet_T(s)
  = \int_0^\tau \;\mathrm{d} M^\bullet_\theta(s) .
\end{equation}
\end{lemma}
\begin{proof}[Proof of \Cref{lem:martingale-representation}]
For $h\in L^2(P_T)$, the martingale isometry and Fubini's theorem give
\[
  E\left\{\int_0^\infty h(t)\,\mathrm{d}M_T^\bullet(t)\right\}^{2}
  =\int_0^\infty h(t)^2S(t)\lambda(t)\,\mathrm{d}t
  =Eh(T)^2<\infty.
\]
The integral is $T$-measurable and has mean zero, proving the first inclusion.
Conversely, for $g\in\mathcal T_P$, the closed martingale
$M_g(t):=E\{g(T)\mid\mathcal F_t^\bullet\}$ is square-integrable, with terminal value $g(T)$.
The point-process representation theorem \cite{jacod1975multivariate} gives
$M_g(t)=\int_0^tH(s)\,\mathrm{d} M_T^\bullet(s)$.
Before the single jump, the event history specifies only that $T\ge t$;
thus $H(t)=h(t)Y^*(t)$ for a deterministic $h$.
Since $M_T^\bullet$ has no increments after $T$, this proves the reverse inclusion.
The isometry also gives $Eh(T)^2=Eg(T)^2<\infty$.

To identify $h$, put $V(t):=E\{g(T)\,|\,T>t\}$.
Strictly prior to $T$, $M_g(t)=V(t)$; at $T$, $M_g$ has a jump and takes the value $M_g(T)=g(T)$.
Absolute continuity of the event-time law makes $V$ continuous on the survival support,
so $\Delta M_g(T)=g(T)-V(T)=h(T)$.
For $g(t)=\phi^*(t)=I(t>\tau)-\theta$, the conditional survival probability
$v(t)$ defined above therefore gives $h(t)=I(t>\tau)-v(t)$.
It equals $-v(t)$ on $[0,\tau]$ and zero after $\tau$, proving \eqref{eq:IF-M-star}.
The hazard-score interpretation of $h$ follows from the likelihood differentiation in \eqref{eq:hazard-score}.
\end{proof}

The following calculation proves the coarsened-to-missingness-observation formula \eqref{eq:IF-Delta} in Theorem~\ref{thm:eif-main}.

With the goal of contrasting semiparametric efficiency of censoring and missingness,
we treat the censored observations $(\tilde T,\Delta=0)$ as missing data.
In other words, we use the censoring indicator $\Delta$ as the (non)missingness indicator.
This ignores the fraction of the population that is censored with mass $P\{C<T\}$.
We work in a fixed arm $A=a$ and drop the superscript $a$ from the notation.
Define
$T^\Delta\coloneqq\Delta \tilde T+(1-\Delta)\emptyset$
to be $\tilde T=T$, when the observed event time is death, and a missing value $\emptyset$, when the observed event time is $\tilde T=C$.
Consider the observed data $V\coloneqq(T^\Delta,\Delta)$.
The missingness mechanism $P\{\Delta=1\,|\,T\}=G(T)$ is MNAR,
and assumed to be known or identifiable from the standard censored data $(\tilde T,\Delta)$.
Below we consider perturbations to $P_C$ and find them to (i) span the observed-law tangent space, (ii) to cancel out in the pathwise derivative of $\theta$, we then ignore variations in $P_C$.

\begin{proof}[Proof of \Cref{eq:IF-Delta}]
Let $F=(T,C)$ denote the full data, with law $P_F$, and let $P_V$ denote the observed-data law.
The full-data factorization is
$P_F(\mathrm{d}t,\mathrm{d}c)=f_{T}(t)\mathrm{d}t\times f_{C}(c)\mathrm{d}c .$
The probability of observing $T$, conditional on $T$, is $P\{\Delta=1\,|\,T\}=P\{C>T\,|\,T\}=G(T).$
Thus $P\{\Delta=0\,|\,T\}=1-G(T)$, and the observed-data law is the continuous--discrete mixture
$$
  P_V(\mathrm{d}t,\delta)
  =
  \Big\{ f_T(t)G(t)\mathrm{d}t \Big\}^\delta \cdot
  \bigg\{ I_\emptyset(t)\int_0^\infty [1-G(s)] f_T(s) \; \mathrm{d}s \bigg\}^{1-\delta} .
$$

Fix a full-data score $s_F(t,c)=s_T(t)+s_C(c)$ for $P_F$,
where $s_T(t)=\partial_\epsilon\big|_{\epsilon=0}\log f^{(\epsilon)}_T(t)$,
and $s_C(c)=\partial_\epsilon\big|_{\epsilon=0}\log f^{(\epsilon)}_C(c)$
correspond to the perturbation $f^{(\epsilon)}_{T}(t)f^{(\epsilon)}_{C}(c)$ of $P_F$.
The induced observed-data score $s_V$ is obtained by differentiating each component of the mixture.
When $\Delta=0$,
\[
s_V(\emptyset,0)=\partial_\epsilon\big|_{\epsilon=0}
  \log \int I(c<t) \;\mathrm{d}P^{(\epsilon)}_{T,C}(t,c)
  =
  \int_0^\infty \int_0^t \big\{s_T(t)+s_C(c)\big\} \frac{I(c<t) f_{T,C}(t,c)}{P\{C<T\}} \;\mathrm{d}c \;\mathrm{d}t ,
\]
and, when $\Delta=1$,
\begin{align*}
s_V(t,1)
  =
  \partial_\epsilon\big|_{\epsilon=0}
  \log f^{(\epsilon)}_T(t) G^{(\epsilon)}(t)
  &=
  s_T(t) + \frac{1}{G(t)} \int_t^\infty s_C(c) f_C(c) \;\mathrm{d} c
  .
\end{align*}
Thus $s_V(t,\delta)= L[s_F]\coloneqq E[s_F(T,C)\,|\,V=(t,\delta)]$ is the conditional expectation of the $P_F$ score,
and the induced tangent space at $P_V$ is
\begin{align}
  \mathcal{T}_V
  &=
  \Big\{ s_V(T^\Delta,\Delta) = E[s_T(T)+s_C(C)| T^\Delta,\Delta ]
  \ ;\
  s_T\in L^2_0(P_T),\ s_C\in L^2_0(P_C) \Big\}, \notag \\
\label{eq:score-SV}
  s_V(t,\delta)
  &=
  \delta\cdot \big[s_T(t) + \tilde s_C(t)\big] + (1-\delta)\cdot E[s_T(T)+s_C(C) \,|\, C<T]
\end{align}
where $\tilde s_C(t) =  E[s_C(C) \,|\, C>t] $.
The missing-not-at-random mechanism $\Delta$ gives $s_V(\emptyset,0)\neq 0$.

The functional
\begin{equation}
\label{eq:observed-law-functional}
  \theta(P_V)
  :=
  \int \frac{\delta I(t>\tau)}{G(t)} \;\mathrm{d}P_V(t,\delta)
  =
  \int \frac{I(t>\tau)}{G(t)} f_T(t)G(t) \;\mathrm{d}t
\end{equation}
identifies $\theta(P_F)$. Along a perturbation $P^{(\epsilon)}_F$ with score $s_F$,
\begin{align*}
\partial_\epsilon\theta(P^{(\epsilon)}_V)
  &=
  \int \frac{\delta I(t>\tau)}{G(t)} \{ s_T(t)+\tilde s_C(t) \} \;\mathrm{d}P_V(t,\delta)
    +
    \int \frac{\delta I(t>\tau)}{G(t)} \{ -\tilde s_C(t) \}  \;\mathrm{d}P_V(t,\delta) .
\end{align*}
The $\tilde s_C$ terms cancel. Under the maintained independence assumption,
$\theta(P_V)$ is therefore invariant under perturbations of $G$.

The EIF of $\theta(P_V)$ is a $P_V$-score of the form
$\phi_V(t,\delta)\equiv\delta\cdot \phi_T(t) + (1-\delta)\cdot \phi_0$
for a $P_T$-score $\phi_T$ and $\phi_0\coloneqq{E}[\phi_T(T) \,|\, C<T]$.
Iterated expectation gives $P_V\phi_V=P_T\phi_T=0$.
On $\mathcal{T}_V$, the $L^2(P_V)$ inner product of pairs of scores of the form \eqref{eq:score-SV} is
\begin{align} \label{eq:observed-law-product}
  \langle u_V, v_V \rangle_{P_V}
  &=
  \int \big\{u_T(t)+\tilde u_C(t)\big\}\big\{v_T(t)+\tilde v_C(t)\big\} f_T(t) G(t) \;\mathrm{d} t
  \\
  &\hspace{1cm} + P\{C<T\} \cdot {E}[u_F(T,C)\,|\,C<T]\cdot {E}[v_F(T,C)\,|\,C<T] . \notag
\end{align}
The score components $L[s_T]$ and $L[s_C]$ are not orthogonal.
In particular, the only $s\in\mathcal{T}_V$ orthogonal to $E[s_C(C)\,|\,V]$ is the zero score.
Thus the missing-not-at-random model requires knowledge of $G$ external to the observed $V$ data for identification.

We now restrict attention to $\mathcal{T}_T$, the perturbations of $P_T$. We seek
$\phi_V\in\mathcal{T}_V$ of the form
$\phi_V(t,\delta)=\delta \phi_T(t) + (1-\delta)\phi_0$,
$\phi_0 = \mathbb{E}[\phi_T(T)\,|\,C<T]$, $\phi_T\in L^2_0(P_T)$,
such that, for every $s_T\in \mathcal{T}_F$ inducing \eqref{eq:score-SV},
\[
  P_F\{\phi_Fs_T\} = P_V\{\phi_Vs_V\} ,
  \qquad \forall  s_T\in\mathcal{T}_T,
  \qquad \phi_F(t)=I(t>\tau)-\theta .
\]
The form \eqref{eq:score-SV} does not require $\phi_V$ to be the image of
$\phi_F$; in particular, $\phi_T$ need not equal $\phi_F$.
By \cref{eq:observed-law-functional},
\begin{equation} \label{eq:IF-phiF}
  P_F\{\phi_Fs_F\}
  =
   \int \phi_F(t) s_T(t) \;\mathrm{d}P_F
  =
  \int \frac{\phi_F(t)}{G(t)} s_T(t) G(t) \;\mathrm{d}P_T(t)
\end{equation}
whereas the inner product in \cref{eq:observed-law-product} gives
\begin{align}
  P_V\{\phi_Vs_V\}
  &=
  \int \phi_T(t) s_T(t) G(t)  \;\mathrm{d}P_T(t)
  + \phi_0 \cdot \int
  s_T(t)[1-G(t)]
  \;\mathrm{d}P_T(t) \notag \\
\label{eq:IF-phiV}
  &=
  \int
    \Big\{ \phi_T(t)G(t)
    +
    \phi_0(1-G(t))
    \Big\} s_T(t)
  \;\mathrm{d}P_T(t),
  \qquad\qquad \forall  s_T\in\mathcal{T}_T.
\end{align}
Matching the integrands in \cref{eq:IF-phiF} and \cref{eq:IF-phiV} yields
$\phi_F(t) = \Big\{ \phi_T(t)G(t) + \phi_0 (1-G(t)) \Big\}$, equivalently $\phi_T(t) = \frac{\phi_F(t)}{G(t)} - \frac{\phi_0(1-G(t))}{G(t)}$.
Substituting the preceding expression for $\phi_T$ into
$\phi_0=E\{\phi_T(T)\mid C<T\}$ yields
\[
\phi_0\int\frac{1-G(t)}{G(t)}\,\mathrm{d}P_T(t)
=\int\phi_F(t)\frac{1-G(t)}{G(t)}\,\mathrm{d}P_T(t).
\]
Thus $\phi_0=\int\phi_F(t)\omega(t)\,\mathrm{d}P_T(t)$,
where $\omega(t)\coloneqq\frac{(1-G(t))/G(t)}{\int(1-G(t))/G(t)\,\mathrm{d}P_T(t)}$.
The centering constant is thus a tilted expectation of the full-data IF.
Substitution into the expression for $\phi_T$ gives
\begin{align}
  \phi_T(t)
  &=
  \frac{1}{G(t)} \big\{ \phi_F(t) - \phi_0(1-G(t)) \big\}
  =
  \frac{1}{G(t)} \big\{ \phi_F(t) - \phi_0 \big\} + \phi_0
  \notag \\
  \phi_V(t,\delta)
  &=
  \delta \phi_T(t) + (1-\delta)\phi_0
  =
  \frac{\delta}{G(t)} \big\{ \phi_F(t)-\phi_0\big\} + \phi_0
  =
  \frac{\delta}{G(t)}\phi_F(t) + \Big\{ 1-\frac{\delta}{G(t)} \Big\}\phi_0
  \notag \\
  &=
  \phi_F(t) + \Big\{ \frac{\delta}{G(t)}-1 \Big\} \big\{ \phi_F(t)-\phi_0\big\} .
  \label{eq:IF-phiVortho}
\end{align}
Finally, the decomposition in \cref{eq:IF-phiVortho} is orthogonal:
conditioning the product of its two terms on $T$ gives zero because
$E\{\Delta/G(T)-1\mid T\}=0$.

\end{proof}

We work in a fixed arm $a$ and drop the superscripts $a$ from the notation.

\begin{proof}[Proof of \Cref{eq:IF-circ}]
Let $F=(T,C)$ and $O=(\tilde T,\Delta)$, and work under \Cref{ass:time-reg,ass:C-ignore}.
For this proof, let $\mathcal F_t^\bullet=\sigma\{N_T^*(s),N_C^*(s):0\le s\le t\}$ be the joint full-data filtration,
and write $\bm N(t)=(N(t),N_C(t))$ and $Y(t)=I(\tilde T\ge t)$.
The full-data and observed-data likelihoods, \cite{jacod1975multivariate}, are
\begin{align*}
  L^\bullet(\bm N^*)
  &=
  \bigg\{\prod_{0<t\le\tau_0}\lambda(t)^{\mathrm{d}N^*_T(t)} \lambda_C(t)^{\mathrm{d}N^*_C(t)}\bigg\}
  \exp\bigg(-\int_0^{\tau_0}\{\lambda(t)Y^*_T+\lambda_C(t)Y^*_C(t)\}\,\mathrm{d}t\bigg)
  \\
  L^\circ(\bm N)
  &=
  \bigg\{\prod_{0<t\le\tau_0}\lambda(t)^{\mathrm{d}N(t)} \lambda_C(t)^{\mathrm{d}N_C(t)}\bigg\}
  \exp\bigg(-\int_0^{\tau_0}\{\lambda(t)+\lambda_C(t)\}Y(t)\,\mathrm{d}t\bigg).
\end{align*}
For a regular perturbation of the death hazard, define
\begin{equation} \label{eq:hazard-score}
  h(t)\coloneqq\partial_\epsilon\big|_{\epsilon=0}\log\lambda^{(\epsilon)}(t).
\end{equation}
Differentiating the log-likelihood gives the death score
$s^\circ(O)=\partial_\epsilon\log L^{\circ(\epsilon)}(\bm N) =\int_0^{\tau_0}h(t)\,\mathrm{d}M^\circ(t)$.
The corresponding full-data score is $s^\bullet(F)=\int_0^{\tau_0}h(t)\,\mathrm{d}M^\bullet(t)$;
the same calculation applies to the censoring hazard.
Thus the full-data and observed-data tangent spaces have the decompositions
$\mathcal T^\bullet=\mathcal T^\bullet_T\otimes\mathcal T^\bullet_C$ and
$\mathcal T^\circ=\mathcal T^\circ_T\otimes\mathcal T^\circ_C$, respectively, with
\[
\mathcal T^\circ
=\bigg\{
  \int_0^{\tau_0}h_T(t)\,\mathrm{d}M^\circ(t)
  +\int_0^{\tau_0}h_C(t)\,\mathrm{d}M_C^\circ(t) \;:\;
  h_T\in L^2(P_T),\ h_C\in L^2(P_C)
\bigg\}.
\]
The score operator $L[s^\bullet]=E\{ s^\bullet(F)\,|\,O\}$ replaces each full-data martingale integral by its observed-data counterpart.
Absence of simultaneous event and censoring jumps makes the two score components orthogonal.
Perturbing the censoring law does not change the target, although the influence function may depend on that law.

Differentiating $\theta=S(\tau)$ along a death-hazard perturbation gives
\begin{equation} \label{eq:IF-h-product}
  \left.\partial_\epsilon\theta(P^{\circ(\epsilon)})\right|_{\epsilon=0}
  =-S(\tau)\int_0^\tau h(t)\lambda(t)\,\mathrm{d}t.
\end{equation}
For two event-hazard scores, the martingale isometry gives
\begin{equation}
\label{eq:M-isometry}
  E\big\{s^\circ_1(O)s^\circ_2(O)\big\}
  =E\int_0^{\tau_0}h_1(t)h_2(t)Y(t)\lambda(t)\,\mathrm{d}t.
\end{equation}
Since $E\{Y(t)\}=S(t)G(t)$ and $v(t)=S(\tau)/S(t)$ for $t\le\tau$,
matching \eqref{eq:IF-h-product} with \eqref{eq:M-isometry} gives
$b(t)=-v(t)/G(t)$ on $[0,\tau]$ and zero thereafter.
This representer lies in the event-score tangent space and is orthogonal to the censoring scores.
Hence
\begin{equation} \label{eq:IF-Va-M}
  \phi^\circ(O)=\int_0^\tau\frac{-v(t)}{G(t)}\,\mathrm{d}M^\circ(t)
  =\int_0^\tau\frac{Y_C^*(t)}{G(t)}\,\mathrm{d}M_\theta^\bullet(t),
\end{equation}
where the last equality uses stopping at censoring and \eqref{eq:M-theta-bullet}.

For the augmentation representations, put $U(t)=I(C>t)/G(t)$ and $V(t)=M_\theta^\bullet(t)$.
The product rule is
\begin{equation} \label{eq:UV-integration}
  \mathrm{d}[UV](t)=U(t-)\mathrm{d}V(t)+V(t-)\mathrm{d}U(t)
  +\mathrm{d}[U,V](t).
\end{equation}
Here
$\mathrm{d}U(t)=I(C\ge t)\,\mathrm{d}\Lambda_C(t)/G(t) -\mathrm{d}N_C^*(t)/G(t)=-\mathrm{d}M_C^\bullet(t)/G(t)$ and
$\mathrm{d}V(t)=-v(t)\,\mathrm{d}M^\bullet(t)$.
The cross-variation vanishes because the event and censoring processes have no simultaneous jumps.
Therefore
\begin{equation} \label{eq:UV-formula}
  \mathrm{d}\left\{\frac{I(C>t)M_\theta^\bullet(t)}{G(t)}\right\}
  =\frac{I(C\ge t)}{G(t)}\,\mathrm{d}M_\theta^\bullet(t)
  -\frac{M_\theta^\bullet(t-)}{G(t)}\,\mathrm{d}M_C^\bullet(t).
\end{equation}
The initial product is zero. At $\tilde T\wedge\tau$, it is
$I(C>\tilde T\wedge\tau)\phi^*(\tilde T)/G(\tilde T\wedge\tau)$:
the indicator removes the censored cases in which the score martingale has not yet reached its terminal value.
Integrating \eqref{eq:UV-formula} to this endpoint and setting
$\Delta_{C,\tau}=I(C>\tilde T\wedge\tau)$ yields
\begin{equation} \label{eq:IF-circ-two}
  \phi^\circ(O)=\frac{\Delta_{C,\tau}}{G(\tilde T\wedge\tau)}\phi^*(\tilde T)
  +\int_0^{\tilde T\wedge\tau}
  \frac{M_\theta^\bullet(t-)}{G(t)}\,\mathrm{d}M_C^\bullet(t).
\end{equation}
The stopped interval of integration is essential for $\mathcal{F}^\circ$-measurability:
after $T$ the centered score martingale is constant but need not be zero,
so the full-data censoring integral without stopping need not be a function of $O$.

We develop a second augmented IF representation.
Use the product identity \eqref{eq:UV-integration} with
$V(t)=M_\theta^\bullet(t)+\theta=v(t)Y_T^*(t+)$.
The differentials are unchanged, while the boundary products at $0$ and
$\tau$ are $\theta$ and $I(C>\tau)I(T>\tau)/G(\tau)$.
Thus
\begin{equation} \label{eq:UV-MMC}
  \frac{I(C>\tau)I(T>\tau)}{G(\tau)}-\theta
  =\int_0^\tau\frac{-v(t)Y_C^*(t)}{G(t)}\,\mathrm{d}M^\bullet(t)
  -\int_0^\tau\frac{v(t)Y_T^*(t)}{G(t)}\,\mathrm{d}M_C^\bullet(t).
\end{equation}
Substitution of \eqref{eq:IF-Va-M} gives
\[
  \phi^\circ(O)=\frac{I(\tilde T>\tau)}{G(\tau)}-\theta
  +\int_0^\tau\frac{M_\theta^\bullet(t-)+\theta}{G(t)}\,\mathrm{d}M_C^\circ(t).
\]
Here $M_\theta^\bullet(t-)+\theta=v(t)Y_T^*(t)$ vanishes after death,
which permits the observed censoring martingale increments on $[0,\tau]$.

Finally, \eqref{eq:M-isometry} and independence of $T,C$ give
\[
\mathbb V\{\phi^\circ(O)\}
=\int_0^\tau\frac{v(t)^2S(t)G(t)\lambda(t)}{G(t)^2}\,\mathrm{d}t
=\theta^2\int_0^\tau\frac{\lambda(t)}{S(t)G(t)}\,\mathrm{d}t.
\]
The two martingale terms in \eqref{eq:UV-MMC} are orthogonal, and the
isometry gives the censoring term variance
$\theta^2\int_0^\tau\lambda_C(t)/\{S(t)G(t)\}\,\mathrm{d}t$.
Since $P(\tilde T>\tau)=S(\tau)G(\tau)$, the variance of its left side is
$\theta/G(\tau)-\theta^2$.
Subtracting proves the second expression:
\[
  \mathbb V\{\phi^\circ(O)\}
  =\frac{\theta}{G(\tau)}-\theta^2
  -\theta^2\int_0^\tau\frac{\lambda_C(t)}{S(t)G(t)}\,\mathrm{d}t.
\]
\end{proof}

\begin{proof}[Proof of \Cref{eq:IF-lorenge}]
Let $F=(T,R,C)$ and $O=(\tilde T,\Delta,\tilde R,\Gamma)$ denote the full and observed data; let $\bm{N^*}$ and $\bm{N}$ denote the corresponding transition processes.
The superscripts $\blacklozenge$ and $\lozenge$ denote full-data and observed-data regimes.

The observed-data likelihood, \cite{jacod1975multivariate}, is
\begin{align*}
L^\lozenge(O)&=
  \bigg\{ \prod_{0<t<\tau_0}
    \lambda_{01}(t)^{\mathrm{d}N_{01}(t)}
    \lambda_{02}(t)^{\mathrm{d}N_{02}(t)}
    \lambda_{12}(t)^{\mathrm{d}N_{12}(t)}
    \lambda_C(t)^{\mathrm{d}N_C(t)}
  \bigg\} \\
&\times
  \exp\bigg(
    -\int_0^{\tau_0}
    \{\lambda_{01}(t)+\lambda_{02}(t) \}Y_0(t)
    + \lambda_{12}(t)Y_1(t)
    + \lambda_C(t)Y_C(t)
    \;\mathrm{d} t
  \bigg) .
\end{align*}
Differentiating this likelihood, and its full-data counterpart with
uncensored risk sets, gives the score representations below.
The corresponding scores admit the martingale representations
\begin{equation} \label{eq:T-mrt}
  s^\blacklozenge_h(F)  = \sum_{jk\in\{01,02,12,C\}} \int_0^{\tau_0} h_{jk}(t) \;\mathrm{d} M^\blacklozenge_{jk}(t),
  \qquad
  s^\lozenge(O) = \sum_{jk\in\{01,02,12,C\}} \int_0^{\tau_0} h_{jk}(t) \;\mathrm{d} M^{\lozenge}_{jk}(t)
\end{equation}
where $h_{jk}(t)=\partial_\epsilon\big|_{\epsilon=0} \log \lambda^{(\epsilon)}_{jk}(t)$ are the hazard scores.
The tangent spaces factorize as
$\mathcal{T}^\blacklozenge = \mathcal{T}^\blacklozenge_{01} \otimes \mathcal{T}^\blacklozenge_{02} \otimes \mathcal{T}^\blacklozenge_{12} \otimes \mathcal{T}^\blacklozenge_C$, and
$\mathcal{T}^\lozenge= \mathcal{T}^\lozenge_{01} \otimes \mathcal{T}^\lozenge_{02} \otimes \mathcal{T}^\lozenge_{12} \otimes \mathcal{T}^\lozenge_C$
with $\mathcal{T}^\blacklozenge_{jk}=\big\{\int_0^{\tau_0} h(t) \;\mathrm{d}M^\blacklozenge_{jk} \ ;\ h\in L^2(P) \big\}$
and $\mathcal{T}^\lozenge_{jk}=\big\{\int_0^{\tau_0} h(t) \;\mathrm{d}M^\lozenge_{jk} \ ;\ h\in L^2(P) \big\}$.
The score operator $L:\mathcal{T}^\blacklozenge\to\mathcal{T}^\lozenge$
satisfies $s^\lozenge(O)=L[s^\blacklozenge]=\mathbb{E}[s^\blacklozenge(F)|O]$.
This follows either from the martingale representations in $\mathcal{F}^\blacklozenge$ and
$\mathcal{F}^\lozenge$ or from the martingale property of $M^\blacklozenge$ stopped at $C$:
\[
E\bigg[
  \int_0^C h_{jk}(t) \;\mathrm{d}M^\blacklozenge_{jk} +
  \int_C^{\tau_0} h_{jk}(t) \;\mathrm{d}M^\blacklozenge_{jk}
  \;\bigg|\; \mathcal{F}^\lozenge(\tau_0)
  \bigg]
  =
  \int_0^C h_{jk}(t) \;\mathrm{d}M^\blacklozenge_{jk} + 0
  =
  \int_0^{\tau_0} h_{jk}(t) \;\mathrm{d}M^\lozenge_{jk} .
\]

The full-data influence function for $\theta(P^\blacklozenge)$ is $\phi^*(t)=I(t>\tau)-\theta$.
We obtain the observed-data influence function by equating the Riesz representations of
the pathwise derivative of $\theta$ under perturbations of $P^\blacklozenge$ and $P^\lozenge$.
This determines $\phi^\lozenge$ for the functional $\theta(P^\lozenge)$ of $P^\lozenge$.
First express $\phi^*$ as a stochastic integral in $\mathcal{F}^\blacklozenge$ and apply the martingale isometry.
Define
$M^\blacklozenge_\theta(t) \coloneqq E[\phi^*(T)|\mathcal{F}^\blacklozenge(t)] = E[I(T>\tau)-\theta|\mathcal{F}^\blacklozenge(t)]$
which is a closed $L^2(P,\mathcal{F}^\blacklozenge)$ martingale.
It is the centered conditional expectation of the event $(T>\tau)$ given $\mathcal{F}^\blacklozenge(t)$.
At $t=0$, $M^\blacklozenge_\theta(0)=0$;
at $t=\tau$, $I(T>\tau)-\theta$ is $\mathcal{F}^\blacklozenge(\tau)$-measurable.
$M^\blacklozenge_\theta(t)$ is the information available for $\theta$ at $t$.

Representation theorem for multivariate counting processes
\cite{jacod1975multivariate} gives
\[
M^\blacklozenge_\theta(t) = \sum_{jk\in\{01,02,12,C\}} \int_0^t H^\blacklozenge_{jk}(s) \;\mathrm{d} M^\blacklozenge_{jk}(s)
\]
for predictable $L^2(P,\mathcal{F}^\blacklozenge)$ processes $H^\blacklozenge_{jk}$, unique up to indistinguishability.
To identify the integrands of $\phi^\lozenge$, consider the jumps of $M^\blacklozenge_\theta(t)$:
\[
\Delta M^\blacklozenge_\theta(t) \coloneqq M^\blacklozenge_\theta(t) - M^\blacklozenge_\theta(t-)
  =
  H^\blacklozenge_{jk}(t)\Big[\Delta N^*_{jk} - \lambda_{jk}(t)Y^*_{jk}(t)\cdot 0 \Big]
  \equiv
  H^\blacklozenge_{jk}(t)\Delta N^*_{jk} .
\]
Here the summation $\sum_{jk}$ is suppressed.
The definition of the conditional-expectation martingale also gives the jump $\Delta M^\blacklozenge_\theta(t)$ as
$\Delta M^\blacklozenge_\theta(t)=E\big[I(T>\tau)\,|\,\mathcal{F}^\blacklozenge(t)\big] - E\big[I(T>\tau)\,|\,\mathcal{F}^\blacklozenge(t-)\big]$
corresponding to a transition $0\to1,\ 0\to2,\ 1\to2$ or $j\to C$.
For a transition $j\to C$, independence
$N^*_C(t)\indep N^*_T(t)\,|\,N^*_R(t-)$ implies that a jump of $N^*_C(t)$ leaves the conditional survival score unchanged,
so $H^\blacklozenge_C(t)\cdot 1= v_j(t)-v_j(t-)=0$.
For a transition $j\to k$, the score at $t-$ in state $S(t-)=j$ is $v_j(t-)=v_j(t)$ by continuity of $v_j(t)$.
At $t$ in state $S(t)=k$, it equals $v_k(t)$.
Hence $H^\blacklozenge_{jk}(t)\cdot 1 = v_k(t)-v_j(t-)=v_k(t)-v_j(t).$
Since $v_2(t)=0$, the conditional-expectation martingale has the representation
\begin{align} \label{eq:M-theta}
  M^\blacklozenge_\theta(t)
  &=
  \int_0^t \{ v_1(s)-v_0(s) \}Y^*_0(s) \;\mathrm{d} M^*_{01}(s)
  \\
  &\hspace{1cm}+
  \int_0^t -v_0(s)Y^*_0(s) \;\mathrm{d} M^*_{02}(s)
  +
  \int_0^t -v_1(s)Y^*_1(s) \;\mathrm{d} M^*_{12}(s) \notag .
\end{align}
Applying Riesz' identity and martingale isometry to $M^\blacklozenge_\theta$,
using \cref{eq:T-mrt} and \cref{eq:M-theta}, gives
\begin{align*}
  \partial_\epsilon \theta(P^{\blacklozenge(\epsilon)})
  &=
  E\big\{M^\blacklozenge_\theta(\tau_0)\cdot s^\blacklozenge_h(F)\big\}
  =
  \sum_{jk}E \int_0^\tau H^\blacklozenge_{jk}(t) h_{jk}(t) \lambda_{jk}(t)Y^*_j(t)  \;\mathrm{d}t  .
\end{align*}
Thus \cref{eq:M-theta} is the representation of the full-data influence function
$\phi^*$ in terms of the transition counting processes.

We next equate the Riesz representation of $D\theta(P^\lozenge)$ in
$L^2(P^\lozenge, \mathcal{F}^\lozenge)$ with its full-data counterpart to identify $\phi^\lozenge$.
The functional $\theta(P^\lozenge)$ is identified from the observed law,
whose tangent space $\mathcal{T}^\lozenge$ is given in \cref{eq:T-mrt}.
For a perturbation $P^{\blacklozenge(\epsilon)}$ with hazard score $h=h_{01}+h_{02}+h_{12}$,
the score representations under $P^\lozenge$ and $P^\blacklozenge$ give
\begin{align*}
  \partial_\epsilon \theta(P^{\lozenge(\epsilon)})
  &=
  E\big\{ \phi^\lozenge(O) s^\lozenge_h(O) \big\}
  =
  \sum_{jk}E \int_0^\tau H^\lozenge_{jk}(t) h_{jk}(t) \lambda_{jk}(t)Y_j(t)  \;\mathrm{d}t \\
  &=
  E\big\{\phi^\blacklozenge(F) s^\blacklozenge_h(F) \big\}
  =
  \sum_{jk}E \int_0^\tau H^\blacklozenge_{jk}(t) h_{jk}(t) \lambda_{jk}(t)Y^*_j(t)  \;\mathrm{d}t
  = \partial_\epsilon \theta(P^{\blacklozenge(\epsilon)}) .
\end{align*}
It follows that
\begin{align}
  H^\lozenge_{jk}(t)
  &=
  \frac{ H^\blacklozenge_{jk}(t) E Y^*_j(t) }{ E Y_j(t) }
  =
  \frac{ H^\blacklozenge_{jk}(t)}{G(t)}, \notag \\
  \phi^\lozenge(O)
  &=
  \sum_{jk} \int_0^\tau \frac{H^\blacklozenge_{jk}(t)}{G(t)} Y^*_C(t) \;\mathrm{d}M^\blacklozenge_{jk}(t)
  \label{eq:EIF-M-theta}
  =
  \int_0^\tau \frac{Y^*_C(t)}{G(t)} \;\mathrm{d}M^\blacklozenge_\theta(t) .
\end{align}
In $\phi^\lozenge$, the integrand is the change in conditional survival at a transition $j\to k$.
Censoring reduces the risk sets $Y^*_j(t)$,
and the observed transitions receive inverse-observation-probability weights.

Apply the product identity \eqref{eq:UV-integration} with
$U(t)=Y_C^*(t+)/G(t)$ and $Z(t)=U(t)M_\theta^\blacklozenge(t)$.
The inverse-censoring differential is the one used in \eqref{eq:UV-formula},
now in the augmented filtration.
Joint absolute continuity excludes common censoring and transition jumps,
so $[U,M_\theta^\blacklozenge]=0$.
Hence
\begin{equation} \label{eq:UM-theta}
  \mathrm{d}Z(t) = \frac{Y^*_C(t)}{G(t)}\mathrm{d}M^\blacklozenge_\theta(t)
  -
  \frac{M^\blacklozenge_\theta(t-)}{G(t)}\mathrm{d}M^\blacklozenge_C(t) .
\end{equation}
Integrating \cref{eq:UM-theta} from
$t=0$ to $t=T\wedge\tau$ gives $\int_0^{T\wedge\tau} \;\mathrm{d}Z(t)=Z(T\wedge\tau)-Z(0)$ on the left.
At $t=0$, the boundary value is $Z(0)=1/1\cdot M^*_\theta(0)=0$; at $t=T\wedge\tau$,
it is $Z(T\wedge\tau) = \frac{\Delta_{C,\tau}}{G(T\wedge\tau)}[I(T>\tau)-\theta]$,
with $\Delta_{C,\tau}=I(C>T\wedge\tau)$,
where $\phi^*$ is observed only when $\Delta_{C,\tau}$ holds and is weighted by the inverse probability of that event.
In the integral $\int_0^{T\wedge\tau}$ on the right of \cref{eq:UM-theta},
the first term is $\phi^\lozenge$ from \cref{eq:EIF-M-theta}:
the martingale $M^\blacklozenge_\theta$ is constant after $T$, so stopping at $T\wedge\tau$ or $\tau$ is equivalent.
The second term satisfies
\[
  \int_0^{T\wedge\tau} \frac{M^\blacklozenge_\theta(t-)}{G(t)} \;\mathrm{d}M^\blacklozenge_C(t)
  =
  \int_0^{\tau} \frac{M^\blacklozenge_\theta(t-)}{G(t)} Y^*_T(t) \;\mathrm{d}M^\blacklozenge_C(t)
  =
  \int_0^{\tau} \frac{M^\blacklozenge_\theta(t-)}{G(t)} \;\mathrm{d}M^\lozenge_C(t) .
\]
Substitution in \cref{eq:UM-theta} gives the observable stopped representation
\begin{equation} \label{eq:IF-lozenge-two}
  \phi^\lozenge(O)=
  \frac{\Delta_{C,\tau}}{G(\tilde T\wedge\tau)}\phi^*(\tilde T)
  +\int_0^{\tilde T\wedge\tau}
  \frac{M_\theta^\blacklozenge(t-)}{G(t)}\,\mathrm{d}M_C^\blacklozenge(t).
\end{equation}
This proves both equalities in \cref{eq:IF-lorenge}.
We obtain a third representation.

Reuse the same product identity, now with $U(t)=Y^*_C(t+)/G(t)$ and
\[
  V(t)=E\{I(T>\tau)\,|\,\mathcal{F}^\blacklozenge(t)\}
  =
  Y^*_0(t+)v_0(t)+Y^*_1(t+)v_1(t)
  =
  M^\blacklozenge_\theta(t)+\theta
\]
on $[0,\tau]$. The differentials are unchanged, but the boundary terms give
$$
\frac{I(C>\tau)I(T>\tau)}{G(\tau)} - \theta = \int_0^\tau \frac{Y^*_C(t)}{G(t)} \;\mathrm{d}M^\blacklozenge_\theta(t) - \int_0^\tau V(t-) \;\frac{\mathrm{d} M^\blacklozenge_C(t)}{G(t)}
.
$$
Substituting $\phi^\lozenge$ from \cref{eq:UM-theta} and the definition of
$V(t)$ yields
\begin{align*}
  \phi^\lozenge(O)
  &=
  \frac{I(\tilde T>\tau)}{G(\tau)} - \theta
  +
  \int_0^\tau \big\{ M^\blacklozenge_\theta(t-)+\theta\big\}  \;\frac{\mathrm{d} M^\blacklozenge_C(t)}{G(t)} \\
  &=
  \frac{I(\tilde T>\tau)}{G(\tau)} - \theta
  +
  \int_0^{\tau} \big\{ v_0(t)Y^*_0(t)+v_1(t)Y^*_1(t) \big\} \; \frac{\mathrm{d}M^\blacklozenge_C(t)}{G(t)} .
  \end{align*}

\end{proof}

\subsection{No-covariate AJ estimation}
\label{sec:setup-notation}
This subsection supplies the proof of Theorem~\ref{thm:aj-integrated-main}.
We first establish the deterministic product-integral identities used in both the no-covariate and covariate-adjusted arguments.

We use the transition-hazard matrix and product integral in
\eqref{eq:matrix-product-integral}.
When its generating hazard law must be specified,
$P_A(s,t)$ denotes the matrix written $\mathbf P(s,t)$ in the main text.
Let $\ezero^\top=(1,0,0)$ and $\alive=(1,1,0)^\top$.
For an admissible law $A$, write
$\Psi_t(A)=\ezero^\top P_A(0,t)\alive$ and
$\mu_A(\tau)=\int_0^\tau\Psi_t(A)\,dt$.
Their conditional versions are denoted by $S_A(t\mid x)$ and $\mu_A(\tau\mid x)$,
and $\mu_A(x)$ abbreviates $\mu_A(\tau\mid x)$.
Thus the standardized functional underlying \eqref{eq:os-est-integrated-main} is
$\mu_\tau(A)=P_X\{\mu_A(\tau\mid X)\}$.
As in Section~\ref{sec:covariates}, $P$ and $P_X$ agree on functions of $X$.
For $t\le\tau$, $\Aspace_t$ denotes the admissible hazard class of that section on $[0,t]$: its coordinates are finite-valued,
nondecreasing and cadlag, vanish at zero, and have outgoing row increments at most one.
In particular, $\Aspace_\tau$ denotes the class written $\mathscr A_\tau$ in the main text.

The following norm conventions are used in the proofs.
For vectors, $\|x\|$ is Euclidean and $\|x\|_\infty=\max_j|x_j|$.
For matrices, $\|M\|$ is the spectral norm and
$
    \|M\|_\infty=\max_i\sum_j |M_{ij}|
$
is the row-sum norm.
Fixed-order tensors are identified with stacked vectors.
For scalar $f$ on $[0,t]$,
$
    \|f\|_{\infty,t}:=\sup_{0\le u\le t}|f(u)|,
  $ $
    \|f\|_\infty:=\|f\|_{\infty,\tau}.
$
For measurable scalar, vector, matrix, or tensor
$f$ and probability measure $Q$,
$
    \|f\|_{L^2(Q)}=\|f\|_{Q,2}:=\{Q\|f\|^2\}^{1/2}.
$ For a scalar cadlag function $f$ on $[a,b]$, define
$\PathVar(f;[a,b]) = \sup_{a=t_0<\cdots<t_m=b}\sum_{r=1}^m |f(t_r)-f(t_{r-1})|.$
For a matrix-valued function, $\PathVar$ denotes the sum of the path variations of
its scalar entries.

For a signed bounded-variation triplet
$D=(D_{01},D_{02},D_{12})$,
$\norm{D}_{\infty,t} := \sum_{\ell\in\cJ}\sup_{0\le u\le t}|D_\ell(u)|$, $\norm{D}_{\infty}:=\norm{D}_{\infty,\tau}$, $\norm{D}_{\Delta,t} := \sum_{\ell\in\cJ}\sup_{0<u\le t}|\Delta D_\ell(u)|$.
For a nonnegative hazard triplet, we use a shorthand notation
$
    |A|_t:=A_{01}(t)+A_{02}(t)+A_{12}(t).
$
For finite-variation scalar $F$, $dF$ is its Stieltjes measure, $\Delta F(u)=F(u)-F(u-)$, and $|F|_{\TV,t}:=|F|_{\mathrm{TV},t}:=|dF|([0,t])$.
The notation $|F|([a,b])$ also means $|dF|([a,b])$.
For signed triplets,
$|D|_{\TV,t}:=\sum_{\ell\in\cJ}|D_\ell|_{\TV,t}$,
$[D]_{\Delta,2,t} := \sum_{\ell\in\cJ}\sum_{0<u\le t}\{\Delta D_\ell(u)\}^2$.
For a matrix-valued signed measure $d\mathbf D$, $\|d\mathbf D\|_\infty:=\max_i\sum_j |d\mathbf D_{ij}|$ denotes the row-sum absolute measure.

We distinguish hazard errors from stochastic-integral weights.
Throughout the appendix, $D$ denotes a signed cumulative-hazard difference,
typically $D=B-A$ or
$D_{k,\ell}=\widehat A_{\ell}^{(-k)}-A_{\ell,0}$.  Boldface
$\mathbf D$ denotes the associated transition-hazard matrix difference.
The letter $H$ is reserved for predictable stochastic-integral weights, such
as $H_{\ell,A}$, $H_{\ell,0}$, and $\widehat H_{\ell,n}$.

Source and destination indices $j(\ell),k(\ell)$ are those of Section~\ref{sec:efficiency}.
Constants denoted by $C$, $C_r$, or
$C(M,\tau)$ may change from line to line and  are independent of  $n,m,u$, unless
explicitly stated.

For $A\in\Aspace_{\tau}$,
all jump matrices $I+\Delta\mathbf A(u)$ are row-stochastic, hence
\begin{equation}
\label{eq:row-stoch}
\norm{P_A(s,t)}_\infty=1,
\qquad 0\le s\le t\le \tau.
\end{equation}
The symbols $P_j^A,v_j^A,V_j^A$ denote the state-occupation and
residual-survival quantities of Section~\ref{subsec:estimation}, evaluated
under $A$ and, when applicable, conditionally on $X=x$.
In particular, $P_j^A(u\mid x)=P^A_{0j}(0,u\mid x)$.
At $A_0$, the superscript may be suppressed; a subscript $0$ on $V_j$
likewise denotes the true law. Row-stochasticity gives
\begin{equation}
\label{eq:v-bounds}
0\le v^A_j(u,t\mid x)\le1,
\qquad
0\le V^A_j(u,\tau\mid x)\le \tau.
\end{equation}

We use the product-integral moduli defined in
Section~\ref{sec:covariates}, with the norm conventions of
this subsection. For conditional hazards, the identities below are applied
at each covariate value; the same results will be used in
Section~\ref{app:covariate-proofs}.

\begin{lemma}
\label{lem:product-integral-expansion}
Let $A,B\in\Aspace_{\tau}$, $D=B-A$, and
$0\le s\le t\le\tau$.  Then
\begin{equation}
\label{eq:duhamel}
P_B(s,t)-P_A(s,t)=
\int_{(s,t]}P_A(s,u-)\,d\mathbf D(u)\,P_B(u,t),
\end{equation}
and also the left-hand variation-of-constants form
\begin{equation}
\label{eq:duhamel-left}
P_B(s,t)-P_A(s,t)=
\int_{(s,t]}P_B(s,u-)\,d\mathbf D(u)\,P_A(u,t).
\end{equation}
Moreover,
\begin{align}
\label{eq:product-integral-second-order}
P_B(s,t)-P_A(s,t)
&=
\int_{(s,t]}P_A(s,u-)\,d\mathbf D(u)\,P_A(u,t)
+R_{A,B}^{\mathrm{PI}}(s,t), \\
\label{eq:product-integral-remainder}
R_{A,B}^{\mathrm{PI}}(s,t)
&=
\int_{(s,t]}\int_{(u,t]}
P_A(s,u-)\,d\mathbf D(u)\,P_A(u,v-)\,d\mathbf D(v)\,P_B(v,t).
\end{align}
Moreover, if
\begin{equation}
\label{eq:product-integral-mass-bound}
|A|_{\tau}\vee |B|_{\tau}\le M<\infty,
\end{equation}
then there is a finite constant $C_{\mathrm{PI}}(M,\tau)$ such that
\begin{equation}
\label{eq:product-integral-lipschitz}
\kappa_{\mathrm{PI}}(A,B)
\le C_{\mathrm{PI}}(M,\tau)\norm{D}_{\infty,\tau}.
\end{equation}

\end{lemma}

\begin{proof} [Proof of Lemma \ref{lem:product-integral-expansion}]
First suppose that $A$ and $B$ are step functions on a common grid
$s=t_0<\cdots<t_q=t$. Set $F_r=I+\Delta_r\mathbf A$,
$G_r=I+\Delta_r\mathbf B$, and $K_r=G_r-F_r=\Delta_r\mathbf D$;
write $F_{a:b}=F_a\cdots F_b$ and $G_{a:b}=G_a\cdots G_b$, with an empty
product equal to $I$. Telescoping gives
\begin{equation}
\label{eq:finite-telescope}
G_{1:q}-F_{1:q}=\sum_{a=1}^qF_{1:a-1}K_aG_{a+1:q}.
\end{equation}
Take refining partitions whose mesh tends to zero and which contain all
jumps larger than the inverse partition index. The finite-variation
construction gives uniform convergence of the Euler products, while the
sum in \eqref{eq:finite-telescope} converges to the Stieltjes integral
with factors $P_A(s,u-)$ and $P_B(u,t)$. This proves
\eqref{eq:duhamel}, with the displayed left-limit convention.
Interchanging $A$ and $B$ gives \eqref{eq:duhamel-left}.
Adding and subtracting $P_A(u,t)$ in \eqref{eq:duhamel}, and applying
the same identity on $(u,t]$ to the difference $P_B(u,t)-P_A(u,t)$,
gives \eqref{eq:product-integral-second-order}--\eqref{eq:product-integral-remainder}.

For the modulus bound, first note that the signed matrix measure
\[
d\mathbf D(u)=
\begin{pmatrix}
-(dD_{01}+dD_{02}) & dD_{01} & dD_{02}\\
0&-dD_{12}&dD_{12}\\
0&0&0
\end{pmatrix},
\]
satisfies the row-wise absolute-measure bound
$\norm{d\mathbf D(u)}_\infty \le 2\,d|D_{01}|(u)+2\,d|D_{02}|(u)+2\,d|D_{12}|(u) =2\,dV_D(u)$,
where $dV_D=d|D_{01}|+d|D_{02}|+d|D_{12}|$. Assume now \eqref{eq:product-integral-mass-bound}. By the preceding
Euler-product approximation, it suffices to obtain a bound uniform over
finite grids. Each scalar entry of $G_{1:q}-F_{1:q}$ in
\eqref{eq:finite-telescope} is a finite sum of terms
\begin{equation}
\label{eq:product-integral-scalar-sum}
    \sum_{a=1}^q c_a x_a,
\end{equation}
Here $x_a$ is one of $\Delta D_{01}(t_a)$, $\Delta D_{02}(t_a)$, or $\Delta D_{12}(t_a)$, and $c_a$ is a
scalar entry of
\begin{equation}
\label{eq:product-integral-coefficient}
    C_a=F_{1:a-1}E_rG_{a+1:q},
\end{equation}
where $E_r$ ranges over the fixed elementary matrices associated with the
three transitions. Set $X_a=\sum_{b=1}^a x_b, \qquad E_x=\max_{0\le a\le q}|X_a|.$ Since $x_a$ is an increment of a
component of $D$ over $[s,t]$, we have $X_a=D_\ell(t_a)-D_\ell(s)$ and hence
\begin{equation}
\label{eq:partial-sum-bound}
    E_x\le 2\norm{D}_{\infty,\tau}.
\end{equation}
The factor two is absorbed into the final constant. With $c_{q+1}=0$, Abel
summation gives
\begin{equation}
\label{eq:abel-summation}
\sum_{a=1}^q c_a x_a
=\sum_{a=1}^q c_a(X_a-X_{a-1})
=\sum_{a=1}^q(c_a-c_{a+1})X_a,
\end{equation}
and therefore
\begin{equation}
\label{eq:abel-summation-bound}
\left|\sum_{a=1}^q c_a x_a\right|
\le E_x\left\{|c_q|+\sum_{a=1}^{q-1}|c_a-c_{a+1}|\right\}.
\end{equation}

It remains to bound the coefficient variation. Write $C_a$ for the matrix
in \eqref{eq:product-integral-coefficient}. Row-stochasticity of the $F$- and
$G$-products implies
\begin{equation}
\label{eq:coefficient-terminal-bound}
\norm{C_q}_\infty=\norm{F_{1:q-1}E_r}_\infty\le C_0,
\end{equation}
where $C_0$ depends only on the fixed elementary matrices. For $a=1,\ldots,q-1$,
\begin{align}
\label{eq:coefficient-increment}
C_a-C_{a+1}
&=F_{1:a-1}E_rG_{a+1:q}-F_{1:a}E_rG_{a+2:q} \\
&=F_{1:a-1}\{E_rG_{a+1}-F_aE_r\}G_{a+2:q} \notag \\
&=F_{1:a-1}\{E_r(G_{a+1}-I)-(F_a-I)E_r\}G_{a+2:q}. \notag
\end{align}
Row-stochasticity and boundedness of the elementary matrices then give
\begin{equation}
\label{eq:coefficient-increment-bound}
\norm{C_a-C_{a+1}}_\infty
\le C_0\left\{\norm{G_{a+1}-I}_\infty+\norm{F_a-I}_\infty\right\}.
\end{equation}
The one-step factors satisfy
\begin{equation}
\label{eq:product-factor-bound}
\norm{F_a-I}_\infty\le 2\{\Delta A_{01}(t_a)+\Delta A_{02}(t_a)+\Delta A_{12}(t_a)\},
\end{equation}
with the corresponding bound for $G_{a+1}-I$ obtained with $B$
in place of $A$. Summing \eqref{eq:coefficient-increment-bound} over $a$ and using
$|A|_{\tau}\vee |B|_{\tau}\le M$ yields
\begin{equation}
\label{eq:coefficient-variation-bound}
    \norm{C_q}_\infty+
    \sum_{a=1}^{q-1}\norm{C_a-C_{a+1}}_\infty
    \le C(M,\tau).
\end{equation}
This also bounds the variation of each scalar entry $c_a$ of $C_a$.
Combining \eqref{eq:product-integral-scalar-sum}--\eqref{eq:coefficient-variation-bound}
and summing over the finitely many matrix entries and transition coordinates,
we obtain
$\norm{P_B(s,t)-P_A(s,t)}_\infty \le C_{\mathrm{PI}}(M,\tau)\norm{D}_{\infty,\tau}$.
Taking the supremum over $0\le s\le t\le\tau$ proves \eqref{eq:product-integral-lipschitz}
on the grid. The constants are independent of the grid, so the preceding
approximation extends the bound to general cadlag finite-variation hazards.

\end{proof}

\begin{lemma}
\label{lem:rmst-hazard-derivative}
Let $A,B\in\Aspace_{\tau}$ and $D=B-A$.  Write
$M_{\tau}(A):=\mu_A(\tau).$
Throughout the paper,
$\dotmu_{\tau,A}[D]$ denotes the first-order pathwise derivative of the
deterministic RMST map $M_{\tau}$ at the hazard triplet $A$, evaluated in
the signed finite-variation direction $D$.  Equivalently, if
$A_\varepsilon=A+\varepsilon D$ is an admissible finite-variation path for small
$\varepsilon$, then
$\dotmu_{\tau,A}[D] =\left.\frac{d}{d\varepsilon}M_{\tau}(A_\varepsilon) \right|_{\varepsilon=0}.$
The paper uses the following explicit Stieltjes representation as the definition of
this derivative for all signed finite-variation directions $D$:
\begin{align}
\label{eq:rmst-pathwise-derivative}
\dotmu_{\tau,A}[D]
&=
\int_{(0,\tau]}
P^A_0(u-)
\{V^A_1(u,\tau)-V^A_0(u,\tau)\}
\,dD_{01}(u) \\
&\quad-
\int_{(0,\tau]}
P^A_0(u-)V^A_0(u,\tau)\,dD_{02}(u) \notag \\
&\quad-
\int_{(0,\tau]}
P^A_1(u-)V^A_1(u,\tau)\,dD_{12}(u). \notag
\end{align}
Thus $\dotmu_{\tau,A}[D]$ is a scalar linear functional of the three signed
hazard increments $(dD_{01},dD_{02},dD_{12})$.  It is the linear term in the RMST
expansion
\begin{equation}
\label{eq:rmst-first-order-expansion}
\mu_B(\tau)-\mu_A(\tau)=\dotmu_{\tau,A}[D]+r_{A,B},
\end{equation}
with remainder satisfying
$|r_{A,B}|\le \tau\omega_{\mathrm{PI}}(A,B)$.
Also,
$|\mu_B(\tau)-\mu_A(\tau)| \le 2\tau\kappa_{\mathrm{PI}}(A,B).$
\end{lemma}

\begin{proof} [Proof of Lemma \ref{lem:rmst-hazard-derivative}]
For $t\in[0,\tau]$, Lemma~\ref{lem:product-integral-expansion} gives
\begin{align}
\label{eq:survival-first-order-expansion}
\Psi_t(B)-\Psi_t(A)
&=\ezero^\top\int_{(0,t]}P_A(0,u-)\,d\mathbf D(u)P_A(u,t)\alive+\rho_t,\\
\label{eq:survival-remainder-bound}
|\rho_t|&\le\omega_{\mathrm{PI}}(A,B).
\end{align}
Write $P_j^A(u-)=P_{0j}^A(0,u-)$ and
$v^A(u,t)=P_A(u,t)\alive=(v_0^A(u,t),v_1^A(u,t),0)^\top$.
Multiplying the transition-hazard matrix by this vector yields
\begin{align}
\label{eq:survival-derivative}
D\Psi_{t,A}[D]
&:=\ezero^\top\int_{(0,t]}P_A(0,u-)\,d\mathbf D(u)P_A(u,t)\alive\notag\\
&=\int_{(0,t]}P_0^A(u-)\{v_1^A(u,t)-v_0^A(u,t)\}\,dD_{01}(u)\notag\\
&\quad-\int_{(0,t]}P_0^A(u-)v_0^A(u,t)\,dD_{02}(u)
      -\int_{(0,t]}P_1^A(u-)v_1^A(u,t)\,dD_{12}(u).
\end{align}
Each scalar coefficient has absolute value at most one by
\eqref{eq:v-bounds}, and
\begin{equation}
\label{eq:rmst-fubini-bound}
\sum_{\ell\in\cJ}\int_0^\tau\int_{(0,t]}d|D_\ell|(u)\,dt
=\sum_{\ell\in\cJ}\int_{(0,\tau]}(\tau-u)\,d|D_\ell|(u)
\le\tau\sum_{\ell\in\cJ}|D_\ell|([0,\tau])<\infty.
\end{equation}
Fubini's theorem for finite signed measures therefore permits integration
of \eqref{eq:survival-derivative} in $t$. The inner integrals are $V_1^A-V_0^A$,
$V_0^A$, and $V_1^A$, respectively, proving
\eqref{eq:rmst-pathwise-derivative}.
Integrating \eqref{eq:survival-first-order-expansion} gives
$r_{A,B}=\int_0^\tau\rho_t\,dt$ and
$|r_{A,B}|\le\tau\omega_{\mathrm{PI}}(A,B)$.
Finally,
$|\Psi_t(B)-\Psi_t(A)|\le2\|P_B(0,t)-P_A(0,t)\|_\infty$,
so integration proves
\begin{equation}
\label{eq:rmst-product-integral-bound}
|\mu_B-\mu_A|\le2\tau\kappa_{\mathrm{PI}}(A,B).
\end{equation}
\end{proof}

We use the Nelson--Aalen and AJ estimators of
Section~\ref{subsec:estimation} and~\eqref{eq:AJ}, under
Assumption~\ref{ass:aj-main-integrated}. The averages $\bar N_\ell,\bar Y_j$
and the population source risks $q_j$ have the meanings given there.
For the proofs, put $J_j(u)=\1\{\bar Y_j(u)>0\}$ and
$\bar M_{\ell,0}=n^{-1}\sum_iM_{\ell,0,i}$, retaining the zero-risk ratio
convention of the main text. Write
$\widehat\mu^{\AJ}=\widehat\mu_\tau^{\textsf{AJ}}$ and $\mu_0=\mu_\tau$.
The notation $H_{\ell,0}$ denotes the true RMST weight
$H_\ell^\lozenge$ in \eqref{eq:main-nocov-rmst-H}; the corresponding
influence function is \eqref{eq:main-nocov-rmst-eif}.

For the proofs in this appendix, define the transition matrices
\[
\mathbf B_{01}=\begin{pmatrix}-1&1&0\\0&0&0\\0&0&0\end{pmatrix},\qquad
\mathbf B_{02}=\begin{pmatrix}-1&0&1\\0&0&0\\0&0&0\end{pmatrix},\qquad
\mathbf B_{12}=\begin{pmatrix}0&0&0\\0&-1&1\\0&0&0\end{pmatrix}.
\]
Thus $d\mathbf H(u)=\sum_{\ell\in\cJ}\mathbf B_\ell dH_\ell(u)$.  The
row-stochasticity bound \eqref{eq:row-stoch} and censoring positivity imply
\begin{equation}
\label{eq:rmst-weight-bound}
    H_{\max}:=
    \max_{\ell\in\cJ}\sup_{0\le u\le\tau}|H_{\ell,0}(u)|
    \le 2\tau/g_0<\infty .
\end{equation}

\begin{lemma}
\label{lem:aj-jump-regularity}
Under Assumptions~\ref{ass:time-reg}, \ref{ass:C-ignore}, \ref{ass:Markov-state}, and~\ref{ass:aj-main-integrated}, let $R_i$ be the latent recurrence time and $T_i$
the latent terminal event time.  Define the transition times
\[
S_{01,i}=\begin{cases}R_i, & R_i<T_i,\\ \infty, & R_i\ge T_i,\end{cases}
\qquad
S_{02,i}=\begin{cases}T_i, & T_i<R_i,\\ \infty, & T_i\ge R_i,\end{cases}
\qquad
S_{12,i}=\begin{cases}T_i, & R_i<T_i,\\ \infty, & R_i\ge T_i.\end{cases}
\]
Then the following hold.
\begin{enumerate}[label=\textnormal{(\roman*)}]
\item The true transition compensators are continuous:
$\Delta A_{\ell,0}(u)=0, \qquad 0<u\le\tau, \qquad \ell\in\cJ .$
\item The total transition-hazard masses are finite and
$A_0\in\Aspace_{\tau}$.
\item The Nelson--Aalen estimator is row-admissible pathwise:
$\widehat A^{\AJ}\in\Aspace_{\tau}$ for every sample for which the displayed
risk-set ratios are defined with the convention $J_j/\bar Y_j=0$ on
$\{\bar Y_j=0\}$.
\item There are no common calendar-time observed transition jumps in the sample:
\[
J_n^{\max}
    :=\max_{0<u\le\tau}
      \sum_{i=1}^n\sum_{\ell\in\cJ}\Delta N_{\ell,i}(u)
    \le 1
    \qquad\textnormal{almost surely}.
\]
Consequently $J_n^{\max}=O(1)=o(\sqrt n)$ almost surely.
\end{enumerate}
\end{lemma}

\begin{proof} [Proof of Lemma \ref{lem:aj-jump-regularity}]
Assumptions~\ref{ass:time-reg}, \ref{ass:C-ignore}, and~\ref{ass:Markov-state}
imply $\Pr(R_i=T_i)=0$. Thus a subject cannot have simultaneous recurrence and
terminal-event jumps. When both events occur, their times are $R_i$ and
$T_i$ with $R_i<T_i$; when death precedes recurrence, the only observed
transition is $0\to2$. Censoring may remove a transition but cannot create one.

The same assumptions exclude ties between subjects: for $i\ne k$ and
$\ell,r\in\cJ$,
$\Pr\{S_{\ell,i}=S_{r,k}<\infty\}=0.$
A finite union over subject and transition pairs gives
$\Pr\{\exists i\ne k,\exists\ell,r,\ S_{\ell,i}=S_{r,k}<\infty\}=0.$
Together with the absence of within-subject simultaneous jumps, this yields
$J_n^{\max}\le1$ almost surely.

For continuity of the compensators, fix $\ell=(j,k)$ and $u\in(0,\tau]$. The absence
of atoms in $S_{\ell,i}$ implies
$\Pr\{\Delta N_{\ell,i}(u)=1\}=0 .$
Taking expectations in the Doob--Meyer decomposition gives
$0=E\{\Delta N_{\ell,i}(u)\} =E\{Y_{j,i}(u)\}\Delta A_{\ell,0}(u) =q_j(u)\Delta A_{\ell,0}(u).$
When $q_j(u)>0$, we obtain $\Delta A_{\ell,0}(u)=0$. When $q_j(u)=0$, the source-support
condition in Assumption~\ref{ass:aj-main-integrated}\ref{ajmain:support}
gives
$\Delta A_{\ell,0}(u) \le \int_0^{\tau}\mathbf 1\{q_j(v)=0\}\,dA_{\ell,0}(v)=0 .$
Hence $A_{\ell,0}$ is continuous. Since $0\le q_j\le1$,
\[
A_{\ell,0}(\tau)
=\int_{(0,\tau]}\mathbf 1\{q_j(u)>0\}\,dA_{\ell,0}(u)
\le
\int_{(0,\tau]}
\frac{\mathbf 1\{q_j(u)>0\}}{q_j(u)}\,dA_{\ell,0}(u)<\infty .
\]
The true hazards therefore have finite mass, and continuity ensures their
row constraints, so $A_0\in\Aspace_{\tau}$.

Finally, fix a time $u$ and source state $j$. The outgoing
Nelson--Aalen increment from row $j$ is the number of transitions from
subjects at risk in state $j$ just before $u$, divided by the
row-$j$ risk-set size, with value zero for an empty risk set. This
ratio lies in $[0,1]$ because each at-risk subject contributes at most one
outgoing transition. Thus the row constraints defining $\Aspace_{\tau}$ hold
pathwise for $\widehat A^{\AJ}$.
\end{proof}

\begin{theorem}
\label{thm:aj-rmst-linearization}
Under Assumption~\ref{ass:aj-main-integrated},
\begin{equation}
\label{eq:aj-rmst-linearization}
\widehat\mu^{AJ}-\mu_0
=
\dotmu_{\tau,A_0}[\widehat A^{\AJ}-A_0]+o_p(n^{-1/2}).
\end{equation}
\end{theorem}

\begin{proof}[Proof of Theorem~\ref{thm:aj-rmst-linearization}]
We use the left-hand product-integral identity to express the RMST remainder
as a stochastic integral with a predictable integrand, and then bound its
Doob--Meyer components. Set
$\widehat{\mathbf A}^{\AJ}(t)=\sum_{\ell\in\cJ}\mathbf B_\ell \widehat A^{\AJ}_\ell(t)$, $\mathbf A_0(t)=\sum_{\ell\in\cJ}\mathbf B_\ell A_{\ell,0}(t)$, $\mathbf D_n=\widehat{\mathbf A}^{\AJ}-\mathbf A_0$.
For a transition $\ell$, write $j=j(\ell)$ and
$D_{n\ell}(t)=\widehat A^{\AJ}_\ell(t)-A_{\ell,0}(t).$
The decomposition $d\bar N_\ell=d\bar M_{\ell,0}+\bar Y_j\,dA_{\ell,0}$ gives
\begin{equation}
\label{eq:na-error-decomposition}
D_{n\ell}(t)=U_{n\ell}(t)+Z_{n\ell}(t),
\end{equation}
where
$U_{n\ell}(t)=\int_{(0,t]}\frac{J_j(u)}{\bar Y_j(u)}\,d\bar M_{\ell,0}(u)$, $Z_{n\ell}(t)=\int_{(0,t]}\{J_j(u)-1\}\,dA_{\ell,0}(u)$.
and $J_j/\bar Y_j=0$ on $\{\bar Y_j=0\}$.

We first establish the bounds
\begin{equation}
\label{eq:na-sup-rate}
    \max_{\ell\in\cJ}\sup_{t\le\tau}|D_{n\ell}(t)|=O_p(n^{-1/2}),
\end{equation}
and
\begin{equation}
\label{eq:zero-risk-variation}
    \max_{\ell\in\cJ}|Z_{n\ell}|(\tau)=O_p(n^{-1}).
\end{equation}
Since $Z_{n\ell}$ is nonincreasing,
$|Z_{n\ell}|(\tau)=\int_{(0,\tau]} \1\{\bar Y_j(u)=0\}\,dA_{\ell,0}(u).$
Consequently,
\begin{align}
\label{eq:zero-risk-expectation-bound}
E|Z_{n\ell}|(\tau)
&\le \int_{(0,\tau]}P\{\bar Y_j(u)=0\}\,dA_{\ell,0}(u) \\
&\le \int_{(0,\tau]}e^{-nq_j(u)}\,dA_{\ell,0}(u)
\le \frac{1}{n}\int_{(0,\tau]}
\frac{\1\{q_j(u)>0\}}{q_j(u)}\,dA_{\ell,0}(u)=O(n^{-1}), \notag
\end{align}
where Assumption~\ref{ass:aj-main-integrated} excludes hazard mass on
$\{q_j=0\}$. The predictable quadratic variation of the martingale component
satisfies
$\langle U_{n\ell}\rangle(t) \le \frac{1}{n}\int_{(0,t]}\frac{J_j(u)}{\bar Y_j(u)}\,dA_{\ell,0}(u)$.
For $B\sim\mathrm{Bin}(n,q)$,
\begin{equation}
\label{eq:binomial-inverse-bound}
    E\left\{\frac{\1(B>0)}{B/n}\right\}
    =nE\left\{\frac{\1(B>0)}{B}\right\}
    \le \frac{C}{q},
    \qquad 0<q\le1,
\end{equation}
with a universal constant $C$. Thus
$E\langle U_{n\ell}\rangle(\tau) \le \frac{C}{n}\int_{(0,\tau]} \frac{\1\{q_j(u)>0\}}{q_j(u)}\,dA_{\ell,0}(u)=O(n^{-1}).$
Doob's $L_2$ inequality now yields
$E\sup_{t\le\tau}|U_{n\ell}(t)|^2 \le 4E\langle U_{n\ell}\rangle(\tau)=O(n^{-1}).$
Together with \eqref{eq:zero-risk-expectation-bound}, this establishes
\eqref{eq:na-sup-rate} and implies
\begin{equation}
\label{eq:aj-hazard-mass-bound}
    |\widehat A^{\AJ}|_{\tau}=O_p(1).
\end{equation}

Apply \eqref{eq:duhamel-left} with $A=A_0$, $B=\widehat A^{\AJ}$, and $H=\widehat A^{\AJ}-A_0$.
Multiplication by $\ezero^\top$ and $\alive$, followed by integration in
$t$ and Stieltjes--Fubini, gives
\begin{equation}
\label{eq:aj-exact-rmst-identity}
\widehat\mu^{\AJ}-\mu_0
=\sum_{\ell\in\cJ}\int_{(0,\tau]}L_{n\ell}(u)\,dD_{n\ell}(u),
\end{equation}
where
$L_{n\ell}(u) =\int_u^{\tau} \ezero^\top P_{\widehat A^{\AJ}}(0,u-)\mathbf B_\ell P_{A_0}(u,t)\alive\,dt .$
Define the deterministic coefficient
$L_{0\ell}(u) =\int_u^{\tau} \ezero^\top P_{A_0}(0,u-)\mathbf B_\ell P_{A_0}(u,t)\alive\,dt .$
Lemma~\ref{lem:rmst-hazard-derivative} gives
\begin{equation}
\label{eq:aj-rmst-derivative}
\dotmu_{\tau,A_0}[\widehat A^{\AJ}-A_0]
=\sum_{\ell\in\cJ}\int_{(0,\tau]}L_{0\ell}(u)\,dD_{n\ell}(u).
\end{equation}
Subtracting \eqref{eq:aj-rmst-derivative} from
\eqref{eq:aj-exact-rmst-identity}, we obtain
\begin{equation}
\label{eq:aj-predictable-remainder}
\widehat\mu^{\AJ}-\mu_0-\dotmu_{\tau,A_0}[\widehat A^{\AJ}-A_0]
=\sum_{\ell\in\cJ}\int_{(0,\tau]}\Delta L_{n\ell}(u)\,dD_{n\ell}(u),
\end{equation}
where $\Delta L_{n\ell}=L_{n\ell}-L_{0\ell}$.

To bound the coefficient error, let
$\kappa_n=\sup_{0\le u\le\tau} \|P_{\widehat A^{\AJ}}(0,u)-P_{A_0}(0,u)\|_\infty .$
By \eqref{eq:row-stoch} and $\|\alive\|_\infty=1$,
\begin{equation}
\label{eq:aj-coefficient-error-bound}
    \max_{\ell\in\cJ}\sup_{u\le\tau}|\Delta L_{n\ell}(u)|
    \le C\tau \kappa_n .
\end{equation}
On $\{|A_0|_{\tau}\vee |\widehat A^{\AJ}|_{\tau}
\le M\}$, Lemma~\ref{lem:product-integral-expansion} implies
$\kappa_n\le C_M\max_{\ell\in\cJ}\sup_{u\le\tau}|D_{n\ell}(u)|.$
Combining this with \eqref{eq:na-sup-rate} and
\eqref{eq:aj-hazard-mass-bound} yields
\begin{equation}
\label{eq:aj-product-integral-rate}
    \kappa_n=O_p(n^{-1/2}),
    \qquad
    \max_{\ell}\sup_{u}|\Delta L_{n\ell}(u)|=O_p(n^{-1/2}).
\end{equation}
The decomposition \eqref{eq:na-error-decomposition} separates the remainder in
\eqref{eq:aj-predictable-remainder} into
$R_{U,n}=\sum_{\ell}\int\Delta L_{n\ell}\,dU_{n\ell}$, $R_{X,n}=\sum_{\ell}\int\Delta L_{n\ell}\,dZ_{n\ell}$.
By \eqref{eq:aj-coefficient-error-bound} and \eqref{eq:zero-risk-variation},
\begin{equation}
\label{eq:aj-zero-risk-remainder}
|R_{X,n}|
\le \sum_{\ell}\sup_u|\Delta L_{n\ell}(u)|\,|Z_{n\ell}|(\tau)
=O_p(n^{-1/2})O_p(n^{-1})=O_p(n^{-3/2}).
\end{equation}

For each $\ell$, the process
$M_{n\ell}^{\nabla}(t)=\int_{(0,t]}\Delta L_{n\ell}(u)\,dU_{n\ell}(u)$
has a predictable integrand because $\Delta L_{n\ell}(u)$ uses only
$P_{\widehat A^{\AJ}}(0,u-)$ and deterministic product integrals under $A_0$.
It is therefore a square-integrable martingale with predictable quadratic variation
bounded by

\[
\langle M_{n\ell}^{\nabla}\rangle(\tau) \le \frac1n\int_{(0,\tau]} \Delta L_{n\ell}(u)^2\frac{J_j(u)}{\bar Y_j(u)}\,dA_{\ell,0}(u).
\]

Set
$I_{n\ell}=\int_{(0,\tau]} \frac{J_j(u)}{\bar Y_j(u)}\,dA_{\ell,0}(u).$
The binomial bound \eqref{eq:binomial-inverse-bound} gives
$E I_{n\ell}\le C\int_{(0,\tau]} \frac{\1\{q_j(u)>0\}}{q_j(u)}\,dA_{\ell,0}(u)<\infty,$
and hence $I_{n\ell}=O_p(1)$. Therefore \eqref{eq:aj-coefficient-error-bound} and
\eqref{eq:aj-product-integral-rate} imply
\begin{equation}
\label{eq:aj-remainder-bracket}
\langle M_{n\ell}^{\nabla}\rangle(\tau)
\le \frac{C}{n}
\left(\sup_u|\Delta L_{n\ell}(u)|^2\right) I_{n\ell}
=O_p(n^{-2}).
\end{equation}
Lenglart's inequality and \eqref{eq:aj-remainder-bracket} give
$M_{n\ell}^{\nabla}(\tau)=O_p(n^{-1})$.
There are only three transitions, so $R_{U,n}=O_p(n^{-1})$. Together with
\eqref{eq:aj-zero-risk-remainder}, this gives
\begin{equation}
\label{eq:aj-rmst-remainder}
\widehat\mu^{\AJ}-\mu_0-\dotmu_{\tau,A_0}[\widehat A^{\AJ}-A_0]
=O_p(n^{-1})=o_p(n^{-1/2}).
\end{equation}
The expansion \eqref{eq:aj-rmst-linearization} follows. Comparing
\eqref{eq:aj-rmst-remainder} with the exact expansion in
Lemma~\ref{lem:rmst-hazard-derivative} gives the equivalent assertion about the
product-integral remainder.

\end{proof}

\begin{lemma}
\label{lem:weighted-na-linearization}
Fix a transition $\ell=(j,k)$.  Let $h_\ell$ be a bounded deterministic function
on $[0,\tau]$, and define
$L_\ell(u)=q_j(u)h_\ell(u).$
Under Assumption~\ref{ass:aj-main-integrated},
\begin{equation}
\label{eq:weighted-na-linearization}
\int_0^{\tau}
L_\ell(u)\,d\{\widehat A_\ell^{\AJ}-A_{\ell,0}\}(u)
=
\int_0^{\tau}h_\ell(u)\,d\bar M_{\ell,0}(u)
+
O_p(n^{-1}).
\end{equation}
No lower bound on $q_j$ is required.
\end{lemma}

\begin{proof} [Proof of Lemma \ref{lem:weighted-na-linearization}]
The Doob--Meyer decomposition
$d\bar N_\ell(u)=d\bar M_{\ell,0}(u)+\bar Y_j(u)dA_{\ell,0}(u),$
and the convention $J_j/\bar Y_j=0$ on $\bar Y_j=0$ give
$d\{\widehat A_\ell^{\AJ}-A_{\ell,0}\}(u) = \frac{J_j(u)}{\bar Y_j(u)}d\bar M_{\ell,0}(u) +\{J_j(u)-1\}dA_{\ell,0}(u).$
Multiplication by $L_\ell=q_jh_\ell$ and addition and subtraction of $\int h_\ell d\bar M_{\ell,0}$ yield
\begin{align*}
\int L_\ell d(\widehat A_\ell^{\AJ}-A_{\ell,0})
&=
\int h_\ell d\bar M_{\ell,0}+R_{1,n}+R_{2,n}, \\
R_{1,n}
&=
\int h_\ell(u)
\left\{\frac{q_j(u)J_j(u)}{\bar Y_j(u)}-1\right\}d\bar M_{\ell,0}(u), \\
R_{2,n}
&=
-\int q_j(u)h_\ell(u)\1\{\bar Y_j(u)=0\}dA_{\ell,0}(u).
\end{align*}
For $R_{2,n}$, boundedness $\|h_\ell\|_\infty<\infty$ implies
$E|R_{2,n}| \le \|h_\ell\|_\infty \int_0^{\tau}q_j(u)\Pr\{\bar Y_j(u)=0\}dA_{\ell,0}(u).$
At each fixed $u$, $K_u=n\bar Y_j(u)\sim \mathrm{Binomial}(n,q_j(u))$, and therefore $\Pr\{\bar Y_j(u)=0\}=\{1-q_j(u)\}^n\le e^{-nq_j(u)}$. Since $x e^{-nx}\le(en)^{-1}$,
$E|R_{2,n}| \le \frac{\|h_\ell\|_\infty}{en}A_{\ell,0}(\tau)=O(n^{-1}),$
so $R_{2,n}=O_p(n^{-1})$.

For $R_{1,n}$, the predictable quadratic variation satisfies
\begin{align*}
\langle R_{1,n}\rangle(\tau)
&\le
\frac1{n^2}\sum_{i=1}^n\int_0^{\tau}
 h_\ell(u)^2
\left\{\frac{q_j(u)J_j(u)}{\bar Y_j(u)}-1\right\}^2
Y_{j,i}(u)dA_{\ell,0}(u) \\
&=
\frac1n\int_0^{\tau}
 h_\ell(u)^2
\left\{\frac{q_j(u)J_j(u)}{\bar Y_j(u)}-1\right\}^2
\bar Y_j(u)dA_{\ell,0}(u).
\end{align*}
The integrand vanishes on $\{\bar Y_j(u)=0\}$, whereas on $\{\bar Y_j(u)>0\}$,
$\left\{\frac{q_jJ_j}{\bar Y_j}-1\right\}^2\bar Y_j = \frac{\{q_j-\bar Y_j\}^2}{\bar Y_j}.$
Thus
\begin{equation}
E\langle R_{1,n}\rangle(\tau)
 \le
\frac{\|h_\ell\|_\infty^2}{n}
\int_0^{\tau}
E\left[
\frac{\{q_j(u)-\bar Y_j(u)\}^2}{\bar Y_j(u)}\1\{\bar Y_j(u)>0\}
\right]
dA_{\ell,0}(u).\label{eq:weighted-na-bracket-bound}
\end{equation}
Let $K\sim\mathrm{Binomial}(n,q)$ and $\mu=nq$. Then
$E\left[ \frac{\{q-K/n\}^2}{K/n}\1\{K>0\} \right] = \frac1nE\left[\frac{(K-\mu)^2}{K}\1\{K>0\}\right] \le \frac{C}{n},$
where the last inequality follows by splitting on $\{K\ge \mu/2\}$, applying
Chernoff's bound on $\{K<\mu/2\}$, and handling $\mu<2$ using $(K-\mu)^2K^{-1}\1(K>0)\le2K+2\mu^2\1(K>0)$. Substitution in
\eqref{eq:weighted-na-bracket-bound} gives
$E\langle R_{1,n}\rangle(\tau) \le \frac{C\|h_\ell\|_\infty^2}{n^2}A_{\ell,0}(\tau)=O(n^{-2}).$
Doob's $L_2$ inequality implies $R_{1,n}=O_p(n^{-1})$. Combining this with $R_{2,n}=O_p(n^{-1})$
proves \eqref{eq:weighted-na-linearization}.

\end{proof}

We now prove Theorem~\ref{thm:aj-integrated-main}. Its efficiency conclusion
is discussed separately from the asymptotic linear representation and
normal limit.

\begin{proof}[Proof of Theorem~\ref{thm:aj-integrated-main}]
In this proof, $\phi^\lozenge$ abbreviates the RMST influence function
$\phi_\mu^\lozenge$ in \eqref{eq:main-nocov-rmst-eif}.
By Theorem~\ref{thm:aj-rmst-linearization},
\begin{equation}
\label{eq:aj-linearization-step}
\widehat\mu^{\AJ}-\mu_0
=\dotmu_{\tau,A_0}[\widehat A^{\AJ}-A_0]+o_p(n^{-1/2}).
\end{equation}
The derivative coefficients in \eqref{eq:rmst-pathwise-derivative}
satisfy $L_\ell(u)=q_{j(\ell)}(u)H_{\ell,0}(u)$, because
$q_j(u)=P_j(u-)G(u-)$.
The functions $H_{\ell,0}$ in
\eqref{eq:main-nocov-rmst-H}
are bounded by \eqref{eq:rmst-weight-bound}. Thus
Lemma~\ref{lem:weighted-na-linearization}, applied to each transition, gives
\begin{equation}
\label{eq:aj-weighted-na-step}
\int_0^\tau L_\ell(u)\,d\{\widehat A_\ell^{\AJ}-A_{\ell,0}\}(u)
=\int_0^\tau H_{\ell,0}(u)\,d\bar M_{\ell,0}(u)+O_p(n^{-1}).
\end{equation}
Summing over the three transitions in \eqref{eq:aj-linearization-step}
therefore yields
\begin{equation}
\label{eq:aj-influence-expansion}
\widehat\mu^{\AJ}-\mu_0
=\sum_{\ell\in\cJ}\int_0^\tau H_{\ell,0}(u)\,d\bar M_{\ell,0}(u)
 +o_p(n^{-1/2}).
\end{equation}
The leading term equals $\Pn\phi^\lozenge$ since
$\bar M_{\ell,0}=n^{-1}\sum_iM_{\ell,0,i}$. Its expectation is zero by
boundedness and predictability of $H_{\ell,0}$ and square integrability of
$M_{\ell,0}$. This proves the asymptotic linear representation in Theorem~\ref{thm:aj-integrated-main}.

For the normal limit, suppose
$0<\sigma_0^2=\Var\{\phi^\lozenge(O)\}<\infty$ and define

\[
Z_n(t)=n^{-1/2}\sum_{i=1}^n\sum_{\ell\in\cJ}
\int_0^tH_{\ell,0}(u)\,dM_{\ell,0,i}(u),\qquad 0\le t\le\tau.
\]

This is a square-integrable martingale in the pooled subject filtration, with
\begin{equation}
\label{eq:aj-martingale-terminal}
Z_n(\tau)=\sqrt n(\Pn-P)\phi^\lozenge(O).
\end{equation}
Lemma~\ref{lem:aj-jump-regularity} gives continuous
compensators and excludes simultaneous transition jumps. Independence
across subjects and the resulting zero cross-variations imply
$\langle Z_n\rangle(t)=n^{-1}\sum_iR_i(t)$, where
$R_i(t)=\sum_{\ell\in\cJ}\int_0^tH_{\ell,0}(u)^2Y_{j(\ell),i}(u)\,dA_{\ell,0}(u)$.
The variables $R_i(\tau)$ are i.i.d. and integrable because the coefficients
are bounded and the hazard masses are finite. The law of large numbers and
the single-subject martingale isometry give
\begin{equation}
\label{eq:aj-bracket-limit}
\langle Z_n\rangle(\tau)\longrightarrow_p ER_1(\tau)
=\Var\{\phi^\lozenge(O)\}=\sigma_0^2.
\end{equation}
Moreover, the same lemma gives
$J_n^{\max}:=\max_{0<u\le\tau}\sum_{i,\ell}\Delta N_{\ell,i}(u)\le1$
almost surely. Continuity of the compensators therefore yields
\begin{equation}
\label{eq:aj-jump-bound}
\sup_{0<u\le\tau}|\Delta Z_n(u)|
\le H_{\max}J_n^{\max}/\sqrt n
\le H_{\max}/\sqrt n\longrightarrow0
\qquad\text{almost surely}.
\end{equation}
Rebolledo's martingale central limit theorem now gives
$Z_n(\tau)\rightsquigarrow N(0,\sigma_0^2)$.
Together with \eqref{eq:aj-martingale-terminal}, the asymptotic linear representation
just established, and Slutsky's theorem, this proves the normal limit in
Theorem~\ref{thm:aj-integrated-main}.

If the separate semiparametric calculation identifies $\phi^\lozenge$ as
the canonical gradient of the nonparametric censored illness--death model,
the efficiency bound is $P(\phi^\lozenge)^2=\sigma_0^2$, using the centering
established above. The efficiency conclusion is conditional on that
canonical-gradient identification.
\end{proof}

\subsection{Efficiency gain}\label{app:efficiency-gain}
We verify the efficiency gain from $R$ in \eqref{eq:gain} in Section~\ref{subsec:gain-snapshot} using the
influence functions in Theorem~\ref{thm:eif-main}.
Use the stopped representations \cref{eq:IF-circ-two,eq:IF-lozenge-two}.
To compare them, decompose $\phi^\circ$ using \cref{eq:IF-circ-two} and
\cref{eq:IF-lozenge-two}.
Under censoring ignorability in \Cref{ass:C-ignore},
$M^\bullet_C\equiv M^\blacklozenge_C$ is a martingale in both $\mathcal{F}^\blacklozenge$ and $\mathcal{F}^\bullet$.
Thus
$$
\phi^\circ(\tilde T,\Delta)
= \phi^\lozenge(\tilde T,\Delta,\tilde R,\Gamma)
+ \int_0^\tau I(\tilde T>t) \big\{ M^\bullet_{\theta}(t-)-M^\blacklozenge_\theta(t-)\big\} \;\frac{\mathrm{d} M_C^\blacklozenge(t)}{G(t)} .
$$
This decomposition is orthogonal because $\langle M^\blacklozenge_\theta,M^\blacklozenge_C \rangle\equiv0$ in $\mathcal{F}^\blacklozenge$ under
independence of $C$ and $T,R$.
Allowing $C$ and $T$ to be dynamically confounded through $I(R\ge t)$ would invalidate this orthogonality and the consistency of the marginal estimator.

Consequently, $\mathbb{V}\{\phi^\lozenge - \phi^\circ\}=\mathbb{V}\{\phi^\circ\}-\mathbb{V}\{\phi^\lozenge\}$.
The martingale isometry gives
\begin{align*}
  \mathbb{V}\big\{\phi^\lozenge - \phi^\circ\big\}
  &=
  E \int_0^\tau Y^*_T(t)Y^*_C(t)
  \big\{ [M^\bullet_{\theta}(t-)-M^\blacklozenge_\theta(t-) \big\}^2
  \frac{\lambda_C(t)}{G(t)^2}
  \;\mathrm{d} t \\
  &=
  \int_0^\tau E \big[ I(T\ge t) \big\{ M^\bullet_{\theta}(t-)-M^\blacklozenge_\theta(t-)\big\}^2 \big]
  \frac{\lambda_C(t)}{G(t)}
  \;\mathrm{d} t
\end{align*}
The factorization required in this calculation concerns
$
E\big[Y^*_T(t)Y^*_C(t) \big\{M^\bullet_\theta(t^-)-M^\blacklozenge_{\theta}(t^-)\big\}^2 \big].
$
In the no-covariate model considered here, Assumption~\ref{ass:C-ignore}
supplies the full independence $T,R\indep C$ needed for this factorization.
The calculation does not assert validity under time-pointwise independence.
Conditionally on $T\ge t$, $M_\theta^\bullet(t-)$ is the conditional mean of
$M_\theta^\blacklozenge(t-)$, and is constant on that event.
Thus the expectation in the preceding integral equals
$S(t)\mathbb V\{M_\theta^\blacklozenge(t-)-M_\theta^\bullet(t-)\mid T\ge t\}$.
The efficiency gain is therefore
\[
  \mathbb{V}\big\{\phi^\circ-\phi^\lozenge\big\}
   =
  \int_0^\tau
  \frac{1}{G(t)} \cdot \lambda_C(t)S_T(t) \cdot \mathbb{V}\big\{M^\blacklozenge_\theta(t^-)-M^\bullet_{\theta}(t^-)|T\ge t\big\}
  \;\mathrm{d} t .
\]

\section{Proofs for covariate-adjusted estimation}
\label{app:covariate-proofs}
The product-integral and norm conventions are those of
Section~\ref{sec:setup-notation}.

We use the single-split construction of
\eqref{eq:os-est-integrated-main}: $\calD_{\mathrm{tr},n}$ is the training
sample and $O_1^{\mathrm{ev}},\ldots,O_m^{\mathrm{ev}}$ are the independent
evaluation observations. The notation $\Pev$ abbreviates the main-text
$\mathbb P_{{\rm ev},m}$. In a cross-fitting application these samples are,
respectively, the complement of $I_k$ and $I_k$, as in
\eqref{eq:crossfit-plugin-os-main}.
We write $\widehat S_n=S_{\widehat A_n}$ and
$\widehat\mu_n=\mu_{\widehat A_n}$ for the fitted quantities already defined
in Section~\ref{sec:covariates}; their plug-in average is the first estimator
in \eqref{eq:os-est-integrated-main}.
The symbol $\Ptr$ denotes the empirical average over the training sample.

Using the derivative of Lemma~\ref{lem:rmst-hazard-derivative}, the pointwise and standardized derivative terms are
\begin{align*}
\dotmu_{\tau,A_0}[\widehat A_n-A_0](x)
  &:=
  \dotmu_{\tau,A_0(\cdot\mid x)} [\widehat A_n(\cdot\mid x)-A_0(\cdot\mid x)], \\
\dotmu_{\tau,A_0}^{\mathrm{std}}[\widehat A_n-A_0]
  &:=
  P_X\{\dotmu_{\tau,A_0}[\widehat A_n-A_0](X)\}.
\end{align*}

For the independent evaluation sample, let $h$ be a training-sample-measurable
function of $X$ with $P_Xh^2<\infty$. Conditional independence gives
$E\{(\Pev-P_X)h\mid\calD_{\mathrm{tr},n}\}=0$ and
\begin{equation}
\label{eq:conditional-evaluation-variance}
E\!\left[\left.\{(\Pev-P_X)h\}^{2}\right|\calD_{\mathrm{tr},n}\right]
=m^{-1}\{P_Xh^2-(P_Xh)^2\}\le m^{-1}P_Xh^2.
\end{equation}
The same argument applies conditionally on each fold's training sample.
This elementary calculation is used in both subsections below.

\subsection{Cox--Breslow estimation}
\label{sec:cox-breslow-details}

We use the transition model \eqref{eq:cox-main-transition-model}, its
source-risk function $s_{\ell,0}$, and
Assumption~\ref{ass:cox-main-primitives}.
The fitted hazards displayed in Section~\ref{sec:covariates} are the raw
Cox--Breslow hazards. Below these are denoted by $\bar A_n$, while
$\widehat A_n$ denotes their row-stabilized version.

Let $\Theta_\ell=\{\beta\in\mathbb R^{d_\ell}:\|\beta\|\le R_\ell\}$.  For
$\beta\in\Theta_\ell$, define only the empirical denominator needed for the
partial likelihood,
$S_{\ell,n}^{(0)}(\beta,u)=n^{-1}\sum_{i=1}^n Y_{j,i}(u)e^{\beta^\top X_i}$, $J_{\ell,n}(u)=\1\{S_{\ell,n}^{(0)}(\beta,u)>0\}$.
Since the exponential weights are strictly positive, $J_{\ell,n}$ does not depend
on $\beta$.  The transition-specific average partial log-likelihood is
$\ell_{\ell,n}(\beta) = n^{-1}\sum_{i=1}^n\int_0^{\tau} \{\beta^\top X_i-\log S_{\ell,n}^{(0)}(\beta,u)\} dN_{\ell,i}(u),$
where the integrand is evaluated only at event times.  Its negative
$Q_{\ell,n}(\beta):=-\ell_{\ell,n}(\beta), \qquad \beta\in\Theta_\ell,$
is convex.  Define
\begin{equation}
\label{eq:cox-beta-hat-def}
    \widehat\beta_\ell\in\arg\min_{\beta\in\Theta_\ell}Q_{\ell,n}(\beta),
\end{equation}
with a fixed deterministic measurable tie-breaking rule.  Lemma~\ref{lem:cox-root-consistency}
shows that, with probability tending to one, this minimizer is unique, interior,
solves the Cox score equation, and is $n^{-1/2}$-consistent.

The predictable-support Breslow estimator is
$\widehat\Lambda_{0,\ell}(t)= \int_0^t \frac{J_{\ell,n}(u)} {S_{\ell,n}^{(0)}(\widehat\beta_\ell,u)} \,\Ptr dN_\ell(u).$
For covariate value $x$, set
$d\bar A_{\ell,n}(u\mid x) =e^{\widehat\beta_\ell^\top x}\,d\widehat\Lambda_{0,\ell}(u)$, $dA_{\ell,0}(u\mid x) =e^{\beta_{\ell,0}^\top x}\,d\Lambda_{0,\ell,0}(u)$.
The raw empirical cumulative hazards $\bar A_n$ are made row-admissible by the
following deterministic stabilization.  For row zero, write
$x_{0k}(u,x)=\Delta\bar A_{0k,n}(u\mid x)$, $k=1,2$, and
$x_0=x_{01}+x_{02}$.  If $x_0>0$, set
$\Delta\widehat A_{0k,n}(u\mid x) =\frac{x_{0k}(u,x)}{x_0(u,x)}\{1-e^{-x_0(u,x)}\}, \qquad k=1,2,$
and set both row-zero increments equal to zero if $x_0=0$.  For row one, set
$\Delta\widehat A_{12,n}(u\mid x) =1-\exp\{-\Delta\bar A_{12,n}(u\mid x)\}.$
Between empirical jumps $\widehat A_n$ is constant.

For the Cox proofs, we use the following proof-local notation.  For $r=0,1,2,3$,
$S_{\ell,n}^{(r)}(\beta,u)=n^{-1}\sum_{i=1}^nY_{j(\ell),i}(u)e^{\beta^\top X_i}X_i^{\otimes r}$, $s_\ell^{(r)}(\beta,u)=E\{Y_{j(\ell)}(u)e^{\beta^\top X}X^{\otimes r}\}$.
Let
$j_{\ell,0}(u)=\1\{s_{\ell,0}(u)>0\}$, $\bar Z_{\ell,n}(\beta,u)= \frac{J_{\ell,n}(u)S_{\ell,n}^{(1)}(\beta,u)}{S_{\ell,n}^{(0)}(\beta,u)}$, $\bar z_\ell(\beta,u)= \frac{j_{\ell,0}(u)s_\ell^{(1)}(\beta,u)}{s_\ell^{(0)}(\beta,u)}$,
with ratios equal to zero when the denominator is zero, and
\[
V_{\ell,n}(\beta,u)=
\frac{J_{\ell,n}(u)S_{\ell,n}^{(2)}(\beta,u)}{S_{\ell,n}^{(0)}(\beta,u)}
-\bar Z_{\ell,n}(\beta,u)^{\otimes2},
\qquad
V_\ell(\beta,u)=
\frac{j_{\ell,0}(u)s_\ell^{(2)}(\beta,u)}{s_\ell^{(0)}(\beta,u)}
-\bar z_\ell(\beta,u)^{\otimes2}.
\]
The empirical and population scores are
$U_{\ell,n}(\beta)=n^{-1}\sum_{i=1}^n\int_0^{\tau} \{X_i-\bar Z_{\ell,n}(\beta,u)\}\,dN_{\ell,i}(u)$, $U_\ell(\beta)=E\int_0^{\tau}\{X-\bar z_\ell(\beta,u)\}\,dN_\ell(u)$,
and
\[
\dot U_{\ell,n}(\beta)=-\Ptr\int_0^{\tau}V_{\ell,n}(\beta,u)\,dN_\ell(u),
\qquad
\dot U_\ell(\beta)=-P\int_0^{\tau}V_\ell(\beta,u)\,dN_\ell(u).
\]
Finally, define
$\mathcal U_{\ell,n}:= \sup_{\beta\in\mathcal B_\ell}\|U_{\ell,n}(\beta)-U_\ell(\beta)\|$, $\mathcal D_{\ell,n}:= \sup_{\beta\in\mathcal B_\ell} \|\dot U_{\ell,n}(\beta)-\dot U_\ell(\beta)\|$,
where $\mathcal B_\ell$ denotes the local neighborhood introduced in the proofs.

\begin{definition}
\label{def:cox-martingale}
For a fixed transition $\ell=(j,k)$, define the subject-level Cox martingale
$M_{\ell,0,i}^{\beta}(t) :=N_{\ell,i}(t) -\int_0^tY_{j,i}(u)e^{\beta_{\ell,0}^\top X_i} d\Lambda_{0,\ell,0}(u)$
at the true parameter.
Its training average is
$\bar M_{\ell,0}^{\beta}(t) :=\Ptr M_{\ell,0}^{\beta}(t) =n^{-1}\sum_{i=1}^nM_{\ell,0,i}^{\beta}(t),$
and hence, in differential notation,
\begin{equation}
\label{eq:cox-average-martingale}
    d\bar M_{\ell,0}^{\beta}(u)
    :=\Ptr dN_\ell(u)-S_{\ell,n}^{(0)}(\beta_{\ell,0},u)
      d\Lambda_{0,\ell,0}(u).
\end{equation}
Here $\Ptr dN_\ell(u):=n^{-1}\sum_i dN_{\ell,i}(u)$.
\end{definition}

\begin{lemma}
\label{lem:cox-no-ties}
Under Assumptions~\ref{ass:time-reg}, \ref{ass:C-ignore}, \ref{ass:Markov-state}, and~\ref{ass:cox-main-primitives}, for each
transition $\ell\in\cJ$,
\[
\mathcal T_{\ell,n}:=
\max_{0<u\le\tau}n^{-1}\sum_{i=1}^n\Delta N_{\ell,i}(u)
\le n^{-1}
\qquad\textnormal{almost surely}.
\]
Moreover the baseline cumulative hazard is finite and continuous:
$\Lambda_{0,\ell,0}(\tau)<\infty$, and $\Delta\Lambda_{0,\ell,0}(u)=0$ for $0<u\le\tau$.
\end{lemma}

\begin{proof} [Proof of Lemma \ref{lem:cox-no-ties}]
The path-level no-ties argument in the proof of
Lemma~\ref{lem:aj-jump-regularity} applies under the same
standing continuity and independent-sampling assumptions, conditionally on
the covariates and then after taking expectations. Censoring can remove
but cannot create a transition jump. Thus at most one subject contributes
a transition-$\ell$ jump at a calendar time, and
$\mathcal T_{\ell,n}\le n^{-1}$ almost surely.

To establish continuity of the baseline hazard, fix $u\in(0,\tau]$. The time
$S_\ell$ has no conditional atom and hence no unconditional atom, giving
$\Pr\{\Delta N_\ell(u)=1\}=0$. The Cox model implies
$E\{\Delta N_\ell(u)\} =E\{Y_j(u)e^{\beta_{\ell,0}^{\top}X}\}\Delta\Lambda_{0,\ell,0}(u) =s_{\ell,0}(u)\Delta\Lambda_{0,\ell,0}(u).$
When $s_{\ell,0}(u)>0$, this gives $\Delta\Lambda_{0,\ell,0}(u)=0$. When $s_{\ell,0}(u)=0$, the source-support
convention \eqref{eq:cox-main-source-support} gives
$\Delta\Lambda_{0,\ell,0}(u) \le \int_{(0,\tau]}\mathbf 1\{s_{\ell,0}(v)=0\} d\Lambda_{0,\ell,0}(v)=0 .$
Thus $\Lambda_{0,\ell,0}$ is continuous. Bounded covariates also imply
$0\le s_{\ell,0}(u) =E\{Y_j(u)e^{\beta_{\ell,0}^{\top}X}\} \le e^{\|\beta_{\ell,0}\|K}=:C_\ell .$
The source-support convention and the finite inverse-risk condition now give
\[
\Lambda_{0,\ell,0}(\tau)
=\int_{(0,\tau]}\mathbf 1\{s_{\ell,0}(u)>0\}
  d\Lambda_{0,\ell,0}(u)
\le
C_\ell
\int_{(0,\tau]}
\frac{\mathbf 1\{s_{\ell,0}(u)>0\}}{s_{\ell,0}(u)}
  d\Lambda_{0,\ell,0}(u)<\infty .
\]
as required.
\end{proof}

\begin{lemma}
\label{lem:cox-risk-regularity}
Fix a transition $\ell=(j,k)$, and let
$s_\ell^{(r)}(\beta,u)=E\{Y_j(u)e^{\beta^\top X}X^{\otimes r}\}, \qquad r=0,1,2.$
Let $\mathcal B_\ell$ be a bounded neighborhood of $\beta_{\ell,0}$, and set
$R_\ell=\sup_{\beta\in\mathcal B_\ell}\|\beta\|$.    If, in addition,
$P\{Y_j(u_m)\ne Y_j(u)\}\to0$ whenever $u_m\to u$, then the risk functions are
also continuous in $u$ .
\begin{enumerate}[label=\textnormal{(\alph*)}]
\item If $\|X\|\le K$ a.s., then for $r=0,1,2$,
\begin{equation}
\label{eq:cox-risk-envelope}
    \sup_{(\beta,u)\in\mathcal B_\ell\times[0,\tau]}
    \|s_\ell^{(r)}(\beta,u)\|
    \le e^{R_\ell K}K^r <\infty,
\end{equation}
and for $r=0,1,2$,
\begin{equation}
\label{eq:cox-risk-lipschitz}
    \sup_{u\le\tau}
    \|s_\ell^{(r)}(\beta,u)-s_\ell^{(r)}(\beta',u)\|
    \le e^{R_\ell K}K^{r+1}\|\beta-\beta'\|,
    \qquad \beta,\beta'\in\mathcal B_\ell .
\end{equation}
Moreover, for every fixed $u$, $\beta\mapsto s_\ell^{(r)}(\beta,u)$ is
continuous on $\mathcal B_\ell$.  If the optional at-risk path continuity stated
above holds, then $(\beta,u)\mapsto s_\ell^{(r)}(\beta,u)$ is continuous on
$\mathcal B_\ell\times[0,\tau]$.

\item If instead the Cox dimension is fixed and $X$ is sub-Gaussian, then the same
conclusions hold with deterministic finite constants
\begin{equation}
\label{eq:cox-risk-subgaussian-envelope}
    C_{r,\ell}:=E\{e^{R_\ell\|X\|}\|X\|^r\}<\infty,
    \qquad
    L_{r,\ell}:=E\{e^{R_\ell\|X\|}\|X\|^{r+1}\}<\infty,
\end{equation}
in place of $e^{R_\ell K}K^r$ and $e^{R_\ell K}K^{r+1}$, respectively.
That is,
$
\sup_{(\beta,u)}\|s_\ell^{(r)}(\beta,u)\|\le C_{r,\ell} ,
$
and
$
\sup_u\|s_\ell^{(r)}(\beta,u)-s_\ell^{(r)}(\beta',u)\| \le L_{r,\ell}\|\beta-\beta'\| .
$
The fixed-$u$ $\beta$-continuity conclusion again holds, and the joint
$(\beta,u)$-continuity conclusion holds under the optional at-risk path
continuity stated above.
\end{enumerate}

\end{lemma}

\begin{proof} [Proof of Lemma \ref{lem:cox-risk-regularity}]
Suppose first that $\|X\|\le K$. The bounds $0\le Y_j(u)\le1$ and $\|\beta\|\le R_\ell$ on $\mathcal B_\ell$
imply
$\|Y_j(u)e^{\beta^\top X}X^{\otimes r}\| \le e^{R_\ell K}K^r.$
Taking expectations proves \eqref{eq:cox-risk-envelope}.
For $\beta,\beta'\in\mathcal B_\ell$, the mean-value theorem gives
$|e^{\beta^\top X}-e^{\beta'^\top X}| \le e^{R_\ell K}|X^\top(\beta-\beta')| \le e^{R_\ell K}K\|\beta-\beta'\|.$
Multiplication by $Y_j(u)\|X\|^r$, followed by expectation and the bound $\|X\|\le K$,
proves \eqref{eq:cox-risk-lipschitz}.

At fixed $u$, convergence $\beta_m\to\beta$ implies
$Y_j(u)e^{\beta_m^\top X}X^{\otimes r} \to Y_j(u)e^{\beta^\top X}X^{\otimes r}$
almost surely. Moreover, $e^{R_\ell K}K^r$ permits dominated convergence, giving
continuity at fixed $u$ in $\beta$. Under the optional at-risk path
continuity condition, let $(\beta_m,u_m)\to(\beta,u)$. Decomposing the difference into a term
containing $Y_j(u_m)-Y_j(u)$ and a term containing $e^{\beta_m^\top X}-e^{\beta^\top X}$ gives joint continuity
in $(\beta,u)$ by the same argument.

In the sub-Gaussian case, fixed dimension implies a Gaussian tail for
$\|X\|$. For sufficiently small $\eta>0$, we therefore have $E\exp\{\eta\|X\|^2\}<\infty$.
The inequality $R_\ell\|X\|\le \eta\|X\|^2+R_\ell^2/(4\eta)$ implies $E\{e^{R_\ell\|X\|}\|X\|^m\}<\infty$ for every fixed integer $m$,
so the constants in \eqref{eq:cox-risk-subgaussian-envelope} are finite.
The  bounds follow respectively from
$\|Y_j(u)e^{\beta^\top X}X^{\otimes r}\| \le e^{R_\ell\|X\|}\|X\|^r$
and
$|e^{\beta^\top X}-e^{\beta'^\top X}|\, \|X\|^r \le e^{R_\ell\|X\|}\|X\|^{r+1}\|\beta-\beta'\|.$
At fixed $u$, continuity in $\beta$ follows by dominated convergence
with  $e^{R_\ell\|X\|}\|X\|^r$. Under the optional at-risk path continuity condition,
the preceding decomposition also gives joint continuity in $(\beta,u)$:
its first term converges to zero in probability and is uniformly integrable.
\end{proof}

\begin{lemma}
\label{lem:cox-population-score}
Fix a transition $\ell=(j,k)$.  Suppose the correct Cox model
\eqref{eq:cox-main-transition-model} holds, the baseline cumulative hazard is finite, and the
risk functions have the boundedness, measurability, and fixed-$u$
$\beta$-continuity properties proved in
Lemma~\ref{lem:cox-risk-regularity}.  Suppose also that the minimal information
condition Assumption~\ref{ass:cox-main-primitives}\ref{cxmain:information} holds.  Then:
\begin{enumerate}[label=\textnormal{(\roman*)}]
\item the population score is centered at the truth,
$U_\ell(\beta_{\ell,0})=0;$
\item the population score is differentiable on a possibly smaller open convex
neighborhood $\mathcal B_\ell$ of $\beta_{\ell,0}$, and
\begin{equation}
\label{eq:cox-population-score-derivative}
    \dot U_\ell(\beta)
    =-
    \int_{(0,\tau]}
    V_\ell(\beta,u)\,s_{\ell,0}(u)
    \,d\Lambda_{0,\ell,0}(u);
\end{equation}
\item consequently,
$-\dot U_\ell(\beta_{\ell,0})=I_\ell,$
where $I_\ell$ is the matrix in Assumption~\ref{ass:cox-main-primitives}\ref{cxmain:information};
\item  and
\begin{equation}
\label{eq:cox-local-information-bound}
    \sup_{\beta\in\mathcal B_\ell}
    \|\dot U_\ell(\beta)-\dot U_\ell(\beta_{\ell,0})\|
    \le \iota_\ell/4 .
\end{equation}
\end{enumerate}
\end{lemma}

\begin{proof} [Proof of Lemma \ref{lem:cox-population-score}]
Write $\beta_0=\beta_{\ell,0}$, $\Lambda_0=\Lambda_{0,\ell,0}$, $s^{(r)}=s_\ell^{(r)}$, $s_0=s_{\ell,0}$, $\bar x=\bar z_\ell$, and $V=V_\ell$.
Correct specification of the Cox model gives
$E\{dN_\ell(u)\mid X,\mathcal F_{u-}\} =Y_j(u)e^{\beta_0^\top X}\,d\Lambda_0(u).$
Fubini's theorem and finiteness of the hazard mass yield
\begin{align}
U_\ell(\beta)
&=E\int_0^{\tau}\{X-\bar x(\beta,u)\}\,dN_\ell(u) \notag \\
&=\int_{(0,\tau]}
   E\big[\{X-\bar x(\beta,u)\}Y_j(u)e^{\beta_0^\top X}\big]
   \,d\Lambda_0(u) \notag \\
&=\int_{(0,\tau]}
   \{s^{(1)}(\beta_0,u)-\bar x(\beta,u)s^{(0)}(\beta_0,u)\}
   \,d\Lambda_0(u).                         \label{eq:cox-population-score}
\end{align}
On $s_0(u)>0$, $\bar x(\beta_0,u)=s^{(1)}(\beta_0,u)/s^{(0)}(\beta_0,u)$; on $s_0(u)=0$, the support convention
\eqref{eq:cox-main-source-support} excludes $d\Lambda_0$-mass.
It follows from \eqref{eq:cox-population-score} that $U_\ell(\beta_0)=0$.

For fixed $u$, differentiation of $\bar x(\beta,u)=s^{(1)}(\beta,u)/s^{(0)}(\beta,u)$ on the source support gives
the risk-set covariance
$\partial_\beta \bar x(\beta,u) =\frac{s^{(2)}(\beta,u)}{s^{(0)}(\beta,u)} -\bar x(\beta,u)^{\otimes2} =V(\beta,u),$
with both sides set to zero off the support. Since $\|X\|\le K$ in the main
Cox theorem, $\|V(\beta,u)\|\le 2K^2$. Dominated convergence therefore justifies
differentiation under the integral in \eqref{eq:cox-population-score},
proving \eqref{eq:cox-population-score-derivative}. Evaluation at $\beta_0$ gives
$-\dot U_\ell(\beta_0) =\int_{(0,\tau]}V(\beta_0,u)s_0(u)\,d\Lambda_0(u) =I_\ell,$
as required.

For local invertibility, Lemma~\ref{lem:cox-risk-regularity} gives
boundedness and continuity at fixed $u$ in $\beta$ of $s^{(r)}(\beta,u)$ for
$r=0,1,2$. On $s_0(u)>0$, the map
$V(\beta,u) =\frac{s^{(2)}(\beta,u)}{s^{(0)}(\beta,u)} -\bar x(\beta,u)^{\otimes2}$
is continuous in $\beta$ at fixed $u$. Off this support, the
integrand in \eqref{eq:cox-population-score-derivative} is zero, and
\eqref{eq:cox-main-source-support} excludes $d\Lambda_0$-mass.
The bound $\|X\|\le K$ implies $\|V(\beta,u)\|\le2K^2$ uniformly in $(\beta,u)$.
For $\beta_m\to\beta_0$, dominated convergence thus gives
\[
\begin{aligned}
\|\dot U_\ell(\beta_m)-\dot U_\ell(\beta_0)\|
&\le
\int_{(0,\tau]}
\|V(\beta_m,u)-V(\beta_0,u)\|s_0(u)\,d\Lambda_0(u)\\
&\longrightarrow0,
\end{aligned}
\]
because $s_0(u)\le E e^{\beta_0^\top X}\le e^{\|\beta_0\|K}$ and $\Lambda_0(\tau)<\infty$. There is consequently $\delta_\ell>0$ such that
$\sup_{\|\beta-\beta_0\|\le\delta_\ell} \|\dot U_\ell(\beta)-\dot U_\ell(\beta_0)\| \le \iota_\ell/4 .$
Intersecting the original neighborhood with $\mathcal B_\ell\cap\{\beta:\|\beta-\beta_0\|<\delta_\ell\}$ and, if necessary,
taking a smaller open convex ball about $\beta_0$ proves
\eqref{eq:cox-local-information-bound}. The required local invertibility
therefore follows from continuity of the population derivative and the
information lower bound.
\end{proof}

\begin{lemma}
\label{lem:cox-denominator-bounds}
Let $W_i(u)=Y_{j,i}(u)e^{\beta_0^\top X_i}$,
$S_n(u)=n^{-1}\sum_iW_i(u)$, $s(u)=E\{W_i(u)\}$, and
$J_n(u)=\mathbf 1\{S_n(u)>0\}$.  Suppose $\|X\|\le K$ and
$\|\beta_0\|\le R$.  Then there is a constant $C$, depending only on
$(K,R)$, such that for all $u$ and all $n$,
\begin{align}
\label{eq:cox-denominator-zero}
s(u)\Pr\{S_n(u)=0\}&\le Cn^{-1}, \\
\label{eq:cox-denominator-ratio}
E\left[
S_n(u)\left\{\frac{s(u)J_n(u)}{S_n(u)}-1\right\}^2
\right]&\le Cn^{-1}, \\
\label{eq:cox-denominator-inverse}
E\left[\frac{s(u)J_n(u)}{S_n(u)}\right]&\le C, \\
\label{eq:cox-denominator-inverse-risk}
E\left[\frac{J_n(u)}{S_n(u)}\right]&\le
C\,\frac{\mathbf 1\{s(u)>0\}}{s(u)}, \\
\label{eq:cox-denominator-weighted-ratio}
E\left[
s(u)\left\{\frac{s(u)J_n(u)}{S_n(u)}-1\right\}^2
\right]&\le Cn^{-1} .
\end{align}
Moreover, for each integer $m\ge1$ there is a constant $C_m$, depending
only on $(m,K,R)$, such that
\begin{align}
\label{eq:cox-denominator-moment-ratio}
E\left[\left\{\frac{s(u)J_n(u)}{S_n(u)}\right\}^m\right]&\le C_m, \\
\label{eq:cox-denominator-moment-weighted}
E\left[\frac{s(u)^mJ_n(u)}{S_n(u)^{m-1}}\right]&\le C_m s(u), \\
\label{eq:cox-denominator-moment-inverse}
E\left[\frac{s(u)^mJ_n(u)}{S_n(u)^{m}}\right]&\le C_m\,\mathbf 1\{s(u)>0\}.
\end{align}
\end{lemma}

\begin{proof} [Proof of Lemma \ref{lem:cox-denominator-bounds}]
Let $K_n(u)=\sum_iY_{j,i}(u)$ and $q(u)=P\{Y_j(u)=1\}$.
Bounded covariates and coefficients give constants $0<c<C<\infty$,
depending only on $(K,R)$, such that
\begin{equation}
\label{eq:cox-risk-count-comparison}
cK_n(u)/n\le S_n(u)\le CK_n(u)/n,
\qquad cq(u)\le s(u)\le Cq(u).
\end{equation}
Hence $J_n=\mathbf1\{K_n>0\}$ and $K_n\sim\mathrm{Binomial}(n,q)$.
The case $q=0$ is immediate under the zero-denominator convention.
For $q>0$, $q(1-q)^n\le(en)^{-1}$ proves
\eqref{eq:cox-denominator-zero}.

Suppress $u$ and use the three regimes
$A_1=\{K_n\ge nq/2\}$,
$A_2=\{0<K_n<nq/2,\ nq\ge2\}$, and
$A_3=\{K_n>0,\ nq<2\}$.
The last regime is used when $nq<2$, and the first two when $nq\ge2$.
On $A_1$, \eqref{eq:cox-risk-count-comparison} and
$E(S_n-s)^2\le Cq/n$ give
\[
E\left\{\frac{(S_n-s)^2}{S_n}\mathbf1_{A_1}\right\}
+E\left\{s\frac{(S_n-s)^2}{S_n^2}\mathbf1_{A_1}\right\}
\le Cn^{-1}.
\]
On $A_2$, the two integrands are at most $Cnq^2$ and $Cn^2q^3$,
respectively, while Chernoff's inequality gives $P(A_2)\le e^{-cnq}$.
Their expectations are therefore at most
$Cn^{-1}(nq)^2e^{-cnq}$ and $Cn^{-1}(nq)^3e^{-cnq}$, both bounded by
$Cn^{-1}$. On $A_3$, the elementary inequalities
$(a-b)^2/a\le2a+2b^2/a$ and $(a-b)^2/a^2\le2+2b^2/a^2$
give, respectively,
$\frac{(S_n-s)^2}{S_n}\le C\{K_n/n+nq^2\}$, $s\frac{(S_n-s)^2}{S_n^2}\le C\{q+n^2q^3\}\le Cq$.
Taking expectations and using $nq<2$ bounds both contributions by $Cn^{-1}$.
The first integrand is zero when $S_n=0$; the second ratio expression in
\eqref{eq:cox-denominator-weighted-ratio} then contributes
$sP(S_n=0)\le Cn^{-1}$. This proves
\eqref{eq:cox-denominator-ratio} and \eqref{eq:cox-denominator-weighted-ratio}.

The inverse and higher moments use the same regimes. For a fixed integer
$m\ge1$, on $A_1$ the ratio $nq/K_n$ is at most two. On $A_2$, its
$m$th moment is at most $(nq)^me^{-cnq}$, up to a fixed constant.
On $A_3$, $nq/K_n<2$. Consequently,
\begin{equation}
\label{eq:binomial-ratio-moments}
E\left\{(nq/K_n)^m\mathbf1(K_n>0)\right\}\le C_m,
\qquad
E\left\{\frac{q^m\mathbf1(K_n>0)}{(K_n/n)^{m-1}}\right\}\le C_mq.
\end{equation}
For the second inequality, use
$q^m/(K_n/n)^{m-1}=q(nq/K_n)^{m-1}$ on $K_n>0$;
the same reasoning also gives
$E\{q^m\mathbf1(K_n>0)/(K_n/n)^m\}\le C_m\mathbf1(q>0)$.
Transferring these bounds through \eqref{eq:cox-risk-count-comparison} proves
\eqref{eq:cox-denominator-moment-ratio}--\eqref{eq:cox-denominator-moment-inverse}.
The case $m=1$ yields \eqref{eq:cox-denominator-inverse}; division by
$s$ on its positive support yields \eqref{eq:cox-denominator-inverse-risk}.
\end{proof}

\begin{lemma}
\label{lem:cox-self-normalized-ratios}
Fix a transition $\ell=(j,k)$.  Work under Assumption~\ref{ass:cox-main-primitives}.
Let $\mathcal B_\ell$ be a fixed compact neighborhood of $\beta_{\ell,0}$.
Let $g_\beta(u,x)$ be a scalar or fixed-dimensional vector family satisfying
$\sup_{\beta,u,x}\|g_\beta(u,x)\|\le C_g$, $\sup_{u,x}\|g_\beta(u,x)-g_{\beta'}(u,x)\| \le L_g\|\beta-\beta'\|$,
and assume the corresponding class is bounded Euclidean.  Define
$G_n(\beta,u)=\Ptr\{Y_j(u)e^{\beta^\top X}g_\beta(u,X)\}$, $G(\beta,u)=P\{Y_j(u)e^{\beta^\top X}g_\beta(u,X)\}$,
$\bar g_n(\beta,u)= \frac{J_{\ell,n}(u)G_n(\beta,u)}{S_{\ell,n}^{(0)}(\beta,u)}$, $\bar g(\beta,u)= \frac{\mathbf 1\{s_\ell^{(0)}(\beta,u)>0\}G(\beta,u)} {s_\ell^{(0)}(\beta,u)}$,
with ratios equal to zero when their denominators are zero.  Then, for every fixed
$\beta\in\mathcal B_\ell$,
\begin{align}
\label{eq:selfnorm-compensator-rate}
\int_0^{\tau}
\|\bar g_n(\beta,u)-\bar g(\beta,u)\|
S_{\ell,n}^{(0)}(\beta_{\ell,0},u)d\Lambda_{0,\ell,0}(u)
&=O_p(n^{-1/2}), \\
\label{eq:selfnorm-counting-rate}
\Ptr\int_0^{\tau}
\|\bar g_n(\beta,u)-\bar g(\beta,u)\|dN_\ell(u)
&=O_p(n^{-1/2}).
\end{align}
Moreover, the corresponding suprema over $\beta\in\mathcal B_\ell$ are
$o_p(1)$.  In particular, the fixed-$\beta$ statement applies at
$\beta=\beta_{\ell,0}$, and the uniform $o_p(1)$ statement applies with
$g_\beta(u,x)=z_r$ and with $g_\beta(u,x)=z_rz_s$.
\end{lemma}

\begin{proof} [Proof of Lemma \ref{lem:cox-self-normalized-ratios}]
Write $\beta_0=\beta_{\ell,0}$, $S_n(\beta,u)=S_{\ell,n}^{(0)}(\beta,u)$, $s(\beta,u)=s_\ell^{(0)}(\beta,u)$, $s_0(u)=s(\beta_0,u)$, $J=J_{\ell,n}$, and $d\Lambda_0=d\Lambda_{0,\ell,0}$.
Constants may depend on the fixed compact set $\mathcal B_\ell$, but not on
$n,u$ or $\beta$. Bounded covariates give
\begin{equation}
\label{eq:selfnorm-denominator-comparison}
    cS_n(\beta_0,u)\le S_n(\beta,u)\le CS_n(\beta_0,u),
    \qquad
    cs_0(u)\le s(\beta,u)\le Cs_0(u) .
\end{equation}
and $\|G(\beta,u)\|\le C s(\beta,u)$ by boundedness of $g_\beta$.

The denominator estimates needed below hold uniformly in $\beta\in\mathcal B_\ell$:
\begin{align}
\label{eq:selfnorm-ratio-square-bound}
E\left[
S_n(\beta_0,u)
\left\{\frac{s(\beta,u)J(u)}{S_n(\beta,u)}
      -\mathbf 1\{s(\beta,u)>0\}\right\}^{2}
\right]&\le C n^{-1}, \\
\label{eq:selfnorm-ratio-l1-bound}
E\left[
S_n(\beta_0,u)
\left|\frac{s(\beta,u)J(u)}{S_n(\beta,u)}
      -\mathbf 1\{s(\beta,u)>0\}\right|
\right]&\le C n^{-1/2}s_0(u)^{1/2}.
\end{align}
The first follows from \eqref{eq:cox-denominator-ratio}, transferred
from $\beta_0$ to $\beta$ through \eqref{eq:selfnorm-denominator-comparison}.
For the second, Cauchy--Schwarz
with $R_n=s(\beta,u)J/S_n(\beta,u)-\mathbf 1\{s(\beta,u)>0\}$,
gives
$E\{S_n|R_n|\} \le \{ES_n\}^{1/2}\{E(S_nR_n^2)\}^{1/2} \le C s_0(u)^{1/2}n^{-1/2}$.

For fixed $\beta$, decompose
$\bar g_n(\beta,u)-\bar g(\beta,u)=A_n(\beta,u)+B_n(\beta,u),$
with
$A_n(\beta,u)=\frac{J}{S_n(\beta,u)}\{G_n(\beta,u)-G(\beta,u)\}$
and
$B_n(\beta,u)=G(\beta,u) \left\{ \frac{J}{S_n(\beta,u)}- \frac{\mathbf 1\{s(\beta,u)>0\}}{s(\beta,u)} \right\}.$
At fixed $\beta$, the bounded Euclidean maximal inequality yields
$\Delta_{G,n}(\beta):= \sup_{u\le\tau}\|G_n(\beta,u)-G(\beta,u)\|=O_p(n^{-1/2}).$
We first bound the compensator integral at fixed $\beta$. For $A_n$,
\eqref{eq:selfnorm-denominator-comparison} gives
\[
\int\|A_n(\beta,u)\|S_n(\beta_0,u)d\Lambda_0(u)
\le C\Delta_{G,n}(\beta)\int_0^{\tau}J(u)d\Lambda_0(u)
=O_p(n^{-1/2}),
\]
Here $\Lambda_0(\tau)<\infty$ follows from CX2 and bounded covariates: $s_0(u)\le C$ and
$\Lambda_0(\tau)\le C\int \mathbf 1\{s_0(u)>0\}s_0(u)^{-1}\,d\Lambda_0(u)<\infty$. For $B_n$, use $\|G(\beta,u)\|\le Cs(\beta,u)$ and
\eqref{eq:selfnorm-denominator-comparison} to obtain
\[
\|B_n(\beta,u)\|S_n(\beta_0,u)
\le C S_n(\beta_0,u)
\left|\frac{s(\beta,u)J(u)}{S_n(\beta,u)}
      -\mathbf 1\{s(\beta,u)>0\}\right| .
\]
Taking expectations and applying \eqref{eq:selfnorm-ratio-l1-bound}, with
$s_0\le C$ and $\Lambda_0(\tau)<\infty$, gives
\[
E\int\|B_n(\beta,u)\|S_n(\beta_0,u)d\Lambda_0(u)
\le Cn^{-1/2}\int s_0(u)^{1/2}d\Lambda_0(u)=O(n^{-1/2}).
\]
This proves \eqref{eq:selfnorm-compensator-rate}.

For the counting integral at fixed $\beta$, the decomposition
$\Ptr dN_\ell(u)=S_n(\beta_0,u)d\Lambda_0(u)+d\bar M_{\ell,0}^{\beta}(u),$
reduces the remaining argument to the martingale integral with integrand
$f_n(\beta,u)=\|\bar g_n(\beta,u)-\bar g(\beta,u)\|$. The predictable quadratic variation of $\int f_n(\beta,u)d\bar M_{\ell,0}^{\beta}(u)$ is at most
\[
\frac1n\int f_n(\beta,u)^2S_n(\beta_0,u)d\Lambda_0(u).
\]
For $A_n$, this bound is at most
$\frac{C\Delta_{G,n}(\beta)^2}{n} \int_0^{\tau}\frac{J(u)}{S_n(\beta_0,u)}d\Lambda_0(u)=O_p(n^{-2})$,
by \eqref{eq:cox-main-source-support} and
\eqref{eq:cox-denominator-inverse-risk}. For $B_n$,
\eqref{eq:selfnorm-ratio-square-bound} gives
$E\left[\frac1n\int\|B_n(\beta,u)\|^2S_n(\beta_0,u)d\Lambda_0(u)\right] \le Cn^{-2}\Lambda_0(\tau).$
Lenglart's inequality bounds the martingale contribution by $O_p(n^{-1})$.
Together with the compensator bound, this proves
\eqref{eq:selfnorm-counting-rate}.

To obtain the uniform $o_p(1)$ conclusion, we establish stochastic
equicontinuity of the self-normalized component. Differentiating the ratio
and using bounded covariates gives, for $\beta,\beta'\in\mathcal B_\ell$,
\begin{equation}
\label{eq:selfnorm-ratio-lipschitz}
\left|\frac{s(\beta,u)J}{S_n(\beta,u)}-
      \frac{s(\beta',u)J}{S_n(\beta',u)}\right|
\le Q_n(u)\|\beta-\beta'\|,
\end{equation}
where
$Q_n(u)=C\left\{1+\frac{Js_0(u)}{S_n(\beta_0,u)} +\frac{Js_0(u)^2}{S_n(\beta_0,u)^2}\right\}.$
With a larger constant and the Euclidean modulus for $g_\beta$, the same
bound controls each scalar coordinate of $\{\bar g_n(\beta,u)-\bar g(\beta,u)\}-
  \{\bar g_n(\beta',u)-\bar g(\beta',u)\}$, apart from the ordinary
empirical-process increment already considered. The moment estimates
\eqref{eq:cox-denominator-moment-ratio}--\eqref{eq:cox-denominator-moment-inverse}
imply, for every fixed integer $r\ge1$,
\begin{align}
\label{eq:selfnorm-envelope-moment}
E\left[\left\{
\int_0^{\tau}Q_n(u)^2S_n(\beta_0,u)d\Lambda_0(u)
\right\}^r\right]&\le C_r, \\
\label{eq:selfnorm-jump-envelope-moment}
E\int_0^{\tau}Q_n(u)^{2r}S_n(\beta_0,u)d\Lambda_0(u)&\le C_r.
\end{align}
Indeed, expanding the powers of $Q_n$ gives expectations of $s_0(u)^mJ/S_n(\beta_0,u)^{m-1}$
and expectations of $s_0(u)^mJ/S_n(\beta_0,u)^m$, controlled by \eqref{eq:cox-denominator-moment-weighted} and
\eqref{eq:cox-denominator-moment-inverse}.
The remaining integrals are finite because $s_0(u)\le C$ and $\Lambda_0(\tau)<\infty$, by CX2 and bounded covariates.

For the compensator term, \eqref{eq:selfnorm-ratio-lipschitz} and
\eqref{eq:selfnorm-envelope-moment}, together with the ordinary Euclidean
empirical-process modulus, give
\[
\sup_{\beta\ne\beta'}
\frac{
\int \|D_n(\beta,u)-D_n(\beta',u)\|S_n(\beta_0,u)d\Lambda_0(u)}
{\|\beta-\beta'\|}
=O_p(1),
\]
where $D_n(\beta,u)=\bar g_n(\beta,u)-\bar g(\beta,u)$. A fixed $\varepsilon$-net, followed by $n\to\infty$ and then
$\varepsilon\downarrow0$, therefore bounds the compensator supremum by $o_p(1)$.

For a scalar coordinate of the martingale term, put
$M_n(\beta)=\int_0^{\tau}D_n(\beta,u)d\bar M_{\ell,0}^{\beta}(u)$, $\Delta_n(\beta,\beta')=M_n(\beta)-M_n(\beta')$.
For any integer $r>p_\ell+2$, the purely discontinuous Burkholder--Davis--Gundy
inequality\citep{jacod1975multivariate}, the bracket bound, \eqref{eq:selfnorm-ratio-lipschitz}, and
\eqref{eq:selfnorm-envelope-moment} give
\begin{align}
\label{eq:selfnorm-martingale-increment}
E|\Delta_n(\beta,\beta')|^{2r}
&\le
C_r E\{\langle\Delta_n(\beta,\beta')\rangle(\tau)^r\}
+C_rE\sum_{u\le\tau}|\Delta\Delta_n(\beta,\beta';u)|^{2r} \\
&\le C_r n^{-r}\|\beta-\beta'\|^{2r}. \notag
\end{align}
For the jump contribution, Lemma~\ref{lem:cox-no-ties} implies
$\Delta\bar M_{\ell,0}^{\beta}(u)=\Delta\Ptr N_\ell(u)\le n^{-1}$. Hence
\[
E\sum_{u\le\tau}|\Delta\Delta_n(\beta,\beta';u)|^{2r}
\le
C_r n^{-2r+1}\|\beta-\beta'\|^{2r}
E\int Q_n(u)^{2r}S_n(\beta_0,u)d\Lambda_0(u),
\]
which is bounded by the right-hand side of
\eqref{eq:selfnorm-martingale-increment}, since $r\ge1$ and
\eqref{eq:selfnorm-jump-envelope-moment} holds.

Let $T_m$ be a Euclidean $2^{-m}$-net of $\mathcal B_\ell$ with $|T_m|\le C2^{mp_\ell}$.
Chaining from a fixed coarse net and applying Markov's inequality with
\eqref{eq:selfnorm-martingale-increment} gives, for $A>0$,
\[
\Pr\left(
\max_{\text{level }m\text{ pairs}}
|M_n(t)-M_n(s)|>A n^{-1/2}2^{-m/2}
\right)
\le C_rA^{-2r}2^{-m(r-p_\ell)} .
\]
Because $r>p_\ell+2$, the bound is summable in $m$. The finitely many
coarse-net values converge to zero by the fixed-$\beta$ martingale bound,
and the chained increments are $O_p(n^{-1/2})$. Thus
$\sup_{\beta\in\mathcal B_\ell}|M_n(\beta)|=o_p(1).$
Combining martingale and compensator equicontinuity shows that the
counting-integral supremum over $\mathcal B_\ell$ is $o_p(1)$.

\end{proof}

\begin{lemma}
\label{lem:cox-empirical-score}
Under Assumption~\ref{ass:cox-main-primitives}\ref{cxmain:bounded}--\ref{eq:cox-main-source-support}, for each
transition $\ell\in\cJ$,
\begin{equation}
\label{eq:cox-score-at-truth}
    \|U_{\ell,n}(\beta_{\ell,0})\|=O_p(n^{-1/2}),
\end{equation}
and
\begin{equation}
\label{eq:cox-score-derivative-uniform}
    \mathcal D_{\ell,n}:=
    \sup_{\beta\in\mathcal B_\ell}
    \|\dot U_{\ell,n}(\beta)-\dot U_\ell(\beta)\|=o_p(1).
\end{equation}
\end{lemma}

\begin{proof} [Proof of Lemma \ref{lem:cox-empirical-score}]
Fixed dimension allows us to work coordinatewise and take finite maxima.
Suppress the transition subscript and write $\beta_0=\beta_{\ell,0}$, $\Lambda_0=\Lambda_{0,\ell,0}$, $S_n^{(r)}=S_{\ell,n}^{(r)}$,
$s^{(r)}=s_\ell^{(r)}$, $\bar Z_n=\bar Z_{\ell,n}$, and $\bar x=\bar z_\ell$. For $r=0,1,2,3$, the class
$\mathcal F_r= \{O\mapsto Y_j(u)e^{\beta^\top X}X^{\otimes r}: \beta\in\mathcal B_\ell, 0\le u\le\tau\}$
is bounded and Euclidean, so
$\sup_{\beta\in\mathcal B_\ell,\,u\le\tau} \|S_n^{(r)}(\beta,u)-s^{(r)}(\beta,u)\|=O_p(n^{-1/2})$, $r=0,1,2,3$.
Likewise, for a bounded, fixed-dimensional Lipschitz family $g_\beta(u,x)$,
\begin{equation}
\label{eq:cox-counting-empirical-process}
    \sup_{\beta\in\mathcal B_\ell}
    \left|
    (\Ptr-P)\int_0^{\tau}g_\beta(u,X)dN_\ell(u)
    \right|=O_p(n^{-1/2}).
\end{equation}
This applies to $g_{\beta_0}(u,X)=X-\bar x(\beta_0,u)$ and to each scalar entry of $V(\beta,u)$, since
$\bar x$ and $V$ are bounded and locally Lipschitz by
Lemma~\ref{lem:cox-risk-regularity}.

At $\beta_0$, adding and subtracting the population risk-set mean gives
\begin{align*}
U_n(\beta_0)-U(\beta_0)
&=(\Ptr-P)\int_0^{\tau}\{X-\bar x(\beta_0,u)\}dN_\ell(u) \\
&\quad+
\Ptr\int_0^{\tau}\{\bar x(\beta_0,u)-\bar Z_n(\beta_0,u)\}dN_\ell(u).
\end{align*}
The first term is $O_p(n^{-1/2})$ by \eqref{eq:cox-counting-empirical-process}.
The second is $O_p(n^{-1/2})$ by the fixed-$\beta$ conclusion of
Lemma~\ref{lem:cox-self-normalized-ratios}, with $g_\beta(u,x)=x$ at
$\beta=\beta_0$. Since $U(\beta_0)=0$ by Lemma~\ref{lem:cox-population-score}, this
proves \eqref{eq:cox-score-at-truth}.

The derivative difference has the decomposition
\[
\dot U_n(\beta)-\dot U(\beta)
=-(\Ptr-P)\int_0^{\tau}V(\beta,u)dN_\ell(u)
-
\Ptr\int_0^{\tau}\{V_n(\beta,u)-V(\beta,u)\}dN_\ell(u).
\]
The first term is uniformly $O_p(n^{-1/2})$ by
\eqref{eq:cox-counting-empirical-process}. For the second, write
\[
V_n-V=
\left\{\frac{J_nS_n^{(2)}}{S_n^{(0)}}-
       \frac{j_0s^{(2)}}{s^{(0)}}\right\}
-
\left\{\bar Z_n^{\otimes2}-\bar x^{\otimes2}\right\}.
\]
Each entry of the first bracket is a self-normalized ratio error covered
by Lemma~\ref{lem:cox-self-normalized-ratios} with $g_\beta(u,x)=z_rz_s$.
Its integral against $\Ptr dN_\ell$ is therefore uniformly over $\beta\in\mathcal B_\ell$ of
order $o_p(1)$. For the second bracket,
$\|\bar Z_n^{\otimes2}-\bar x^{\otimes2}\| \le \|\bar Z_n-\bar x\|\{\|\bar Z_n\|+\|\bar x\|\} \le 2K\|\bar Z_n-\bar x\|$,
The same lemma with $g_\beta(u,x)=x$ gives a uniformly $o_p(1)$ integral.
Consequently, $\mathcal D_{\ell,n}=o_p(1)$, proving \eqref{eq:cox-score-derivative-uniform}.

\end{proof}

\begin{lemma}
\label{lem:cox-root-consistency}
Under Assumption~\ref{ass:cox-main-primitives} together with
Lemmas~\ref{lem:cox-population-score} and
\ref{lem:cox-empirical-score}, for each transition
$\ell\in\cJ$, the convex minimizer $\widehat\beta_\ell$ defined in
\eqref{eq:cox-beta-hat-def} satisfies, with probability tending to one,
$U_{\ell,n}(\widehat\beta_\ell)=0, \qquad \widehat\beta_\ell\in\operatorname{int}(\Theta_\ell), \qquad \|\widehat\beta_\ell-\beta_{\ell,0}\|=O_p(n^{-1/2}).$
\end{lemma}

\begin{proof} [Proof of Lemma \ref{lem:cox-root-consistency}]
The criterion $Q_{\ell,n}=-\ell_{\ell,n}$ is continuous and convex on the
compact convex set $\Theta_\ell$, so its minimizer set is nonempty.
Since $\beta_{\ell,0}$ is interior, a fixed neighborhood of it lies inside
$\Theta_\ell$. On that neighborhood,
Lemmas~\ref{lem:cox-population-score} and
\ref{lem:cox-empirical-score} give local curvature and
uniform convergence of the score derivative. More precisely, for any
$\eta>0$, choose $L<\infty$ so that, for all sufficiently large $n$, the event
$\mathcal E_n(L)=\{\sqrt n\|U_{\ell,n}(\beta_{\ell,0})\|\le L, \ \mathcal D_{\ell,n}\le\iota_\ell/4\}$
has probability at least $1-\eta$.
For fixed $M$, the ball
$B_n(M)=\{\beta:\|\beta-\beta_{\ell,0}\|\le Mn^{-1/2}\}$ lies inside
that neighborhood for large $n$. The population derivative differs from
$\dot U_\ell(\beta_{\ell,0})=-I_\ell$ by at most $\iota_\ell/4$ there.
Since $\lambda_{\min}(I_\ell)\ge\iota_\ell$, on $\mathcal E_n(L)$,
\begin{equation}
\label{eq:cox-local-curvature}
v^\top\dot U_{\ell,n}(\beta)v
\le-\frac{\iota_\ell}{2}\|v\|^2,
\qquad \beta\in B_n(M).
\end{equation}
For $\|h\|=M$, Taylor's formula and
\eqref{eq:cox-local-curvature} imply
\[
\ell_{\ell,n}(\beta_{\ell,0}+n^{-1/2}h)
 -\ell_{\ell,n}(\beta_{\ell,0})
\le n^{-1}\{LM-\iota_\ell M^2/4\}<0
\]
when $M>4L/\iota_\ell$. Thus the maximizer over $B_n(M)$ is interior;
strict concavity there makes it unique and gives the score equation.
It is also the global maximizer over $\Theta_\ell$: an exterior point
with at least as large a criterion value would, by concavity along the
joining segment, give a boundary point with criterion value at least
$\ell_{\ell,n}(\beta_{\ell,0})$, contradicting the displayed inequality.
Consequently the estimator in \eqref{eq:cox-beta-hat-def} is this unique
interior maximizer on $\mathcal E_n(L)$ and satisfies
$\|\widehat\beta_\ell-\beta_{\ell,0}\|\le Mn^{-1/2}$ and
$U_{\ell,n}(\widehat\beta_\ell)=0$.
Since $\eta$ was arbitrary, the asserted rate, interiority, and uniqueness follow.
\end{proof}

\begin{lemma}
\label{lem:breslow-rates}
Under Assumption~\ref{ass:cox-main-primitives}, for each transition $\ell\in\cJ$,
\begin{align}
\label{eq:breslow-sup-rate}
\sup_{0\le t\le\tau}
|\widehat\Lambda_{0,\ell}(t)-\Lambda_{0,\ell,0}(t)|&=O_p(n^{-1/2}), \\
\label{eq:breslow-squared-jumps}
\sum_{0<u\le\tau}\{\Delta\widehat\Lambda_{0,\ell}(u)\}^2&=O_p(n^{-1}).
\end{align}
\end{lemma}

\begin{proof} [Proof of Lemma \ref{lem:breslow-rates}]
Write $\beta_0=\beta_{\ell,0}$, $\widehat\beta=\widehat\beta_\ell$, $S_n(\beta,u)=S_{\ell,n}^{(0)}(\beta,u)$, and $s_0(u)=s_{\ell,0}(u)$. We first control the
Breslow estimator at $\beta_0$, then the effect of replacing $\beta_0$ by
$\widehat\beta$, and finally its squared jumps. Define
$\widetilde\Lambda_{0,\ell}(t)=\int_0^t\frac{J_{\ell,n}(u)}{S_n(\beta_0,u)}\,\Ptr dN_\ell(u).$
The training-average martingale decomposition in
Definition~\ref{def:cox-martingale} gives
$\Ptr dN_\ell(u)=d\bar M_{\ell,0}^{\beta}(u)+S_n(\beta_0,u)d\Lambda_{0,\ell,0}(u)$.
and therefore
\begin{equation}
\label{eq:breslow-truth-decomposition}
\widetilde\Lambda_{0,\ell}(t)-\Lambda_{0,\ell,0}(t)
 =
\int_0^t\frac{J_{\ell,n}(u)}{S_n(\beta_0,u)}\,d\bar M_{\ell,0}^{\beta}(u)
-
\int_0^t\mathbf 1\{S_n(\beta_0,u)=0\}\,d\Lambda_{0,\ell,0}(u).
\end{equation}

Since $S_n(\beta_0,u)=0$ is equivalent to an empty empirical source-state risk set,
\eqref{eq:cox-denominator-zero} in
Lemma~\ref{lem:cox-denominator-bounds} gives
$P\{S_n(\beta_0,u)=0\} \le C n^{-1}\frac{\mathbf 1\{s_0(u)>0\}}{s_0(u)}$
for $d\Lambda_{0,\ell,0}$-almost every $u$. The set $\{s_0=0\}$ has no baseline-hazard
mass by \eqref{eq:cox-main-source-support}. Hence
$E\int_0^{\tau}\mathbf 1\{S_n(\beta_0,u)=0\}\,d\Lambda_{0,\ell,0}(u) \le Cn^{-1}V_\ell^{\mathrm{cox}}=O(n^{-1}).$
Thus the zero-risk term in \eqref{eq:breslow-truth-decomposition} is $O_p(n^{-1})$
uniformly in $t$. The martingale term has predictable quadratic
variation bounded by
$\frac1n\int_0^t\frac{J_{\ell,n}(u)}{S_n(\beta_0,u)}\,d\Lambda_{0,\ell,0}(u)$.
Taking expectations and using \eqref{eq:cox-denominator-inverse-risk} gives
\[
E\left\langle\int_0^\cdot\frac{J_{\ell,n}}{S_n(\beta_0)}d\bar M_{\ell,0}^{\beta}\right\rangle(\tau)
\le \frac{C}{n}\int\frac{\mathbf 1\{s_0(u)>0\}}{s_0(u)}d\Lambda_{0,\ell,0}(u)
=O(n^{-1}).
\]
Doob's $L_2$ inequality yields
$\sup_{t\le\tau}\left|\int_0^t\frac{J_{\ell,n}}{S_n(\beta_0)}d\bar M_{\ell,0}^{\beta}\right| =O_p(n^{-1/2}).$
Combining the two bounds, we obtain
$\sup_{t\le\tau}|\widetilde\Lambda_{0,\ell}(t)-\Lambda_{0,\ell,0}(t)| =O_p(n^{-1/2})$, $\widetilde\Lambda_{0,\ell}(\tau)=O_p(1)$.
To replace $\beta_0$ by $\widehat\beta$, work on $\|\widehat\beta-\beta_0\|\le1$. Bounded covariates give
$e^{-K\|\widehat\beta-\beta_0\|}S_n(\beta_0,u) \le S_n(\widehat\beta,u) \le e^{K\|\widehat\beta-\beta_0\|}S_n(\beta_0,u).$
Thus, on $J_{\ell,n}=1$,
$\left|\frac{1}{S_n(\widehat\beta,u)}- \frac{1}{S_n(\beta_0,u)}\right| \le C\|\widehat\beta-\beta_0\|\frac{1}{S_n(\beta_0,u)}.$
Consequently,
$\sup_t|\widehat\Lambda_{0,\ell}(t)-\widetilde\Lambda_{0,\ell}(t)| \le C\|\widehat\beta-\beta_0\|\widetilde\Lambda_{0,\ell}(\tau) =O_p(n^{-1/2}),$
by Lemma~\ref{lem:cox-root-consistency}. This proves
\eqref{eq:breslow-sup-rate}.

For the squared jumps, Lemma~\ref{lem:cox-no-ties} ensures that
at most one subject has a transition-$\ell$ jump at any calendar time.
Thus, as discrete measures,
$(\Ptr dN_\ell)^2=n^{-1}\Ptr dN_\ell.$
On $\|\widehat\beta-\beta_0\|\le1$, the preceding denominator comparison gives
\[
\sum_u\{\Delta\widehat\Lambda_{0,\ell}(u)\}^2
\le
\frac{C}{n}\int_0^{\tau}
\frac{J_{\ell,n}(u)}{S_n(\beta_0,u)^2}\,\Ptr dN_\ell(u).
\]
The integrand at $\beta_0$ is predictable, with compensator
$\frac{C}{n}\int_0^{\tau} \frac{J_{\ell,n}(u)}{S_n(\beta_0,u)}\,d\Lambda_{0,\ell,0}(u),$
whose expectation is at most $Cn^{-1}V_\ell^{\mathrm{cox}}$ by
\eqref{eq:cox-denominator-inverse-risk}. Markov's inequality proves
\eqref{eq:breslow-squared-jumps}.
\end{proof}

\begin{lemma}
\label{lem:cox-risk-mean-replacement}
Fix a transition $\ell=(j,k)$, write $\beta_0=\beta_{\ell,0}$,
$s_0(u)=s_\ell^{(0)}(\beta_0,u)$, and let $h$ be bounded and deterministic.
Under Assumption~\ref{ass:cox-main-primitives},
\begin{equation}
\label{eq:cox-risk-mean-replacement}
\int_0^{\tau}s_0(u)h(u)
\{\bar Z_{\ell,n}(\beta_0,u)-\bar z_\ell(\beta_0,u)\}
\,d\Lambda_{0,\ell,0}(u)=O_p(n^{-1/2}).
\end{equation}
The same conclusion holds coordinatewise after multiplying the integrand by any
fixed bounded coordinate of $X$.
\end{lemma}

\begin{proof}[Proof of Lemma \ref{lem:cox-risk-mean-replacement}]
It suffices to work coordinatewise, since the Cox dimension is fixed. Write
$S_n^{(0)}(u)=S_{\ell,n}^{(0)}(\beta_0,u),\qquad S_n^{(1)}(u)=n^{-1}\sum_iY_{j,i}(u)e^{\beta_0^\top X_i}Z_{ir},$
and $s^{(0)}(u)=s_0(u)$, $s^{(1)}(u)=E\{Y_j(u)e^{\beta_0^\top X}Z_r\}$. Put $J(u)=\mathbf 1\{S_n^{(0)}(u)>0\}$. The bound $|Z_r|\le K$ implies $|s^{(1)}(u)|\le K s_0(u)$,
and $s^{(1)}(u)=0$ whenever $s_0(u)=0$. The exact decomposition is
\begin{equation}
\label{eq:cox-risk-mean-decomposition}
s_0J\bar Z_{\ell,n,r}-s_0\bar z_{\ell,r}
 =\frac{s_0J}{S_n^{(0)}}\{S_n^{(1)}-s^{(1)}\}
+s^{(1)}\left\{\frac{s_0J}{S_n^{(0)}}-1\right\}
-s^{(1)}\mathbf 1\{s_0=0\} .
\end{equation}
The last term has zero $d\Lambda_{0,\ell,0}$-mass by
\eqref{eq:cox-main-source-support}.

For the first term, set
$A_n(u)=\frac{s_0(u)J(u)}{S_n^{(0)}(u)}\{S_n^{(1)}(u)-s^{(1)}(u)\}.$
We establish
\begin{equation}
\label{eq:cox-risk-mean-first-moment}
E\{A_n(u)^2\}\le C n^{-1}s_0(u).
\end{equation}
Let $K_n(u)=\sum_iY_{j,i}(u)$ and $q(u)=P\{Y_j(u)=1\}$. As in
Lemma~\ref{lem:cox-denominator-bounds}, $S_n^{(0)}$ and $s_0$ are
comparable to $K_n/n$ and $q$, respectively. Consider $A_1=\{K_n\ge nq/2\}$,
$A_2=\{0<K_n<nq/2,nq\ge2\}$, and $A_3=\{0<K_n,nq<2\}$. On $A_1$, the multiplier $s_0J/S_n^{(0)}$ is bounded,
and
$E\{(S_n^{(1)}-s^{(1)})^2\}\le Cq/n\le C s_0/n,$
so the contribution of $A_1$ is bounded by $Cs_0/n$. On $A_2$,
the multiplier is at most $C nq$ and $|S_n^{(1)}-s^{(1)}|\le C(K_n/n+q)\le Cq$. The integrand is therefore
bounded by $C(nq)^2q^2$, giving contribution $C(nq)^2q^2e^{-cnq}\le Cs_0/n$ by Chernoff's inequality.
On $A_3$, we have $nq<2$ and $s_0J/S_n^{(0)}\le Cnq\le C$, while
$E\{(K_n/n+q)^2\mathbf 1(K_n>0)\}\le Cq/n.$
Thus the contribution of $A_3$ is also bounded by $Cs_0/n$, proving
\eqref{eq:cox-risk-mean-first-moment}.

For the second term in \eqref{eq:cox-risk-mean-decomposition}, use $|s^{(1)}|\le Ks_0$ and
\eqref{eq:cox-denominator-weighted-ratio}. Since $0\le s_0(u)\le C$,
\begin{align}
\label{eq:cox-risk-mean-second-moment}
&E\left[
|s^{(1)}(u)|^2\left\{\frac{s_0(u)J(u)}{S_n^{(0)}(u)}-1\right\}^2
\right] \\
&\qquad\le C s_0(u)
E\left[
s_0(u)\left\{\frac{s_0(u)J(u)}{S_n^{(0)}(u)}-1\right\}^2
\right]
\le C n^{-1}s_0(u). \notag
\end{align}
Let $D_n(u)$ be the sum of the first two terms in
\eqref{eq:cox-risk-mean-decomposition}. Cauchy--Schwarz for the finite measure $d\Lambda_{0,\ell,0}$
gives
\[
E\left[\left\{\int_0^{\tau}h(u)D_n(u)d\Lambda_{0,\ell,0}(u)\right\}^2\right]
\le \|h\|_\infty^2 \Lambda_{0,\ell,0}(\tau)
\int_0^{\tau}E\{D_n(u)^2\}d\Lambda_{0,\ell,0}(u).
\]
The integral on the right is $O(n^{-1})$ by
\eqref{eq:cox-risk-mean-first-moment} and \eqref{eq:cox-risk-mean-second-moment},
since $\int s_0d\Lambda_{0,\ell,0}$ and $\Lambda_{0,\ell,0}(\tau)$ are finite. Thus the second moment of the
integral in \eqref{eq:cox-risk-mean-replacement} is $O(n^{-1})$, as required.
\end{proof}

\begin{lemma}
\label{lem:cox-self-normalized-equicontinuity}
Fix a transition $\ell$ and a bounded deterministic function $h$.  Let
$p_\ell=\dim(\beta_{\ell,0})$, which is fixed by
Assumption~\ref{ass:cox-main-primitives}\ref{cxmain:bounded}.  For every fixed $M<\infty$,
\begin{equation}
\label{eq:selfnorm-beta-equicontinuity}
\sup_{\|\beta-\beta_{\ell,0}\|\le Mn^{-1/2}}
\left|
\int_0^{\tau}h(u)
\left\{
\frac{s_{\ell,0}(u)J_{\ell,n}(u)}{S_{\ell,n}^{(0)}(\beta,u)}-
\frac{s_{\ell,0}(u)J_{\ell,n}(u)}{S_{\ell,n}^{(0)}(\beta_{\ell,0},u)}
\right\}
 d\bar M_{\ell,0}^{\beta}(u)
\right|=O_p(n^{-1}).
\end{equation}
\end{lemma}

\begin{proof} [Proof of Lemma \ref{lem:cox-self-normalized-equicontinuity}]
Write $\beta_0=\beta_{\ell,0}$, $S_n(\beta,u)=S_{\ell,n}^{(0)}(\beta,u)$, $s_0=s_{\ell,0}$, and $J=J_{\ell,n}$. On the shrinking ball
$B_n(M)=\{\beta:\|\beta-\beta_0\|\le Mn^{-1/2}\}$, define
$$M_n(\beta)=\int_0^{\tau}h(u) \left\{\frac{s_0(u)J(u)}{S_n(\beta,u)}- \frac{s_0(u)J(u)}{S_n(\beta_0,u)}\right\} d\bar M_{\ell,0}^{\beta}(u).$$

For $\beta,\beta'\in B_n(M)$, the mean-value theorem and bounded covariates give,
$\left|\frac1{S_n(\beta,u)}-\frac1{S_n(\beta',u)}\right| \le C\|\beta-\beta'\|\frac1{S_n(\beta_0,u)}$,
on $J=1$ .
Indeed, $S_n(\beta,u)/S_n(\beta_0,u)$ lies between $e^{-K\|\beta-\beta_0\|}$ and $e^{K\|\beta-\beta_0\|}$ on this ball.
With
$g_{\beta,\beta'}(u)=h(u)s_0(u)J(u) \left\{\frac1{S_n(\beta,u)}-\frac1{S_n(\beta',u)}\right\}$,
we therefore have
\begin{equation}
\label{eq:selfnorm-lipschitz-bound}
|g_{\beta,\beta'}(u)|
\le C\|\beta-\beta'\|\,\frac{s_0(u)J(u)}{S_n(\beta_0,u)}.
\end{equation}

Let
$\Delta_n(\beta,\beta')=M_n(\beta)-M_n(\beta') =\int g_{\beta,\beta'}(u)d\bar M_{\ell,0}^{\beta}(u)$.
For an integer $r>p_\ell+2$, the Burkholder--Davis--Gundy inequality for purely
discontinuous square-integrable martingales gives
\begin{equation}
\label{eq:selfnorm-bdg-bound}
E|\Delta_n(\beta,\beta')|^{2r}
\le C_r E\{\langle\Delta_n(\beta,\beta')\rangle(\tau)^r\}
+C_rE\sum_{u\le\tau}|\Delta\Delta_n(\beta,\beta';u)|^{2r}.
\end{equation}
The predictable quadratic variation satisfies
$\langle\Delta_n(\beta,\beta')\rangle(\tau) \le \frac1n\int_0^{\tau}g_{\beta,\beta'}(u)^2 S_n(\beta_0,u)d\Lambda_{0,\ell,0}(u).$
Apply \eqref{eq:selfnorm-lipschitz-bound}, Jensen's inequality for the finite measure
$s_0(u)d\Lambda_{0,\ell,0}(u)$, and \eqref{eq:cox-denominator-moment-ratio} to obtain
\begin{equation}
\label{eq:selfnorm-bracket-moment}
E\{\langle\Delta_n(\beta,\beta')\rangle(\tau)^r\}
\le C_r n^{-r}\|\beta-\beta'\|^{2r}.
\end{equation}
For the jump contribution, Lemma~\ref{lem:cox-no-ties} gives
$\Delta\Lambda_{0,\ell,0}(u)=0$. Thus the compensator $S_n(\beta_0,u)d\Lambda_{0,\ell,0}(u)$ is continuous and
$\Delta\bar M_{\ell,0}^{\beta}(u)=\Delta\Ptr N_\ell(u).$
At each transition-$\ell$ jump time, at most one subject jumps, so
$\Delta\Ptr N_\ell(u)\le n^{-1}$. Consequently,
\[
\sum_{u\le\tau}|\Delta\Delta_n(\beta,\beta';u)|^{2r} \le n^{-2r}\sum_{i=1}^n\int_0^{\tau} |g_{\beta,\beta'}(u)|^{2r}\,dN_{\ell,i}(u).
\]
Taking expectations with the predictable compensator of $N_{\ell,i}$ yields
\[
E\sum_{u\le\tau}|\Delta\Delta_n(\beta,\beta';u)|^{2r}
\le n^{-2r+1}
\int E\{|g_{\beta,\beta'}(u)|^{2r}S_n(\beta_0,u)\}
    d\Lambda_{0,\ell,0}(u).
\]
By \eqref{eq:selfnorm-lipschitz-bound} and \eqref{eq:cox-denominator-moment-weighted},
\begin{align}
\label{eq:selfnorm-jump-moment}
E\sum_{u\le\tau}|\Delta\Delta_n(\beta,\beta';u)|^{2r}
&\le C_r n^{-2r+1}\|\beta-\beta'\|^{2r}
\int s_0(u)d\Lambda_{0,\ell,0}(u) \\
&\le C_r n^{-r}\|\beta-\beta'\|^{2r}. \notag
\end{align}
The final inequality uses $n^{-2r+1}\le n^{-r}$ for $r\ge1$. Combining
\eqref{eq:selfnorm-bdg-bound}--\eqref{eq:selfnorm-jump-moment} gives
\begin{equation}
\label{eq:selfnorm-increment-moment}
E|M_n(\beta)-M_n(\beta')|^{2r}
\le C_r n^{-r}\|\beta-\beta'\|^{2r}.
\end{equation}
Taking $\beta'=\beta_0$ and using $\|\beta-\beta_0\|\le Mn^{-1/2}$ gives
\begin{equation}
\label{eq:selfnorm-pointwise-moment}
E|M_n(\beta)|^{2r}\le C_r n^{-2r}.
\end{equation}

For the uniform bound, apply the dyadic-net argument used in the proof of
Lemma~\ref{lem:cox-self-normalized-ratios}, now on $B_n(M)$.
A $2^{-m}Mn^{-1/2}$-net has cardinality at most $C2^{mp_\ell}$.
The paths are continuous in $\beta$, since $S_n(\beta,u)$ is smooth
and its reciprocal is used only where $J=1$.
For the successive net increments, \eqref{eq:selfnorm-increment-moment} and
Markov's inequality give
\[
\Pr\left\{\max_{\text{level }m\text{ pairs}}
|M_n(t)-M_n(s)|>An^{-1}2^{-m/2}\right\}
\le C_rA^{-2r}2^{-m(r-p_\ell)}.
\]
The right-hand side is summable because $r>p_\ell+2$.
The coarse-net contribution is $O_p(n^{-1})$ by
\eqref{eq:selfnorm-pointwise-moment}, and the remaining thresholds sum to
$CAn^{-1}$. Thus $\sup_{\beta\in B_n(M)}|M_n(\beta)|=O_p(n^{-1})$,
proving \eqref{eq:selfnorm-beta-equicontinuity} .
\end{proof}

\begin{lemma}
\label{lem:breslow-weighted-linearization}
Fix a transition $\ell=(j,k)$, write $s_{\ell,0}(u)=s_\ell^{(0)}(\beta_{\ell,0},u)$,
and let $h_\ell$ be a bounded deterministic function.  Under
Assumption~\ref{ass:cox-main-primitives},
\begin{align}
\label{eq:breslow-weighted-linearization}
&\int_0^{\tau}s_{\ell,0}(u)h_\ell(u)
\,d\{\widehat\Lambda_{0,\ell}-\Lambda_{0,\ell,0}\}(u) \\
&\qquad=
\int_0^{\tau}h_\ell(u)\,d\bar M_{\ell,0}^{\beta}(u)
-
B_{\ell,h}^{\top}(\widehat\beta_\ell-\beta_{\ell,0})
+O_p(n^{-1}), \notag
\end{align}
where $d\bar M_{\ell,0}^{\beta}$ is the training-average Cox martingale
increment defined in \eqref{eq:cox-average-martingale},
and
$B_{\ell,h} =\int_0^{\tau}s_{\ell,0}(u)h_\ell(u) \bar z_\ell(\beta_{\ell,0},u)\,d\Lambda_{0,\ell,0}(u)$.
Consequently the left side is $O_p(n^{-1/2})$.
\end{lemma}

\begin{proof} [Proof of Lemma \ref{lem:breslow-weighted-linearization}]
Write $\beta_0=\beta_{\ell,0}$, $\widehat\beta=\widehat\beta_\ell$, $S_n(\beta,u)=S_{\ell,n}^{(0)}(\beta,u)$, $s_0=s_{\ell,0}$, and $J=J_{\ell,n}$. The Breslow
definition and the decomposition $\Ptr dN_\ell=d\bar M_{\ell,0}^{\beta}+S_n(\beta_0,u)d\Lambda_{0,\ell,0}$ give
\begin{align*}
\int s_0h\,d(\widehat\Lambda_0-\Lambda_0)
&=R_{1,n}+R_{2,n}+R_{3,n}, \\
R_{1,n}&=\int h\,d\bar M_{\ell,0}^{\beta}, \\
R_{2,n}&=\int h\left\{\frac{s_0J}{S_n(\widehat\beta)}-1\right\}d\bar M_{\ell,0}^{\beta}, \\
R_{3,n}&=\int s_0h\left\{\frac{JS_n(\beta_0)}{S_n(\widehat\beta)}-1\right\}
 d\Lambda_{0,\ell,0}.
\end{align*}
For $R_{2,n}$, write

\[
\begin{aligned}
R_{2,n}&=\int h\left\{\frac{s_0J}{S_n(\beta_0)}-1\right\}d\bar M_{\ell,0}^{\beta}\\
&\quad+\int h\left\{\frac{s_0J}{S_n(\widehat\beta)}-\frac{s_0J}{S_n(\beta_0)}\right\} d\bar M_{\ell,0}^{\beta}.
\end{aligned}
\]

The first term has a predictable integrand and expected quadratic
variation bounded by

\[
\frac{\|h\|_\infty^2}{n}\int E\left[ S_n(\beta_0,u)\left\{\frac{s_0(u)J(u)}{S_n(\beta_0,u)}-1\right\}^2 \right]d\Lambda_{0,\ell,0}(u)=O(n^{-2})
\]

using \eqref{eq:cox-denominator-ratio}; hence it is $O_p(n^{-1})$.
The second contains the random parameter $\widehat\beta$. On $\|\widehat\beta-\beta_0\|\le Mn^{-1/2}$,
Lemma~\ref{lem:cox-self-normalized-equicontinuity} bounds it by $O_p(n^{-1})$.
Since this event has probability tending to one by
Lemma~\ref{lem:cox-root-consistency}, $R_{2,n}=O_p(n^{-1})$.

For $R_{3,n}$, the indicator $J$ does not depend on $\beta$.
On $\|\widehat\beta-\beta_0\|\le1$, Taylor expansion of $S_n(\beta_0,u)/S_n(\widehat\beta,u)$ on $\{J=1\}$ gives
\[
\frac{S_n(\beta_0,u)}{S_n(\widehat\beta,u)}-1
= -\bar Z_{\ell,n}(\beta_0,u)^\top(\widehat\beta-\beta_0)+\rho_n(u),
\qquad
|\rho_n(u)|\le C\|\widehat\beta-\beta_0\|^2.
\]
The contribution from $\{J=0\}$ has expectation bounded by $Cn^{-1}\int \mathbf 1\{s_0>0\}s_0(u)^{-1}d\Lambda_{0,\ell,0}(u)$ and
is therefore $O_p(n^{-1})$. Since $\|\widehat\beta-\beta_0\|^2=O_p(n^{-1})$ and $\int s_0d\Lambda_{0,\ell,0}<\infty$, the Taylor remainder
contributes $O_p(n^{-1})$. Thus
\[
R_{3,n}=-(\widehat\beta-\beta_0)^\top
\int s_0(u)h(u)\bar Z_{\ell,n}(\beta_0,u)d\Lambda_{0,\ell,0}(u)+O_p(n^{-1}).
\]
Lemma~\ref{lem:cox-risk-mean-replacement} replaces $\bar Z_{\ell,n}$ in this
integral by $\bar z_\ell$ with error $O_p(n^{-1/2})$. Multiplication by $\|\widehat\beta-\beta_0\|=O_p(n^{-1/2})$
makes this error $O_p(n^{-1})$, so $R_{3,n}=-B_{\ell,h}^\top(\widehat\beta-\beta_0)+O_p(n^{-1})$. Combining the three terms proves
\eqref{eq:breslow-weighted-linearization}.

The leading martingale $R_{1,n}$ has expected quadratic variation bounded
by $n^{-1}\|h\|_\infty^2\int s_0d\Lambda_{0,\ell,0}$ and is consequently $O_p(n^{-1/2})$. The coefficient-estimation term
is also $O_p(n^{-1/2})$.
\end{proof}

\begin{lemma}
\label{lem:cox-row-stabilization}
Under Assumption~\ref{ass:cox-main-primitives}, the stabilized Cox plug-in satisfies, uniformly
over covariate values in the support of $X$,
\begin{align}
\label{eq:cox-conditional-hazard-sup}
\|\widehat A_n(\cdot\mid x)-A_0(\cdot\mid x)\|_{\infty,\tau}&=O_p(n^{-1/2}), \\
\label{eq:cox-conditional-hazard-jumps}
[\widehat A_n(\cdot\mid x)-A_0(\cdot\mid x)]_{\Delta,2,\tau}&=O_p(n^{-1}).
\end{align}
Moreover, $\widehat A_n(\cdot\mid x)\in\Aspace_{\tau}$ pathwise for every
$x$.  The only total-variation bound asserted in this lemma is the
bound for the artificial row-stabilization correction $\widehat A_n-\bar A_n$:
\begin{equation}
\label{eq:cox-row-stabilization-variation}
\sum_{\ell\in\cJ}\sup_x
\left|\widehat A_{\ell,n}(\cdot\mid x)-\bar A_{\ell,n}(\cdot\mid x)\right|_{\mathrm{TV},\tau}
=O_p(n^{-1}).
\end{equation}
 \end{lemma}

\begin{proof}[Proof of Lemma \ref{lem:cox-row-stabilization}]
Lemmas~\ref{lem:breslow-rates} and~\ref{lem:cox-root-consistency}
give the raw-hazard bound
$\sup_x\|\bar A_n(\cdot\mid x)-A_0(\cdot\mid x)\|_{\infty,\tau}=O_p(n^{-1/2}).$
Indeed, on the localization event of Lemma~\ref{lem:cox-root-consistency},
bounded covariates imply
$|e^{\widehat\beta_\ell^\top x}-e^{\beta_{\ell,0}^\top x}| \le C\|\widehat\beta_\ell-\beta_{\ell,0}\|, \qquad \sup_x e^{\widehat\beta_\ell^\top x}\le C$
The squared-jump bound \eqref{eq:breslow-squared-jumps} likewise gives
$\sup_x\sum_u\{\Delta\bar A_{\ell,n}(u\mid x)\}^2=O_p(n^{-1}).$
Continuity of $A_0$ makes this also the squared-jump bound for the raw
hazard error.

The row transformation is pathwise admissible: the nonnegative row-zero
increments sum to $1-e^{-x_0}\le1$, and the row-one increment is $1-e^{-x_{12}}\le1$.
Hence $\widehat A_n(\cdot\mid x)\in\Aspace_{\tau}$. To bound the correction $\widehat A_n-\bar A_n$ in total variation,
fix $u$ and write the raw row-zero increments as
$x_{01}=\Delta\bar A_{01,n}(u\mid x)$, $x_{02}=\Delta\bar A_{02,n}(u\mid x)$, $x_0=x_{01}+x_{02}$.
The transformed increments are
$\Delta\widehat A_{0k,n}(u\mid x) =\frac{x_{0k}}{x_0}\{1-e^{-x_0}\}, \qquad k=1,2,$
with value zero when $x_0=0$. The inequality $|1-e^{-x}-x|\le x^2/2$ gives
$\left|\Delta\widehat A_{0k,n}(u\mid x)-x_{0k}\right| = x_{0k} \left|\frac{1-e^{-x_0}}{x_0}-1\right| \le x_{0k}x_0/2 \le x_0^2/2 .$
Summing over $k=1,2$ and using $(x_{01}+x_{02})^2\le2x_{01}^2+2x_{02}^2$ bounds the row-zero contribution
by a constant times $x_{01}^2+x_{02}^2$. For row one, set $x_{12}=\Delta\bar A_{12,n}(u\mid x)$; then
$|\Delta\widehat A_{12,n}(u\mid x)-x_{12}| =|1-e^{-x_{12}}-x_{12}| \le x_{12}^2/2.$
Summing over the empirical jump times yields
\[
\sum_{\ell\in\cJ}\sup_x
\left|\widehat A_{\ell,n}(\cdot\mid x)-\bar A_{\ell,n}(\cdot\mid x)\right|_{\mathrm{TV},\tau}
\le C\sup_x\sum_{\ell\in\cJ}\sum_{0<u\le\tau}
\{\Delta\bar A_{\ell,n}(u\mid x)\}^2
=O_p(n^{-1}),
\]
The final equality follows from \eqref{eq:breslow-squared-jumps},
transferred to the conditional hazards by bounded covariates. The sup-norm
correction is therefore also $O_p(n^{-1})$. Finally,
$|\Delta\widehat A_{\ell,n}|\le |\Delta\bar A_{\ell,n}|+C(\Delta\bar A_{\ell,n})^2$
and the raw squared-jump bound give the corresponding bound for the
transformed hazards. This proves \eqref{eq:cox-conditional-hazard-sup},
\eqref{eq:cox-conditional-hazard-jumps}, and \eqref{eq:cox-row-stabilization-variation}.
\end{proof}

\begin{theorem}
\label{thm:cox-breslow-rates}
Under Assumption~\ref{ass:cox-main-primitives}, the explicit coefficient and
baseline-hazard estimators above satisfy the corresponding rates stated
in Proposition~\ref{prop:cox-main-regime}. For each $\ell\in\cJ$ and fixed
bounded deterministic $h_\ell$, they also satisfy
\begin{equation}
\label{eq:cox-source-weighted-rate}
\int_0^{\tau}s_{\ell,0}(u)h_\ell(u)
\,d\{\widehat\Lambda_{0,\ell}-\Lambda_{0,\ell,0}\}(u)=O_p(n^{-1/2}).
\end{equation}
For the row-stabilized conditional hazards specified above,
\begin{equation}
\label{eq:cox-plugin-rate}
\widehat\mu_{n,m}^{\mathrm{plug}}(\tau)-\mu_\tau(A_0)
=O_p(m^{-1/2}+n^{-1/2}).
\end{equation}
In particular, the rate is $O_p(n^{-1/2})$ when $m\asymp n$.
No condition of the form $\inf_{u\le\tau}s_{12,0}(u)>0$ is imposed.
\end{theorem}

\begin{proof}[Proof of Theorem~\ref{thm:cox-breslow-rates}]
The coefficient and baseline-hazard rates stated in
Proposition~\ref{prop:cox-main-regime} follow, for the estimators specified
here, from Lemmas~\ref{lem:cox-root-consistency}
and~\ref{lem:breslow-rates}.
Lemma~\ref{lem:breslow-weighted-linearization} gives
\eqref{eq:cox-source-weighted-rate}: its leading martingale has order
$n^{-1/2}$, the coefficient-estimation term is $O_p(n^{-1/2})$, and the remainder
is $O_p(n^{-1})$.

By Lemma~\ref{lem:cox-row-stabilization}, $\widehat A_n(\cdot\mid x)\in\Aspace_{\tau}$ for every $x$.
Lemma~\ref{lem:cox-row-stabilization} gives
$\sup_x\|\widehat A_n(\cdot\mid x)-A_0(\cdot\mid x)\|_{\infty,\tau} =O_p(n^{-1/2}).$
Its proof obtains this rate for the raw hazards and transfers it to the
row-admissible hazards using the $O_p(n^{-1})$ total-variation correction in
\eqref{eq:cox-row-stabilization-variation}.

The true hazards satisfy
$\sup_x A_{\ell,0}(\tau\mid x) \le e^{\|\beta_{\ell,0}\|K}\Lambda_{0,\ell,0}(\tau)<\infty.$
For the fitted hazards, Lemma~\ref{lem:breslow-rates} gives
$\widehat\Lambda_{0,\ell}(\tau)=O_p(1)$. Bounded covariates and $\widehat\beta_\ell=\beta_{\ell,0}+O_p(n^{-1/2})$ imply $\sup_x e^{\widehat\beta_\ell^\top x}=O_p(1)$. Row stabilization
changes the total mass by at most the $O_p(n^{-1})$ total-variation correction.
Thus, for each $\eta>0$, there exists $M_\eta<\infty$ such that
$\Pr\left( \sup_x |A_0(\cdot\mid x)|_{\tau} \vee \sup_x |\widehat A_n(\cdot\mid x)|_{\tau} \le M_\eta \right)\ge1-\eta$
for sufficiently large $n$.
On this event, \eqref{eq:product-integral-lipschitz} in
Lemma~\ref{lem:product-integral-expansion} gives
\[
\sup_x\kappa_{\mathrm{PI}}\{A_0(\cdot\mid x),\widehat A_n(\cdot\mid x)\}
\le
C_{\mathrm{PI}}(M_\eta,\tau)
\sup_x\|\widehat A_n(\cdot\mid x)-A_0(\cdot\mid x)\|_{\infty,\tau}.
\]
The conditional-hazard rate above makes the right-hand side $O_p(n^{-1/2})$.
Applying \eqref{eq:rmst-product-integral-bound} in
Lemma~\ref{lem:rmst-hazard-derivative} yields
\[
\sup_x|\mu_{\widehat A_n}(\tau\mid x)-\mu_{A_0}(\tau\mid x)|
\le 2\tau
\sup_x\kappa_{\mathrm{PI}}\{A_0(\cdot\mid x),\widehat A_n(\cdot\mid x)\}
=O_p(n^{-1/2}).
\]
and in particular
$P_X\{\mu_{\widehat A_n}(\tau\mid X)-\mu_{A_0}(\tau\mid X)\} =O_p(n^{-1/2}).$

Adding and subtracting the evaluation and population means gives
\[
\widehat\mu_{n,m}^{\mathrm{plug}}-\mu_\tau(A_0)
=(\Pev-P_X)\mu_{A_0}
+P_X(\mu_{\widehat A_n}-\mu_{A_0})
+(\Pev-P_X)(\mu_{\widehat A_n}-\mu_{A_0})
\]
The first term is $O_p(m^{-1/2})$ by boundedness of $\mu_{A_0}$,
and the second is $O_p(n^{-1/2})$ by the preceding uniform bound.
For the last term, conditional independence and
\eqref{eq:conditional-evaluation-variance} give a variance bounded by
$m^{-1}P_X(\mu_{\widehat A_n}-\mu_{A_0})^2=O_p(m^{-1}n^{-1})$.
Conditional Chebyshev's inequality proves
\eqref{eq:cox-plugin-rate}; the stated special case follows when $m\asymp n$.
\end{proof}

We record a further consequence for the standardized RMST derivative.

\begin{lemma}
\label{lem:cox-rmst-source-domination}
Let
$L_{01}(u,x)=p_0(u-\mid x)\{V_{1,0}(u,\tau\mid x)-V_{0,0}(u,\tau\mid x)\},$
$L_{02}(u,x)=-p_0(u-\mid x)V_{0,0}(u,\tau\mid x)$, $L_{12}(u,x)=-p_1(u-\mid x)V_{1,0}(u,\tau\mid x)$.
For $r=0,1,\ldots,d$, with $Z_0\equiv1$, define
$a_{\ell,r}(u)=P_X\{L_\ell(u,X)e^{\beta_{\ell,0}^\top X}Z_r\}$, $s_{\ell,0}(u)=E\{Y_{j(\ell)}(u)e^{\beta_{\ell,0}^\top X}\}$.
Under conditional independent censoring with $\inf_{u,x}G(u-\mid x)\ge g_0>0$,
there are bounded deterministic functions $h_{\ell,r}$ such that
$a_{\ell,r}(u)=s_{\ell,0}(u)h_{\ell,r}(u)$, and $\|h_{\ell,r}\|_\infty\le K_r\tau/g_0,$
where $K_0=1$ and $K_r=K$ for $r\ge1$.
\end{lemma}

\begin{proof} [Proof of Lemma \ref{lem:cox-rmst-source-domination}]
For $\ell=(j,k)$, the bound $|L_\ell(u,x)|\le\tau p_j(u-\mid x)$ gives
$|a_{\ell,r}(u)| \le K_r\tau P_X\{p_j(u-\mid X)e^{\beta_{\ell,0}^\top X}\}.$
Conditional independent censoring implies
\[
s_{\ell,0}(u)=P_X\{p_j(u-\mid X)G(u-\mid X)e^{\beta_{\ell,0}^\top X}\}
\ge g_0P_X\{p_j(u-\mid X)e^{\beta_{\ell,0}^\top X}\}.
\]
Consequently, $|a_{\ell,r}(u)|\le K_r\tau g_0^{-1}s_{\ell,0}(u)$.
Taking $h_{\ell,r}=a_{\ell,r}/s_{\ell,0}$ on $\{s_{\ell,0}>0\}$ and zero elsewhere
proves the assertion.
\end{proof}

\begin{lemma}
\label{lem:cox-standardized-derivative-rate}
Recall that $\dotmu_{\tau,A_0}[\widehat A_n-A_0](x)$ is the conditional
RMST derivative from Lemma~\ref{lem:rmst-hazard-derivative}, evaluated at the true
conditional hazard $A_0(\cdot\mid x)$ in the direction
$\widehat A_n(\cdot\mid x)-A_0(\cdot\mid x)$.   Under Assumption~\ref{ass:cox-main-primitives},
\begin{equation}
\label{eq:cox-standardized-derivative-rate}
\dotmu_{\tau,A_0}^{\mathrm{std}}[\widehat A_n-A_0]=O_p(n^{-1/2}).
\end{equation}
\end{lemma}

\begin{proof} [Proof of Lemma \ref{lem:cox-standardized-derivative-rate}]
Lemma~\ref{lem:rmst-hazard-derivative} gives
$\dotmu_{\tau,A_0}[\widehat A_n-A_0](x) =\sum_{\ell\in\cJ}\int_0^{\tau}L_\ell(u,x) \,d\{\widehat A_{\ell,n}-A_{\ell,0}\}(u\mid x).$
We first treat the raw Cox hazards $\bar A_n$ and then account for the
row-stabilized hazards $\widehat A_n$. Fix $\ell$ and write $\Delta_\beta=\widehat\beta_\ell-\beta_{\ell,0}$.
For bounded $x$, Taylor's formula gives
\begin{align}
\label{eq:cox-rmst-taylor-expansion}
&\int L_\ell(u,x)\,d\{\bar A_{\ell,n}-A_{\ell,0}\}(u\mid x) \\
&=\int L_\ell(u,x)e^{\beta_{\ell,0}^\top x}
\,d\{\widehat\Lambda_{0,\ell}-\Lambda_{0,\ell,0}\}(u) \notag \\
&\quad+\Delta_\beta^\top
\int L_\ell(u,x)e^{\beta_{\ell,0}^\top x}x\,d\Lambda_{0,\ell,0}(u) \notag \\
&\quad+\Delta_\beta^\top
\int L_\ell(u,x)e^{\beta_{\ell,0}^\top x}x\,
   d\{\widehat\Lambda_{0,\ell}-\Lambda_{0,\ell,0}\}(u)
+R_{\ell,n}(x). \notag
\end{align}
The last line retains the product of the regression and baseline-hazard
errors. The remaining Taylor term satisfies
$P_X|R_{\ell,n}(X)|=O_p(n^{-1}).$
Indeed, the second derivative of $e^{\beta^\top x}$ is uniformly bounded on the
local-root event, giving
$$|R_{\ell,n}(x)|\le C\|\Delta_\beta\|^2 \left\{\Lambda_{0,\ell,0}(\tau)+ \sup_{t\le\tau}|\widehat\Lambda_{0,\ell}(t)-\Lambda_{0,\ell,0}(t)|\right\}.$$
By Lemmas~\ref{lem:cox-root-consistency} and~\ref{lem:breslow-rates},
the right-hand side is $O_p(n^{-1})$ after taking $P_X$-expectations.

The coefficient of the first baseline term is
$a_{\ell,0}(u)=P_X\{L_\ell(u,X)e^{\beta_{\ell,0}^\top X}\}$, obtained
by taking $P_X$-expectations in \eqref{eq:cox-rmst-taylor-expansion}.
Lemma~\ref{lem:cox-rmst-source-domination} gives $a_{\ell,0}(u)=s_{\ell,0}(u)h_{\ell,0}(u)$ with bounded
$h_{\ell,0}$.
Hence Lemma~\ref{lem:breslow-weighted-linearization} yields
$\int a_{\ell,0}(u)d\{\widehat\Lambda_{0,\ell}-\Lambda_{0,\ell,0}\}(u) =O_p(n^{-1/2})$.
The pure parameter term is therefore $O_p(n^{-1/2})$ by Lemma~\ref{lem:cox-root-consistency}.
with the finite vector coefficient
$\int P_X\{ L_\ell(u,X)e^{\beta_{\ell,0}^\top X}X\} d\Lambda_{0,\ell,0}(u)$.

For the cross term, the coordinate coefficients
$u_{\ell,r}(u)=P_X\{L_\ell(u,X)e^{\beta_{\ell,0}^\top X}Z_r\}$
satisfy $u_{\ell,r}(u)=s_{\ell,0}(u)h_{\ell,r}(u)$ with bounded $h_{\ell,r}$, again by
Lemma~\ref{lem:cox-rmst-source-domination}.
Thus Lemma~\ref{lem:breslow-weighted-linearization} gives
$\int u_{\ell,r}(u)d\{\widehat\Lambda_{0,\ell}-\Lambda_{0,\ell,0}\}(u) =O_p(n^{-1/2})$
for each $r$.
Fixed dimension and $\|\Delta_\beta\|=O_p(n^{-1/2})$ make the cross term $O_p(n^{-1})$.

Finally, replace $\bar A_n$ by $\widehat A_n$. The derivative coefficients satisfy
$|L_\ell(u,x)|\le\tau$, while \eqref{eq:cox-row-stabilization-variation} in
Lemma~\ref{lem:cox-row-stabilization} gives
$\sum_{\ell\in\cJ}\sup_x \left|\widehat A_{\ell,n}(\cdot\mid x)-\bar A_{\ell,n}(\cdot\mid x)\right|_{\mathrm{TV},\tau} =O_p(n^{-1})$.
Thus row stabilization changes $\dotmu_{\tau,A_0}^{\mathrm{std}}[\cdot]$ by $O_p(n^{-1})$. Summing over the
three transitions proves \eqref{eq:cox-standardized-derivative-rate}.
\end{proof}

\begin{remark}
\label{rem:bounded-vs-subgaussian-cox}
The bounded-support assumption gives the uniform-in-$x$ conditional-hazard bound
used in Proposition~\ref{prop:cox-main-regime}.  On the local-root event,
$\|\widehat\beta_\ell-\beta_{\ell,0}\|=O_p(n^{-1/2})$, and bounded covariates imply
\[
\sup_x\|
\bar A_n(\cdot\mid x)-A_0(\cdot\mid x)\|_{\infty,\tau}
\le C\sum_{\ell\in\cJ}\left
\{\|\widehat\beta_\ell-\beta_{\ell,0}\|+
\sup_{t\le\tau}|\widehat\Lambda_{0,\ell}(t)-\Lambda_{0,\ell,0}(t)|\right\}
=O_p(n^{-1/2}).
\]
For fixed-dimensional sub-Gaussian covariates, the same argument gives the
integrated replacement
$P_X\left[\sup_{t\le\tau} |\bar A_{\ell,n}(t\mid X)-A_{\ell,0}(t\mid X)|^2\right]=O_p(n^{-1}),$
provided the corresponding exponential moments are finite.  The inverse-risk
condition \eqref{eq:cox-main-source-support} is needed in both formulations.
\end{remark}

\subsection{Flexible conditional hazards and one-step estimation}

The following single-split expansion supplies the derivative term and
separate training and evaluation sample-size dependence used in part~(b)
of Theorem~\ref{thm:plugin-onestep-integrated-summary}.
Use $\kappa_{\rm PI}$ and $\omega_{\rm PI}$ from
Section~\ref{sec:covariates}, with the matrix norm fixed in
Section~\ref{sec:setup-notation}. Set
$A_x=A_0(\cdot\mid x)$ and $\widehat A_{n,x}=\widehat A_n(\cdot\mid x)$.
For the random fitted field we use the pointwise abbreviations
$\kappa_n(x):=\kappa_{\mathrm{PI}}(A_x,\widehat A_{n,x})$, $\omega_n(x):=\omega_{\mathrm{PI}}(A_x,\widehat A_{n,x})$.

\begin{theorem}
\label{thm:plugin-rmst-expansion}
Assume the conditional-hazard and sample-splitting setup at the beginning of
Section~\ref{app:covariate-proofs}, with the product-integral notation of
Section~\ref{sec:setup-notation}.  Let $\widehat A_n$ be $\calD_{\mathrm{tr},n}$-measurable and let the evaluation
sample be independent of $\calD_{\mathrm{tr},n}$.  Assume that, with probability
tending to one,
$A_0(\cdot\mid x),\widehat A_n(\cdot\mid x)\in\Aspace_{\tau} \quad\text{for }P_X\text{-a.e. }x,$
and that, for deterministic $a_n,b_n\downarrow0$,
\begin{equation}
\label{eq:plugin-rate-assumptions}
    P_X\{\kappa_n(X)^2\}=O_p(a_n^2),
    \qquad
    P_X\{\omega_n(X)\}=O_p(b_n).
\end{equation}
Then
\begin{equation}
\label{eq:plugin-rate-expansion}
\widehat\mu_{n,m}^{\mathrm{plug}}(\tau)-\mu_\tau(A_0)
=
(\Pev-P_X)\mu_{A_0}
+
\dotmu_{\tau,A_0}^{\mathrm{std}}[\widehat A_n-A_0]
+
O_p\{b_n+m^{-1/2}a_n\}.
\end{equation}
Moreover,
\begin{equation}
\label{eq:plugin-derivative-bound}
\left|\dotmu_{\tau,A_0}^{\mathrm{std}}[\widehat A_n-A_0]\right|
=O_p(a_n+b_n),
\end{equation}
and hence
\begin{equation}
\label{eq:plugin-upper-rate}
\widehat\mu_{n,m}^{\mathrm{plug}}(\tau)-\mu_\tau(A_0)
=
O_p(m^{-1/2}+a_n+b_n).
\end{equation}
In particular, if $m\asymp n$, $a_n=n^{-1/2}$, and $b_n=n^{-1}$, then
$\widehat\mu_{n,m}^{\mathrm{plug}}(\tau)-\mu_\tau(A_0) = O_p(n^{-1/2}).$
\end{theorem}

\begin{proof}[Proof of Theorem~\ref{thm:plugin-rmst-expansion}]
Condition on $\calD_{\mathrm{tr},n}$ and put
$g_n(x)=\mu_{\widehat A_n}(\tau\mid x)-\mu_{A_0}(\tau\mid x)$.
The evaluation observations are conditionally independent with common law $P_X$.
On the admissibility event in the theorem, Lemma~\ref{lem:rmst-hazard-derivative}
gives, for $P_X$-almost every $x$,
\begin{equation}
\label{eq:plugin-pointwise-expansion}
g_n(x)=\dotmu_{\tau,A_0}[\widehat A_n-A_0](x)+r_n(x),
\end{equation}
where $|r_n(x)|\le\tau\omega_n(x)$ and $|g_n(x)|\le2\tau\kappa_n(x)$.
Consequently, $P_X|r_n|\le\tau P_X\omega_n$ and
$P_Xg_n^2\le4\tau^2P_X\kappa_n^2$, which also justify the expectations below.
Adding and subtracting the population and evaluation means yields
\begin{align}
\label{eq:plugin-main-expansion}
\widehat\mu_{n,m}^{\mathrm{plug}}(\tau)-\mu_\tau(A_0)
&=(\Pev-P_X)\mu_{A_0}
 +\dotmu_{\tau,A_0}^{\mathrm{std}}[\widehat A_n-A_0]+R_{n,m},\\
\label{eq:plugin-remainder}
R_{n,m}&=P_Xr_n+(\Pev-P_X)g_n.
\end{align}
Applying \eqref{eq:conditional-evaluation-variance} conditionally on the
training sample, $(\Pev-P_X)g_n$ has mean zero and
\begin{equation}
\label{eq:plugin-conditional-variance}
E\!\left[\left.\{(\Pev-P_X)g_n\}^{2}\right|\calD_{\mathrm{tr},n}\right]
\le m^{-1}P_Xg_n^2
\le 4\tau^2m^{-1}P_X\kappa_n^2.
\end{equation}
Under \eqref{eq:plugin-rate-assumptions}, conditional Chebyshev's inequality
therefore gives $(\Pev-P_X)g_n=O_p(m^{-1/2}a_n)$, while
$P_Xr_n=O_p(b_n)$. Hence
$R_{n,m}=O_p\{b_n+m^{-1/2}a_n\}$,
proving \eqref{eq:plugin-rate-expansion}.

For the derivative, \eqref{eq:plugin-pointwise-expansion} and Cauchy--Schwarz imply
\begin{equation}
\label{eq:plugin-derivative-rate}
\left|\dotmu_{\tau,A_0}^{\mathrm{std}}[\widehat A_n-A_0]\right|
\le 2\tau\{P_X\kappa_n^2\}^{1/2}+\tau P_X\omega_n
=O_p(a_n+b_n).
\end{equation}
This proves \eqref{eq:plugin-derivative-bound}. Since
$0\le\mu_{A_0}(\tau\mid X)\le\tau$, we also have
$(\Pev-P_X)\mu_{A_0}=O_p(m^{-1/2})$.
Combining these bounds in \eqref{eq:plugin-main-expansion}, and using $a_n\downarrow0$,
gives $\widehat\mu_{n,m}^{\mathrm{plug}}-\mu_\tau(A_0)
=O_p(m^{-1/2}+a_n+b_n)$, as required in
\eqref{eq:plugin-upper-rate}. The special case follows by
substituting $m\asymp n$, $a_n=n^{-1/2}$, and $b_n=n^{-1}$.
\end{proof}

The one-step summand, single-split estimators, and cross-fitted estimators
are defined in \eqref{eq:main-psihat-def},
\eqref{eq:os-est-integrated-main}, and~\eqref{eq:crossfit-plugin-os-main}.
Their fitted residual RMST and fitted weights are those specified just
before \eqref{eq:main-psihat-def}. The exact conditional bias identity below
underlies Theorem~\ref{thm:plugin-onestep-integrated-summary}.

For an admissible law $A$, use the residual contrast $W_{\ell,A}$ of
Section~\ref{sec:covariates}. Its weight using the true censoring law is
\begin{equation}
\label{eq:rmst-weight-definition}
H_{\ell,A}(u,x)=\frac{W_{\ell,A}(u,x)}{G(u-\mid x)},\qquad \ell\in\cJ.
\end{equation}
Thus $H_{\ell,0}=H_{\ell,A_0}$ and $H_{\ell,\widehat A_n}$ is the oracle
fitted weight. The implemented $\widehat H_{\ell,n}$ uses $\widehat G_n$,
not $G$, as in \eqref{eq:main-psihat-def}.
Write $\mu_0(x)=\mu_{A_0}(x)$ and
$dM_{\ell,0}(u)=dN_\ell(u)-Y_{j(\ell)}(u)dA_{\ell,0}(u\mid X)$.
The true centered score $\phi^\lozenge$ is the one specified in part~(c) of
Theorem~\ref{thm:plugin-onestep-integrated-summary}. The centered fitted
score is $\widehat\phi_n^\lozenge(O)=\widehat\psi_n(O)-
\widehat\mu_{n,m}^{\os}(\tau)$.

\begin{theorem}
\label{thm:onestep-exact-bias}
Assume the conditional-hazard and sample-splitting setup at the beginning of
Section~\ref{app:covariate-proofs}, with the product-integral conventions of
Section~\ref{sec:setup-notation}.  Assume independent evaluation sampling, conditional independent censoring, and
$\inf_{u,x}G(u-\mid x)>0$.  Conditionally on $\calD_{\mathrm{tr},n}$, assume
that $\widehat A_n$ is measurable, that
$A_0(\cdot\mid X),\widehat A_n(\cdot\mid X)\in\Aspace_{\tau}$ a.s., and
that the fitted weights $\widehat H_{\ell,n}$ are predictable and square
integrable.  Suppose also that
\[
\sum_{\ell\in\cJ}
  P\int_0^{\tau}
  \{|\widehat H_{\ell,n}|+|H_{\ell,0}|\}
  Y_{j(\ell)}
  \,d|\widehat A_{\ell,n}-A_{\ell,0}| <\infty,
  \qquad
  P_X\omega_{\mathrm{PI}}\{A_0,\widehat A_n\}<\infty .
\]
Then
\begin{align}
\label{eq:onestep-decomposition}
  &\widehat\mu_{n,m}^{\os}-\mu_\tau(A_0)
  =(\Pev-P)\widehat\psi_n+B_n, \\
  \label{eq:onestep-bias}
  &B_n
  =
  \sum_{\ell\in\cJ}P\int_0^{\tau}
  \{H_{\ell,0}(u,X)-\widehat H_{\ell,n}(u,X)\}Y_{j(\ell)}(u)
  \,d\{\widehat A_{\ell,n}-A_{\ell,0}\}(u\mid X)
  +P_X r_n,
\end{align}
where
$|r_n(x)| \le \tau \omega_{\mathrm{PI}}\{A_0(\cdot\mid x),\widehat A_n(\cdot\mid x)\}.$
Hence the linear plug-in term cancels exactly.  Equivalently, for the hazard-induced weights that use the fitted censoring survival $\widehat G_n$,
\begin{align}
\label{eq:onestep-censoring-bias}
  B_n
  &=
  \sum_{\ell\in\cJ}P_X\int_0^\tau
  P^{A_0}_{j(\ell)}(u-\mid X)
  \{V^{\widehat A_n}_{k(\ell)}(u,\tau\mid X)-V^{\widehat A_n}_{j(\ell)}(u,\tau\mid X)\} \\
&\hspace{1.0in}\times
  \left\{1-\frac{G(u-\mid X)}{\widehat G_n(u-\mid X)}\right\}
  \,d\{\widehat A_{\ell,n}-A_{\ell,0}\}(u\mid X). \notag
\end{align}
In particular, if $\widehat G_n=G$, the conditional mean of the hazard-induced one-step summand equals $\mu_\tau(A_0)$ for every bounded admissible fitted transition law $\widehat A_n$.
\end{theorem}

\begin{proof}[Proof of Theorem~\ref{thm:onestep-exact-bias}]
Condition on the training sample and write $B=\widehat A_n$, $A=A_0$, and $D=B-A$.
The evaluation observations are then independent and identically distributed,
and
$\widehat\mu_{n,m}^{\os}-\mu_\tau(A_0) =(\Pev-P)\widehat\psi_n+B_n, \qquad B_n=P\widehat\psi_n-\mu_\tau(A_0).$
The true Doob--Meyer decomposition $dN_\ell=dM_{\ell,0}+Y_{j(\ell)}dA_{\ell,0}$ and the mean-zero property of
martingale integrals yield
\[
B_n=P_X\{\mu_B(\tau\mid X)-\mu_A(\tau\mid X)\}
+
\sum_{\ell\in\cJ}P\int_0^\tau
\widehat H_{\ell,n}(u,X)Y_{j(\ell)}(u)
\,d\{A_{\ell,0}-B_\ell\}(u\mid X).
\]
The first-order expansion at $A$ in Lemma~\ref{lem:rmst-hazard-derivative}
gives
\[
\mu_B(\tau\mid x)-\mu_A(\tau\mid x)
=
\sum_{\ell\in\cJ}\int_0^\tau
P^A_{j(\ell)}(u-\mid x)
\{V^A_{k(\ell)}-V^A_{j(\ell)}\}(u,\tau\mid x)
\,dD_\ell(u\mid x)+r_n(x),
\]
with $|r_n(x)|\le \tau\omega_{\rm PI}\{A(\cdot\mid x),B(\cdot\mid x)\}$. Conditional independent censoring implies $E\{Y_{j(\ell)}(u)\mid X=x\}=P^A_{j(\ell)}(0,u-\mid x)G(u-\mid x)$.
Substitution proves \eqref{eq:onestep-decomposition}--\eqref{eq:onestep-bias}.

For the censoring-only representation, start from the exact Duhamel identity
\[
P_B(0,t\mid x)-P_A(0,t\mid x)
=
\int_{(0,t]}P_A(0,u-\mid x)d\mathbf D(u\mid x)P_B(u,t\mid x).
\]
Multiplication by $\alive$, integration over $t\in[0,\tau]$, and Stieltjes--Fubini
give
\[
\mu_B(\tau\mid x)-\mu_A(\tau\mid x)
=
\sum_{\ell\in\cJ}\int_0^\tau
P^A_{j(\ell)}(u-\mid x)
\{V^B_{k(\ell)}-V^B_{j(\ell)}\}(u,\tau\mid x)
\,dD_\ell(u\mid x).
\]
The implemented hazard-induced weights use the same fitted law $B$ in
the fitted RMST, residual contrast, and compensator. Substitution therefore
gives \eqref{eq:onestep-censoring-bias}. When $\widehat G_n=G$, the factor in braces
in that identity vanishes, so $B_n=0$ for every bounded admissible fitted
transition law $B$.
\end{proof}

The identity \eqref{eq:onestep-censoring-bias} is the single-fold form of the hazard-induced cancellation.  It is used fold by fold in the proof of Theorem~\ref{thm:plugin-onestep-integrated-summary}.

\begin{lemma}
\label{lem:rmst-weight-variation}
Fix a covariate value $x$.  Let $A(\cdot\mid x)\in\Aspace_{\tau}$ and
assume the censoring survival satisfies $G(u-\mid x)\ge g_0>0$ and is
nonincreasing in $u$.  Let $H_{\ell,A}$ be the RMST martingale weights induced by
$A$, as in \eqref{eq:rmst-weight-definition}.  Then
\begin{equation}
\label{eq:rmst-weight-variation}
\max_{\ell\in\cJ}
\left[
\sup_{u\le\tau}|H_{\ell,A}(u,x)|
+
\PathVar\{H_{\ell,A}(\cdot,x);[0,\tau]\}
\right]
\le C_{\tau,g_0}\{1+|A(\cdot\mid x)|_{\tau}\} .
\end{equation}
Consequently, if $Y_j$ is an illness--death at-risk indicator, then
\begin{equation}
\label{eq:weighted-risk-variation}
\PathVar\{Y_j(\cdot)H_{\ell,A}(\cdot,x);[0,\tau]\}
+
\sup_{u\le\tau}|Y_j(u)H_{\ell,A}(u,x)|
\le C_{\tau,g_0}\{1+|A(\cdot\mid x)|_{\tau}\} .
\end{equation}
\end{lemma}

\begin{proof} [Proof of Lemma \ref{lem:rmst-weight-variation}]
Fix $x$ and suppress it where unambiguous.
Duhamel's identity with the zero hazard triplet, row-stochasticity, and
$\|d\mathbf A(r)\|_\infty\le2\sum_{\ell\in\cJ}dA_\ell(r\mid x)$ give
$\|P_A(s,t)-I\|_\infty \le2\sum_{\ell\in\cJ}A_\ell((s,t]\mid x)$, $0\le s<t\le\tau$.
For a partition $0=t_0<\cdots<t_m=\tau$, set
$\Delta_r|A|=\sum_{\ell\in\cJ}A_\ell((t_{r-1},t_r]\mid x)$.
Splitting the integrals defining $V_j^A$ at $t_r$ and using
$P_A(t_{r-1},t)=P_A(t_{r-1},t_r)P_A(t_r,t)$ yields
\[
|V_j^A(t_r,\tau\mid x)-V_j^A(t_{r-1},\tau\mid x)|
\le t_r-t_{r-1}+2(\tau-t_r)\Delta_r|A|
\le t_r-t_{r-1}+2\tau\Delta_r|A|.
\]
Summing over the partition and taking the supremum gives
\begin{equation}
\label{eq:V-pvar-direct}
\max_{j=0,1}\PathVar\{V_j^A(\cdot,\tau\mid x);[0,\tau]\}
\le\tau\{1+2|A(\cdot\mid x)|_\tau\}.
\end{equation}
Write $W_{01,A}=V_1^A-V_0^A$, $W_{02,A}=-V_0^A$, and
$W_{12,A}=-V_1^A$, so $H_{\ell,A}=W_{\ell,A}q_G$ with
$q_G(u,x)=G(u-\mid x)^{-1}$.
Then $\|W_{\ell,A}\|_\infty\le\tau$ and
$\PathVar(W_{\ell,A})\le2\tau\{1+2|A(\cdot\mid x)|_\tau\}$.
Monotonicity and positivity of $G$ give
$\|q_G\|_\infty\le g_0^{-1}$ and $\PathVar(q_G)\le g_0^{-1}$.
Applying
$\PathVar(fg)\le\|f\|_\infty\PathVar(g)+\|g\|_\infty\PathVar(f)$
therefore yields
\[
\max_{\ell\in\cJ}\{\|H_{\ell,A}\|_\infty+\PathVar(H_{\ell,A})\}
\le\frac{4\tau}{g_0}\{1+|A(\cdot\mid x)|_\tau\}.
\]
Finally, $\|Y_j\|_\infty\le1$ and $\PathVar(Y_j)\le2$. The same
product-variation inequality gives
\[
\PathVar(Y_jH_{\ell,A})+\|Y_jH_{\ell,A}\|_\infty
\le\PathVar(H_{\ell,A})+3\|H_{\ell,A}\|_\infty
\le\frac{6\tau}{g_0}\{1+|A(\cdot\mid x)|_\tau\},
\]
proving \eqref{eq:rmst-weight-variation} and \eqref{eq:weighted-risk-variation}.
\end{proof}

\begin{proof}[Proof of Theorem~\ref{thm:plugin-onestep-integrated-summary}]
Work on $\mathcal E_N$. Since $\Pr(\mathcal E_N)\to1$, the stochastic orders established there
also hold unconditionally. Constants may depend on $(\tau,g_0,M,K,\underline\pi, \overline\pi)$, but not on
$N$ or the realized training samples. For fold $k$, condition
on $\mathcal D_{-k}$. The observations in $I_k$ are independent of $(\widehat A^{(-k)},\widehat G^{(-k)})$
and have common law $P$. Pointwise in $x$, write
$A=A_0(\cdot\mid x),\qquad B_k=\widehat A^{(-k)}(\cdot\mid x),\qquad D_k=B_k-A .$

Let $\widehat\psi_{-k}$ be the fold-specific summand and put
$B_{k,N}=P\widehat\psi_{-k}-\mu_\tau$.
Apply \eqref{eq:onestep-censoring-bias} from
Theorem~\ref{thm:onestep-exact-bias} conditionally on $\mathcal D_{-k}$.
The fitted hazards and censoring law satisfy its measurability,
admissibility, positivity, and integrability requirements on $\mathcal E_N$;
bounded weights and hazard masses give the latter integrability.
The resulting exact identity is
\begin{align*}
B_{k,N}
&=\sum_{\ell\in\cJ}P_X\int_0^\tau
P^{A_0}_{j(\ell)}(u-\mid X)
\{V^{\widehat A^{(-k)}}_{k(\ell)}-V^{\widehat A^{(-k)}}_{j(\ell)}\}(u,\tau\mid X)\\
&\qquad\times\left\{1-\frac{G(u-\mid X)}{\widehat G^{(-k)}(u-\mid X)}\right\}
\,d\{\widehat A_\ell^{(-k)}-A_{\ell,0}\}(u\mid X).
\end{align*}
This cancellation uses the same fitted transition law in the fitted RMST,
residual weights, and compensator. In particular, $B_{k,N}=0$ when
$\widehat G^{(-k)}=G$, for every bounded admissible fitted transition law.

For part~(a), on $\mathcal E_N$ we have
$|W_{\ell,\widehat A^{(-k)}}(u,x)|\le 2\tau, \qquad P^{A_0}_{j(\ell)}(u-\mid x)\le1,$
and
$|dD_{k,\ell}| \le d\widehat A_{\ell}^{(-k)}+dA_{\ell,0}.$
Both hazard laws have total mass at most $M$, so
\begin{equation}
\label{eq:fold-bias-bound-crossfit}
\max_{1\le k\le K}|B_{k,N}|
\le C_{\tau,M}\max_{1\le k\le K}P_X\{\eta_{-k}(X)\}
=O_p(c_N).
\end{equation}
The last equality follows from Cauchy--Schwarz and
\eqref{eq:main-censoring-rate}.
Let $\mathbb P_{I_k}f=m_k^{-1}\sum_{i\in I_k}f(O_i)$. Adding and subtracting the foldwise means gives
\begin{equation}
\label{eq:crossfit-onestep-decomposition}
\widehat\mu_N^{\rm cf,os}-\mu_\tau
 =\sum_{k=1}^K\frac{m_k}{N}(\mathbb P_{I_k}-P)\widehat\psi_{-k}
+
\sum_{k=1}^K\frac{m_k}{N}B_{k,N}.
\end{equation}
On $\mathcal E_N$,
$0\le \widehat\mu_{-k}(X)\le\tau, \qquad |\widehat H_{\ell}^{(-k)}(u,X)|\le 4\tau/g_0,$
and
$$\sum_{\ell\in\cJ}\int_0^\tau Y_{j(\ell)}(u)d\widehat A_{\ell}^{(-k)}(u\mid X) \le |\widehat A^{(-k)}(\cdot\mid X)|_\tau\le M ,  \qquad \sum_{\ell\in\cJ}N_\ell(\tau)\le2.$$
Thus $|\widehat\psi_{-k}(O)|\le C_{\tau,g_0,M}$. Conditional on $\mathcal D_{-k}$,
$$E\{(\mathbb P_{I_k}-P)\widehat\psi_{-k}\mid\mathcal D_{-k}\}=0, \ \ E\left[\{(\mathbb P_{I_k}-P)\widehat\psi_{-k}\}^2\mid\mathcal D_{-k}\right] \le C m_k^{-1}.$$
Chebyshev's inequality gives $(\mathbb P_{I_k}-P)\widehat\psi_{-k}=O_p(m_k^{-1/2})$ in each fold. Since $K$ is
fixed and $m_k\asymp N$, the weighted sum of empirical terms in
\eqref{eq:crossfit-onestep-decomposition} is $O_p(N^{-1/2})$. Combining this with
\eqref{eq:fold-bias-bound-crossfit} proves \eqref{eq:main-os-rate-with-Ghat}.

For part~(b), apply the pointwise expansion and
conditional variance bound of Theorem~\ref{thm:plugin-rmst-expansion} separately in each
fold. Admissibility and bounded mass hold on $\mathcal E_N$, and
\eqref{eq:main-plugin-precise-rates} and the product-integral Lipschitz bound
give $\max_kP_X\kappa_{-k}^2=O_p(a_N^2)$ and
$\max_kP_X\omega_{-k}=O_p(b_N)$.
The exact decomposition is
\begin{align}
\label{eq:crossfit-plugin-decomposition}
\widehat\mu_N^{\rm cf,plug}(\tau)-\mu_\tau
&=(\mathbb P_N-P_X)\mu_{A_0}
+\sum_{k=1}^K\frac{m_k}{N}P_X\{\mu_{\widehat A^{(-k)}}-\mu_{A_0}\} \\
&\quad+
\sum_{k=1}^K\frac{m_k}{N}(\mathbb P_{I_k}-P_X)
\{\mu_{\widehat A^{(-k)}}-\mu_{A_0}\}. \notag
\end{align}
The first term is $O_p(N^{-1/2})$.
By \eqref{eq:plugin-pointwise-expansion} and \eqref{eq:plugin-derivative-rate},
each population mean difference in the second term is $O_p(a_N+b_N)$.
The conditional variance bound \eqref{eq:plugin-conditional-variance}
makes each foldwise empirical remainder $O_p(m_k^{-1/2}a_N)$.
Since $K$ is fixed and $m_k\asymp N$, their weighted sum is
$O_p(N^{-1/2}a_N)=o_p(N^{-1/2})$, proving \eqref{eq:main-plugin-rate}.
This uses separate foldwise conditional bounds and a finite sum,
not independence between the fold remainders.

For part~(c), impose the additional assumptions in that part and add and
subtract the true residual-RMST contrast:
\begin{align*}
B_{k,N}
&=T_{0,k}+T_{1,k}, \\
T_{0,k}
&=
\sum_{\ell\in\cJ}P_X\int_0^\tau
L_{\ell,0}(u,X)e_{G,-k}(u,X)\,dD_{k,\ell}(u\mid X), \\
T_{1,k}
&=
\sum_{\ell\in\cJ}P_X\int_0^\tau
P^{A_0}_{j(\ell)}(u-\mid X)
\{W_{\ell,\widehat A^{(-k)}}-W_{\ell,A_0}\}(u,X)
 e_{G,-k}(u,X)\,dD_{k,\ell}(u\mid X).
\end{align*}
By \eqref{eq:main-AG-transform-rate},
$\max_k|T_{0,k}|=O_p(a_Nc_N).$
The definition of $\kappa_{-k}$ gives
$\sup_{u\le\tau} |W_{\ell,\widehat A^{(-k)}}(u,x)-W_{\ell,A_0}(u,x)| \le C_\tau\kappa_{-k}(x)$.
Using the bounded-mass inequality
$\int_0^\tau |dD_{k,\ell}(u\mid x)| \le \widehat A_{\ell}^{(-k)}(\tau\mid x)+A_{\ell,0}(\tau\mid x) \le 2M$,
and Cauchy--Schwarz, we obtain
\begin{align*}
\max_k|T_{1,k}|
&\le
C_{\tau,M}\max_kP_X\{\kappa_{-k}(X)\eta_{-k}(X)\} \\
&\le
C_{\tau,M}\max_k
\{P_X\kappa_{-k}(X)^2\}^{1/2}
\{P_X\eta_{-k}(X)^2\}^{1/2}
=O_p(a_Nc_N).
\end{align*}
Consequently,
\begin{equation}
\label{eq:fold-bias-product-rate}
\max_{1\le k\le K}|B_{k,N}|=O_p(a_Nc_N).
\end{equation}

To derive the fitted-score rate, define
$\psi_0(O)=\mu_{A_0}(\tau\mid X)+ \sum_{\ell\in\cJ}\int_0^\tau H_{\ell,0}(u,X)dM_{\ell,0}(u)$, $\phi^\lozenge=\psi_0-\mu_\tau$.
For a fixed fold, the true Doob--Meyer decomposition gives
\begin{align}
\label{eq:fitted-score-decomposition}
\widehat\psi_{-k}-\psi_0
&=
\mu_{\widehat A^{(-k)}}(\tau\mid X)-\mu_{A_0}(\tau\mid X) \\
&\quad+
\sum_{\ell\in\cJ}\int_0^\tau
\{\widehat H_{\ell}^{(-k)}-H_{\ell,0}\}(u,X)\,dM_{\ell,0}(u) \notag \\
&\quad-
\sum_{\ell\in\cJ}\int_0^\tau
Y_{j(\ell)}(u)\widehat H_{\ell}^{(-k)}(u,X)\,dD_{k,\ell}(u\mid X). \notag
\end{align}
The first term is bounded by
$|\mu_{\widehat A^{(-k)}}(\tau\mid x)-\mu_{A_0}(\tau\mid x)| \le 2\tau\kappa_{-k}(x) \le C\delta_{A,-k}(x).$
For the second, positivity and the residual-RMST representation give
\[
\max_{\ell\in\cJ}\sup_{u\le\tau}
|\widehat H_{\ell}^{(-k)}(u,x)-H_{\ell,0}(u,x)|
\le C\{\kappa_{-k}(x)+\eta_{-k}(x)\}.
\]
The martingale isometry, finiteness of the transition set, and bounded
true hazard mass imply
$\left\| \sum_{\ell\in\cJ}\int_0^\tau (\widehat H_{\ell}^{(-k)}-H_{\ell,0})dM_{\ell,0} \right\|_{L_2(P)} =O_p(a_N+c_N).$
For the last term of \eqref{eq:fitted-score-decomposition}, apply
\[
\left|\int_{(0,\tau]}f(u)dD(u)\right|
\le
2\|D\|_{\infty,\tau}
\left\{\|f\|_{\infty,\tau}+\operatorname{pvar}(f;[0,\tau])\right\},
\]
which holds for finite-variation $f,D$ with $D(0)=0$.
Lemma~\ref{lem:rmst-weight-variation}, applied to the fitted transition and
censoring laws, gives uniformly on $\mathcal E_N$
\[
\sup_u|Y_{j(\ell)}(u)\widehat H_{\ell}^{(-k)}(u,X)|
+
\operatorname{pvar}\{Y_{j(\ell)}\widehat H_{\ell}^{(-k)};[0,\tau]\}
\le C_{\tau,g_0,M}.
\]
Hence
$\left| \sum_{\ell\in\cJ}\int_0^\tau Y_{j(\ell)}\widehat H_{\ell}^{(-k)}dD_{k,\ell} \right| \le C\delta_{A,-k}(X),$
with $L_2(P)$-norm $O_p(a_N)$. Combining these bounds in
\eqref{eq:fitted-score-decomposition} yields
$\max_{1\le k\le K} \|\widehat\psi_{-k}-\psi_0\|_{L_2(P)} =O_p(a_N+c_N).$
Thus fitted-score stability follows from the stated nuisance conditions.

Finally, add and subtract $\psi_0$ in \eqref{eq:crossfit-onestep-decomposition}:
\begin{equation}
\label{eq:crossfit-canonical-decomposition}
\widehat\mu_N^{\cf,\os}-\mu_\tau
 =(\mathbb P_N-P)\phi^\lozenge
+\sum_{k=1}^K\frac{m_k}{N}
 (\mathbb P_{I_k}-P)(\widehat\psi_{-k}-\psi_0)
+\sum_{k=1}^K\frac{m_k}{N}B_{k,N}.
\end{equation}
Conditional on $\mathcal D_{-k}$, the variance of the middle fold term is at most
$m_k^{-1}\|\widehat\psi_{-k}-\psi_0\|_{L_2(P)}^2.$
Since $K$ is fixed and $m_k\asymp N$,
$\sum_{k=1}^K\frac{m_k}{N} (\mathbb P_{I_k}-P)(\widehat\psi_{-k}-\psi_0) =O_p\{N^{-1/2}(a_N+c_N)\}.$
The bias bound \eqref{eq:fold-bias-product-rate} gives
$\left|\sum_{k=1}^K\frac{m_k}{N}B_{k,N}\right| \le \max_k|B_{k,N}|=O_p(a_Nc_N).$
Substitution into \eqref{eq:crossfit-canonical-decomposition} proves
\eqref{eq:main-os-canonical-expansion}. Since $a_N+c_N=o(1)$,
\eqref{eq:main-os-canonical-rate} also follows. Under $a_Nc_N=o(N^{-1/2})$, both
remainders are $o_p(N^{-1/2})$, establishing the canonical asymptotic linear
expansion.

\end{proof}

\section{Numerical details}\label{app:simulations}
This section gives the estimator implementations and data-generating
designs supporting Sections~\ref{subsec:gain-snapshot},
\ref{sec:simulations}, and~\ref{sec:application}. We first collect the
simulation conventions and estimator formulas, then give the no-covariate
and covariate scenarios. The semi-synthetic clinical-trial construction,
evaluation, and results are presented together in
Section~\ref{app:application}.

\subsection{Simulation conventions and estimator implementations}
\label{app:sim-notation}
\label{app:sim-estimators}

The data, process, and estimator notation is inherited from
Sections~\ref{sec:problem-setup}, \ref{subsec:estimation},
and~\ref{sec:covariates}; the supplementary empirical-measure notation is
fixed in Section~\ref{app:covariate-proofs}. Here $B$ denotes the number of
Monte Carlo replicates, $N$ the sample size, and $K$ the number of
cross-fitting folds. The target and administrative horizons are specified
for each design. We retain the implementation formulas below for the eight
estimators listed in \eqref{eq:eight-estimators}.

We give implementation formulas for the eight estimators in
\eqref{eq:eight-estimators}, using the notation referenced in
Section~\ref{app:sim-notation}.  For the
marginal (KM) family, the illness--death states $0,1$ are merged into a
single alive state, so $N_T$ and $Y_T$ are the only transition and at-risk
process, and $\lambda\equiv\lambda_T$, $S\equiv S_T$ are the marginal death
hazard and survival.  A covariate-indexed marginal hazard is fit by the
single-transition Cox/Breslow model
$E\{dN_T(u)\mid\calF_{u-},X\}=Y_T(u)\exp(\beta_{T,0}^\top X)\,d\Lambda_{0,T,0}(u)$,
the state-collapsed case of \eqref{eq:cox-main-transition-model}, with fitted
hazard $d\widehat A_{T,n}(u\mid x)=\exp(\widehat\beta_T^\top x)\,d\widehat\Lambda_{0,T}(u)$
as in \eqref{eq:Cox-plug-in}.

\paragraph{Marginal, no covariates (KM).}
\[
\widehat A_T(t)=\int_0^t\frac{I\{\bar Y_T(u)>0\}}{\bar Y_T(u)}\,d\bar N_T(u),
\qquad
\widehat S(t)=\prod_{0<u\le t}\{1-d\widehat A_T(u)\},
\qquad
\widehat\mu_\tau^{\rm KM}=\int_0^\tau\widehat S(t)\,dt.
\]

\paragraph{Recurrence-augmented, no covariates (AJ).}
For an initial-state distribution $(\widehat p_0,\widehat p_1)$,
\[
\widehat\theta(t)=\widehat p_0\{\widehat P_{00}(0,t)+\widehat P_{01}(0,t)\}
+\widehat p_1\widehat P_{11}(0,t),
\qquad
\widehat\mu_\tau^{\rm AJ}=\int_0^\tau\widehat\theta(t)\,dt.
\]
Under the initial-state convention $\mathcal S(0)=0$ in \eqref{eq:S-state},
this reduces to \eqref{eq:AJ}.

\paragraph{Marginal, with covariates, plug-in (KM-Cov$_{\rm plug}$).}
With $\widehat S_n(t\mid x)=\prod_{0<u\le t}\{1-d\widehat A_{T,n}(u\mid x)\}$
and $\widehat\mu_n^T(x)=\int_0^\tau\widehat S_n(t\mid x)\,dt$,
$\widehat\mu_{n,m}^{\rm KM\mbox{-}Cov,plug}(\tau)=\Pev\,\widehat\mu_n^T(X).$

\paragraph{Recurrence-augmented, with covariates, plug-in (AJ-Cov$_{\rm plug}$).}
With the state-occupation recursion of \Cref{sec:covariates} and
$\mu_{\widehat A_n}(\tau\mid x)=\int_0^\tau\{\widehat p_0(t\mid x)
+\widehat p_1(t\mid x)\}\,dt$,
$\widehat\mu_{n,m}^{\rm AJ\mbox{-}Cov,plug}(\tau) \equiv \widehat\mu_{n,m}^{\rm Cox,plug}(\tau) =\Pev\,\mu_{\widehat A_n}(\tau\mid X).$

\paragraph{Marginal, no covariates, one-step (KM$_{\rm 1s}$).}
The state-collapsed case $V_0=V_1=:V$ of \eqref{eq:main-nocov-rmst-H} gives a
single weight $\widehat H_T(u)=-\widehat V(u,\tau)/\widehat G(u-)$, with
$\widehat V(u,\tau)=\int_u^\tau\widehat S(t)/\widehat S(u)\,dt$, and
$\widehat\mu_\tau^{\rm KM_{\rm 1s}} =\widehat\mu_\tau^{\rm KM} +\Pn\int_0^\tau\widehat H_T(u) \{d\bar N_T(u)-\bar Y_T(u)\,d\widehat A_T(u)\}.$

\paragraph{Recurrence-augmented, no covariates, one-step (AJ$_{\rm 1s}$).}
With $\widehat H_{01},\widehat H_{02},\widehat H_{12}$ as in
\eqref{eq:main-nocov-rmst-H}, evaluated at $\widehat A^{AJ}$ and
$\widehat G$,
$\widehat\mu_\tau^{\rm AJ_{\rm 1s}} =\widehat\mu_\tau^{\rm AJ} +\Pn\sum_{\ell\in\cJ}\int_0^\tau\widehat H_\ell(u) \{d\bar N_\ell(u)-\bar Y_{j(\ell)}(u)\,d\widehat A_\ell^{AJ}(u)\}.$

\paragraph{Marginal, with covariates, one-step (KM-Cov$_{\rm 1s}$).}
With $\widehat V_n(u,\tau\mid x)=\int_u^\tau\widehat S_n(t\mid x)/\widehat S_n(u\mid x)\,dt$
and $\widehat H_{T,n}(u,x)=-\widehat V_n(u,\tau\mid x)/\widehat G_n(u-\mid x)$,
the state-collapsed case of \eqref{eq:main-psihat-def},
\begin{align*}
\widehat\psi_n(O)&=\widehat\mu_n^T(X)
+\int_0^\tau\widehat H_{T,n}(u,X)
\{dN_T(u)-Y_T(u)\,d\widehat A_{T,n}(u\mid X)\}, \\
\widehat\mu_{n,m}^{\rm KM\mbox{-}Cov_{\rm 1s}}(\tau)&=\Pev\,\widehat\psi_n(O).
\end{align*}

\paragraph{Recurrence-augmented, with covariates, one-step (AJ-Cov$_{\rm 1s}$).}
Using the hazard-induced summand and fitted weights of \eqref{eq:main-psihat-def},
\begin{align*}
\widehat\psi_n(O)&=\widehat\mu_n(X)
+\sum_{\ell\in\cJ}\int_0^\tau\widehat H_{\ell,n}(u,X)
\{dN_\ell(u)-Y_{j(\ell)}(u)\,d\widehat A_{\ell,n}(u\mid X)\}, \\
\widehat\mu_{n,m}^{\rm AJ\mbox{-}Cov_{\rm 1s}}(\tau)&\equiv
\widehat\mu_{n,m}^{\rm os}(\tau)=\Pev\,\widehat\psi_n(O),
\end{align*}
exactly \eqref{eq:os-est-integrated-main}.  In every case cross-fitting
replaces $\Pn$ or $\Pev$ by the fold-averaged versions of
\eqref{eq:crossfit-plugin-os-main}, with $K=5$ folds in Simulations 1 and 3 and $K=3$
folds in Simulation 2.

\subsection{No-covariate simulations}
\label{app:no-X-simulations}

We illustrate the gain identity \eqref{eq:gain} with synthetic data experiments.
To generate data, we specify the transition hazards $\lambda = (\lambda_{01}, \lambda_{02}, \lambda_{12}, \lambda_{C})$, a finite time horizon $\tau$, and a target number of subjects, $n=1000$.
We approximate  the hazards as piecewise constant on a grid over $[0,\tau]$ with step $dt=0.025$, and track the state of each of the $n$ subjects
forward through time, with all subjects starting in state 0 (alive, pre-recurrence).

In each grid cell, subjects still in state $0$ race an exponential clock with total rate $\lambda_{01}+\lambda_{02}+\lambda_C$.
If the clock fires within the cell, we choose which of the three competing causes occurred (recurrence, death, censoring), with probabilities proportional to the individual hazards, and the subject's state is updated accordingly.
Similarly, subjects who reach state 1 (alive, post-recurrence) then race a second exponential clock with rate $\lambda_{12}+\lambda_C$ for death vs.\ censoring.

In this manner, we generate $800$ Monte Carlo samples, and for each we compute the restricted mean survival time (RMST)
using two estimators: the $\mathcal{F}^\lozenge$-measurable AJ and the $\mathcal{F}^\circ$-measurable KM.
We report empirical variance reduction via boxplots (Fig. \ref{fig:gain-snapshot-combined}, third column) and compare the empirical AJ-vs-KM variance reduction, expressed as a percentage of the KM variance, against our theoretical predictions \ref{eq:M-theta-variance}, computed straight from the input $\lambda$s (Fig. \ref{fig:gain-snapshot-combined}, second column).

To illustrate the impact of censoring, recurrence, and post-recurrence prognostic lift on variance reduction, we vary the hazards according to the following equations:

\begin{equation}
\label{eq:no-X-hazards}
\begin{aligned}
  \lambda_C(t)   &= A\,e^{-(t-8)^2/4},
  \qquad
  \lambda_{01}(t) = B\,e^{-(t-4)^2/3},
  \qquad
  \lambda_{02}(t) = 0.02\,e^{-(t-2)^2/18},
  \\
  \lambda_{12}(t) &= \lambda_{02}(t) + \kappa\, b(t),
    \qquad b(t) = \bigl(0.03 + 0.0005t\bigr)\left(1 + \frac{2.4}{1+e^{-(t-11)/1.7}}\right)
\end{aligned}
\end{equation}

In the first set of experiments,
we vary the censoring hazards according to \eqref{eq:no-X-hazards}, with $A\in\{0,\,0.2,\,0.4,\,0.8\}$ for none, low, medium,
respectively high censoring (Fig.  \ref{fig:gain-snapshot-combined}, first row).
The variance gain increases proportionally with the strength of censoring at a roughly linear rate (row one, column two), but so does the variance of both the KM and the AJ estimators (row one, column three): the stronger the censoring, the harder the estimation problem becomes.

In the second set of experiments (Fig. \ref{fig:gain-snapshot-combined}, second row),
we vary the recurrence amplitude according to \eqref{eq:no-X-hazards}, with $B\in\{0,0.3,0.4,0.8\}$ for the same four levels.
Recurrence amplitude enters the efficiency gain \eqref{eq:gain} through the product $P_{00}(0,t)P_{01}(0,t)$, which is maximized when occupancy is balanced between the two states.
As recurrence increases from zero, subjects begin transitioning into state $1$, and the gain rises sharply; as recurrence increases further, we observe a slow decay, since subjects in state $1$ are much more likely to transition onward to state $2$, decreasing the occupancies for both states $0$ and $1$.

For the third set of experiments (Fig. \ref{fig:gain-snapshot-combined}, third row),
the post-recurrence death hazard is given in \eqref{eq:no-X-hazards}, with $\kappa\in\{0,0.75,1.5,2.5\}$ for none through high.
The prognostic lift affects efficiency gain through the squared value difference in \eqref{eq:gain}.
While in theory the efficiency gain is always positive, in practice, if the surrogate is not predictive of the true outcome, the AJ estimator can be noisier than KM, as we observe in the empirical curve of the second column.
This, however, is not the case as the surrogate becomes even a little predictive, and we observe a linear proportional relationship between the prognostic lift and the efficiency gain.

\subsection{Covariate scenarios}\label{app:sim-covariate-scenarios}

In all three covariate designs the state process is the illness--death
process \eqref{eq:S-state}, the target is $\mu_{20}=E(T\wedge\tau)$ at
$\tau=20$ with administrative horizon $\tau_0=30$, and
$N\in\{400,800,1600\}$.  Truth is fixed once per (scenario,$N$) cell from a
common $100{,}000$-draw scrambled-Sobol integral,
\[
\mu_{20}=\frac1{100000}\sum_{r=1}^{100000}\int_0^{20}
\{p_0(t\mid X_r)+p_1(t\mid X_r)\}\,dt,
\qquad
p_j(t\mid x)=\Pr\{\mathcal S(t)=j\mid X=x\},
\]
propagated on the same piecewise-constant grid used to generate the data,
and reported errors are $\widehat\mu-\mu_{20}$.  In all three designs the
baseline covariate vector is
$X=(X_1,\ldots,X_5)^\top\sim N_5(0,\Sigma)$, $\Sigma_{jk}=0.3^{|j-k|}$,
truncated coordinatewise to $[-2.5,2.5]$, and censoring is calibrated,
scenario by scenario, to $\approx60\%$ observed censoring; only deviations
from this common target are noted below.

\paragraph{Simulation 1: transition-law misspecification.}\label{app:sim1}

This design isolates a single mechanism: the working KM-Cov/AJ-Cov Cox
nuisance model is graded, in a known and controlled way, from correctly
specified to badly misspecified, so that the one-step correction can be
judged against a persistent (not sampling) plug-in bias.  The true
transition hazards are
$\lambda_{jk}(t\mid X)=b_{jk}(t)\exp\{\eta_{jk}(t,X)\}$ with
\[
b_{01}(t)=0.035+0.055\exp\{-(t-5)^2/18\},
\qquad
b_{02}(t)=0.020,
\qquad
b_{12}(t)=0.100,
\]
and linear predictors
\[
\eta_{01}(X)=0.25X_1-0.20X_2+0.15X_3,
\qquad
\eta_{02}(X)=0.20X_1+0.15X_2-0.15X_4,
\]
\[
\eta_{12}^{(s)}(X)=
\underbrace{0.45X_1-0.25X_2+0.15X_3+0.10X_5}_{\eta_{12}^{\rm lin}(X)}
+a_s\tanh(1.5X_1),
\quad
a_s\in\{0.50,1.50,3.00\}
\]
for weak, moderate, and strong misspecification.  The $0\to1$ and $0\to2$
laws are exactly Cox; only the $1\to2$ law departs from it, by the added
term $a_s\tanh(1.5X_1)$, while the working KM-Cov/AJ-Cov nuisance fits use a
linear Cox/Breslow model in $X$ for every transition, so the departure is
unmodeled by construction.  Censoring is $X$-free,
$\lambda_C(t\mid X)=\rho_C$, with $\rho_C$ calibrated to the common
$\approx60\%$ target (realized $\approx61\%$ in every cell).  There
are $B=100$ replicates per (strength,$N$) cell ($9\times100=900$ data sets,
7{,}200 estimator rows).

At $N=800$ under strong misspecification ($a_s=3$), AJ-Cov plug-in has bias
$-0.208$ months and RMSE $0.365$, while AJ-Cov one-step has bias $-0.026$
and RMSE $0.296$; at $N=1600$ the plug-in bias is $-0.238$ (not shrinking),
against a one-step bias of $-0.013$, so the plug-in bias is a persistent
misspecification effect rather than finite-sample noise.  Panel~A of
Figure~\ref{fig:sim-contrasts} plots this bias path across the weak,
moderate, and strong misspecification levels.

\paragraph{Simulation 2: non-Cox hazards under flexible learning.}\label{app:sim2}

This design changes the mechanism tested by Simulation 1: the
hazards are no longer close to Cox at all, and the covariate-adjusted
nuisances are learned by a flexible, cross-fitted machine-learning model
instead of a Cox/Breslow fit, so that the plug-in/one-step comparison can be
checked outside the proportional-hazards family.  Writing $\rho\in\{0.25,1,2\}$ for weak, moderate, and strong recurrence
signal, the transition baselines are
\[
b_{01}(t)=0.040+0.085\exp\{-(t-4.5)^2/12.5\},
\qquad
b_{02}(t)=0.008,
\qquad
b_{12}(t)=0.020+0.064\rho,
\]
and the true hazards contain nonlinear, non-proportional covariate and time
effects,
\begin{align*}
\lambda_{01}(t\mid X)&=b_{01}(t)\exp\{0.55\sin(1.15X_1)+0.35(X_2^2-0.8) \\
&\qquad+0.30\tanh(X_1X_4)+0.30\cos(2\pi t/\tau_0)X_3\}, \\
\lambda_{02}(t\mid X)&=0.008\exp\{-0.30\cos(0.8X_1)+0.35\tanh(X_2X_3) \\
&\qquad+0.25\sin(2\pi t/\tau_0)X_4+0.20\tanh(X_1-X_5)\}, \\
\lambda_{12}(t\mid X)&=b_{12}(t)\exp\big[\rho\{0.85\tanh(1.25X_1)+0.50\sin(X_2) \\
&\qquad+0.35\cos(2\pi t/\tau_0)X_4+0.25\tanh(X_3X_5)\}\big],
\end{align*}
so none of the three transitions has proportional-hazards form.  Censoring
is likewise non-Cox and time-varying,
\begin{align*}
\lambda_C(t\mid X)&=\rho_C\,e^{-(t-6)^2/32}
\exp\{0.40\sin(1.10X_1)+0.25\tanh(X_2^2-1) \\
&\qquad+0.20\cos(2\pi t/\tau_0)\tanh(X_3)+0.15\tanh(X_4X_5)\},
\end{align*}
with $\rho_C$ again calibrated to the common target within each scenario.
Rather than a Cox/Breslow fit, KM-Cov and AJ-Cov nuisances
are learned on person-interval data (grid width $0.10$) using a
gradient-boosted Poisson-loss regression with learning rate $0.07$, $70$
boosting iterations, at most seven leaves, minimum leaf size
$\max\{50,\min(300,\lfloor n_{\rm rows}/200\rfloor)\}$, and $\ell_2$ penalty
$5.0$, under three-fold cross-fitting.  There are $B=50$ replicates per
(signal,$N$) cell ($9\times50=450$ data sets, 3{,}600 rows).

At $N=800$ in the moderate cell ($\rho=1$), AJ-Cov plug-in has bias
$-0.113$/RMSE $0.294$ against one-step bias $-0.057$/RMSE $0.331$ (bias
correction at a finite-sample variance cost); under strong signal ($\rho=2$)
the corresponding values are approximately $-0.176/0.390$ (plug-in) and
$0.012/0.342$ (one-step), so the one-step correction is most valuable where
the learner's residual bias is largest.  Panels~B--C of
Figure~\ref{fig:sim-contrasts} show this one-step-versus-plug-in RMSE
change across signal strength.

\paragraph{Simulation 3: value and timing of recurrence.}\label{app:sim3}

This design isolates a third, independent mechanism: holding the
covariate-adjusted Cox/Breslow working model correctly specified (unlike
Simulation 1) and the hazard family fixed (unlike Simulation 2), it asks
only whether recurrence is itself prognostic for death, and whether
censoring can hide that information before it is observed.  Covariate
effects are fixed across all regimes,
\[
\eta_{01}(X)=0.80X_1-0.35X_2+0.20X_3,
\qquad
\eta_{02}(X)=\eta_{12}(X)=0.30X_1+0.10X_2-0.15X_4,
\]
with baselines
\begin{align*}
b_{01}(t)&=0.060+0.100\exp\{-(t-4.5)^2/12.5\}, \\
b_{02}(t)&=0.025,
\qquad
b_{12}\in\{0.025,\,0.080,\,0.240\}
\end{align*}
for null, moderate, and strong post-recurrence information; in the null
regime $b_{12}=b_{02}$, so state 1 carries no incremental information about
future death given $X$, and only $b_{12}$ -- not $\eta_{01},\eta_{02},\eta_{12}$
-- changes across regimes.  Censoring shares the linear predictor $\eta_{02}$
and is calibrated to the common target, with an early
Gaussian timing pulse centered at $t=6$ or, in the strong/later setting, at
$t=14$ (making recurrence more likely to be observed before censoring).
Four settings (null/early, moderate/early, strong/early, strong/later) are
each run with $B=100$ replicates per $N$ (1{,}200 data sets, 9{,}600 rows),
under five-fold cross-fitted Cox/Breslow nuisance fits as in
Simulation 1.

At $N=800$, the AJ-versus-KM one-step RMSE reduction is about $0.7\%$
(null), $10.6\%$ (moderate), and $23.3\%$ (strong/early); the
covariate-adjusted AJ-Cov-versus-KM-Cov reduction is smaller (about
$0.1\%$, $3.3\%$, $6.1\%$), since baseline covariates already explain part
of the variation recurrence would otherwise recover.  In the strong/later
cell, delaying the observability of recurrence relative to censoring makes
the inverse-censoring correction unstable: the one-step RMSE for KM-Cov and
AJ-Cov rises to approximately $1.20$ and $1.17$ months, with the
censoring-survival clipping floor active on many replicates -- a positivity
diagnostic rather than an estimator ranking.  Panels~D--F of
Figure~\ref{fig:sim-contrasts} show this recurrence-value and
censoring-timing sensitivity.

Read together, the three designs separate the two questions that the
covariate-free gain identity of
\Cref{subsec:gain-snapshot} conflates: Simulations 1 and 2 show that the
one-step correction protects the target under a wrong or merely flexible
nuisance model, while Simulation 3 shows that, even with a correctly
specified model, the AJ-versus-KM gain itself still tracks how strongly
recurrence has split the risk set and separated the state-specific
residual-survival means, collapsing to zero in the null cell and eroding
its finite-sample benefit once censoring hides recurrence before it can be observed.

\begin{figure}[h]
\centering
\includegraphics[width=\linewidth]{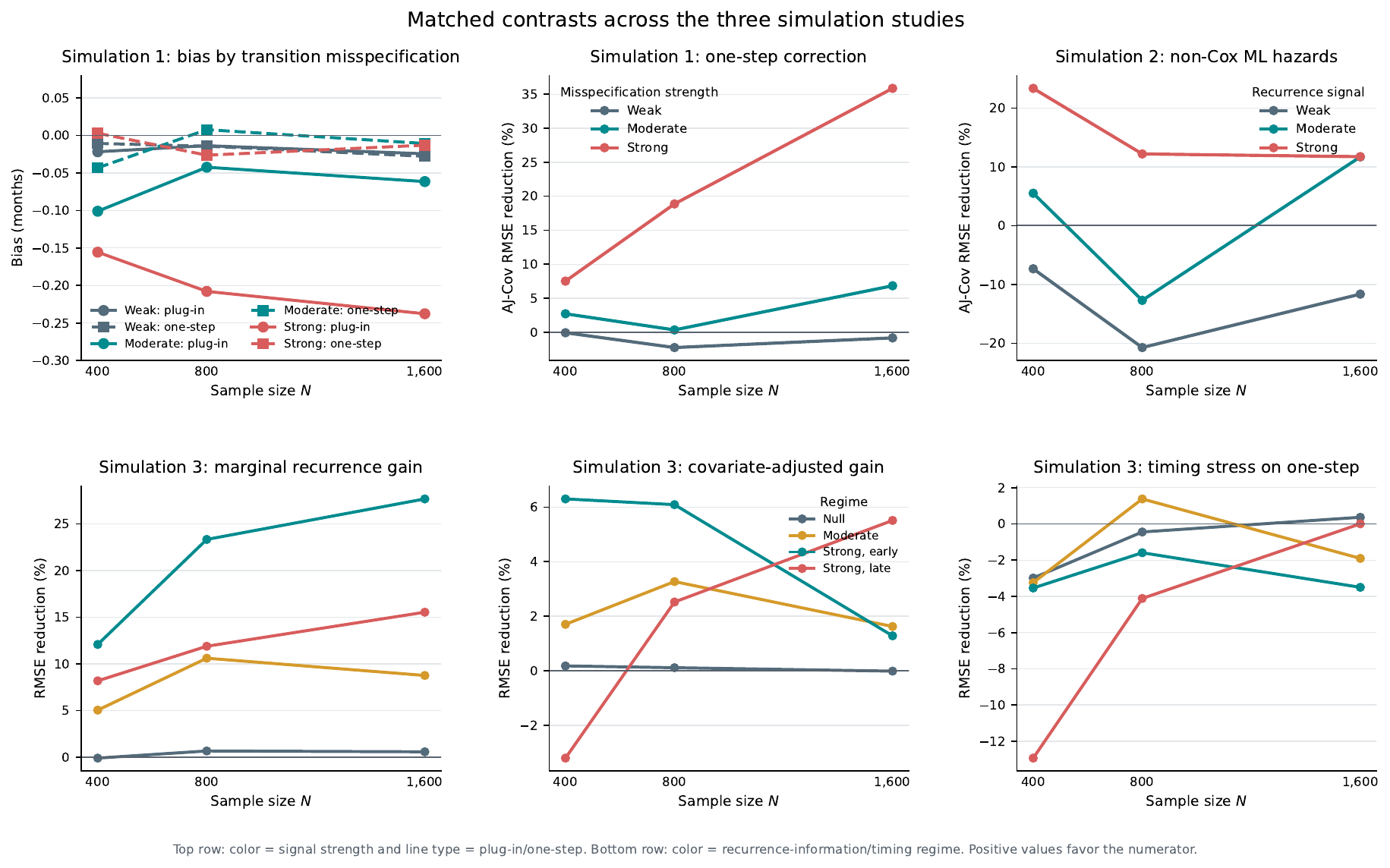}
\caption{Mechanism-specific simulation contrasts.  Panel A displays bias
across transition-law misspecification strength; panels B--C display
one-step RMSE changes; panels D--F display recurrence value and censoring
timing.}
\label{fig:sim-contrasts}
\end{figure}

\subsection{Semi-synthetic clinical-trial experiments}
\label{app:application}

We use metastatic breast cancer (MBC) clinical-trial curves to build semi-synthetic experiments to contrast the performance of our eight estimators.
The data come from a repository of 1,865 drug-therapy trials assembled by \citet{silberholz2019} and digitized by \citet{anderer2021adaptive}.
Our setting requires a two-arm randomized trial with armwise curves for the intermediate (progression/PFS) and terminal (death/OS) events.
We then conducted a quality-control audit to verify that the repository representation of each trial was sufficiently consistent with the corresponding source publication for our analysis.
This yields 40 comparisons we can simulate from.
Table~\ref{tab:rct_summary} in the Appendix summarizes key features of the patients and their outcomes in these comparisons.

From each trial arm, we calibrate a latent, arm-specific event-time law.
The public files record only the censoring-adjusted survival curves and not the underlying patient-level censoring process, so we generate censoring from a separately specified design.
The application is thus semi-synthetic in a precise sense: the event-time marginals are calibrated to the published trial curves, whereas the joint illness--death path, patient covariates, and the censoring mechanism are produced by an explicitly stated model, which we describe next.

The following material gives the trial characteristics and the curve-compatible
illness--death data-generating construction for the semi-synthetic,
MBC-anchored simulations of \Cref{sec:application}.

\subsubsection{Trial characteristics}

\begin{table}[H]
\centering
\begin{tabular}{lrrcc}
\toprule
Characteristic & Mean & Median & IQR & Range \\
\midrule
Publication year                & 2008 & 2010 & [2005, 2012]     & [1985, 2017] \\
Total Patients per comparison         & 376  & 280  & [187, 500]       & [51, 1286] \\
Median patient age (years)      & 55.6 & 55.0 & [53.0, 57.5]     & [47.5, 65.0] \\
Visceral disease (\%)$^{a}$     & 66.6 & 69.0 & [57.3, 74.7]     & [41.0, 86.0] \\
Median overall survival (mo.)$^{b}$ & 23.7 & 23.2 & [18.5, 27.9] & [10.0, 45.9] \\
\addlinespace
RMST difference, OS (mo.)       & 1.28 & 0.10 & [$-0.52$, 2.57]  & [$-4.04$, 32.59] \\
RMST difference, PFS (mo.)      & 1.53 & 1.17 & [$-0.48$, 2.92]  & [$-3.85$, 12.95] \\
\bottomrule
\end{tabular}
\vspace{0.5em}
\begin{minipage}{0.9\linewidth}
\footnotesize
$^{a}$ Available for 36 of 52 comparisons.\quad
$^{b}$ Available for 48 of 52 comparisons.
\end{minipage}
\caption{Characteristics of the 52 randomized comparisons with both PFS and OS KM curves.  Covariates are arm-size-weighted means.  RMST differences are treatment-minus-control differences in restricted mean survival time (months); positive values favor treatment.}
\label{tab:rct_summary}
\end{table}

\subsubsection{Experimental designs}

We run two experiments.
In both we simulate an illness-death process with transitions $0\to1$ (progression), $0\to2$ (death before progression), and $1\to2$ (death after progression) on a quarter-month grid, and in both we hold the arm-specific PFS and OS curves fixed while varying the joint law that generates them.
The estimand throughout is the arm-wise $\tau$-restricted mean survival time with $\tau=24$ months, and the treatment effect $\theta_1-\theta_0$.
Full details of the construction are given in Appendix~\ref{app:datagen}.

\paragraph{Controlled experiment.}
We first run controlled experiments where we fix the marginals at prespecified smooth Weibull PFS and OS curves.
The distributions for these curves are chosen to resemble the MBC trials in scale but not fitted specifically to any of them.
We calibrate a single illness--death process such that its induced marginal curves match these two Weibull distributions.
Holding both marginals smooth and fixed allows us to isolate how much the efficiency gained from observing progression depends on the joint transition mechanism alone.

\paragraph{Trial-anchored experiment.}
We then repeat this exercise using the margins from an actual trial from the processed digitized PFS and OS curves. For clarity of exposition we present one comparison from \cite{conte1996chemotherapy} in full detail. We select this example because it displays the surrogate paradox: at $\tau = 24$ months, treatment improves PFS by 1.322 months on average, while worsening OS by 0.749 months. Later, in Section~\ref{app:allTrials}, we further examine what happens when the same construction is applied to all 40 comparisons. We take only the arm-specific digitized survival values as trial inputs. Here we examine whether the mechanism established under the controlled experiment holds under the real trial margins.

\paragraph{Joint laws compatible with curve pairs.}
The pair of curves alone cannot determine the process that generated it. Progression-free survival counts both recurrence and death without recurrence without distinction. We therefore construct a family of joint laws for recurrence and death, all of which reproduce the PFS and OS margins, and present the results across that family.
Let $\pi_j$ denote the share of recurrence-free exits in interval $j$ allocated to recurrence; choosing $\pi_j$ fixes the remaining transition flows, so the compatible laws form a one-dimensional family in each interval.
We evaluate five constructions sweeping the feasible range, \emph{Lower}, \emph{Quarter}, \emph{Midpoint}, \emph{Three-quarter}, and \emph{Upper}. In the trial-anchored experiment, we evaluate two additional constructions \emph{Smooth} and \emph{Constant projection} (details can be found in~\ref{app:datagen}). The constructions differ in how many patients occupy the post-recurrence state, when they enter it, and how quickly they then die.

Each simulated patient will also carry a baseline covariate $X\sim\mathcal N(0,1)$, drawn independently of the treatment. It enters all three transition hazards through nonlinear, time-varying log-risk terms, with per-interval intercepts calibrated so that averaging over $X$ still reproduces the same PFS and OS margins; see \Cref{app:datagen} for the exact functional form and calibration.

\paragraph{Censoring generation.}
Censoring is imposed after the latent paths are drawn. We draw censoring times independently of recurrence, death, and the baseline covariate given treatment arm. Each arm follows a single two-point law, censoring at a fixed landmark with a probability calibrated so that the pooled-arm expected chance of censoring before death is $50\%$; see \Cref{app:datagen} for the exact distribution and calibration.

\subsubsection{Derivation of curve-compatible illness-death joint laws}\label{app:datagen}

Fix the horizon $\tau=24$ and a quarter-month grid
$0=t_0<t_1<\cdots<t_J=\tau$, $t_j=j\Delta,\quad \Delta=0.25,\quad \tau=24,\quad J=96$.
Given target overall-survival and progression-free-survival curves
$S_{OS},S_{PFS}:[0,\tau]\to[0,1]$ on this grid, we construct $(X,T,R)$ with
a joint illness--death law that reproduces both curves exactly.  We
suppress the arm index $a$ throughout this subsection: the entire
construction below -- allocation, covariate calibration, and path
generation -- is carried out separately for each arm, using that arm's own
$S_{\OS,a}$ and $S_{\PFS,a}$.

\paragraph{Marginal targets and allocation laws.}
Define
$p_{0,j}=S_{\PFS}(t_j)$, $p_{1,j}=S_{\OS}(t_j)-S_{\PFS}(t_j)$, $p_{2,j}=1-S_{\OS}(t_j)$.
These are the target marginal state-occupation probabilities. Further let
$M_j=p_{0,j-1}-p_{0,j},\qquad f_j=\frac{M_j}{p_{0,j-1}}$
be the unconditional state-0 exit mass and conditional state-0 exit probability. To define
a full joint law, we need to specify a recurrence allocation $\pi_j$,  which is the probability that the subject
exits state 0 to state 1. This gives the unconditional exit probabilities
$d_{01,j}=\pi_jM_j,\qquad d_{02,j}=(1-\pi_j)M_j,$
or, equivalently, the discrete hazards $q_{01,j}=\pi_jf_j$ and $q_{02,j}=(1-\pi_j)f_j$ for transitions $0\to1$ and $0\to2$.
Preservation of state 1 requires
$p_{1,j}=p_{1,j-1}+d_{01,j}-d_{12,j}, \qquad d_{12,j}=p_{1,j-1}+\pi_jM_j-p_{1,j}.$
The restrictions $0\le d_{12,j}\le p_{1,j-1}$ yield, when $M_j>0$,
\begin{equation}
\ell_j=\max\left\{0,\frac{p_{1,j}-p_{1,j-1}}{M_j}\right\}
 \le \pi_j \le
 \min\left\{1,\frac{p_{1,j}}{M_j}\right\}=u_j.
 \label{eq:bounds}
\end{equation}
If $M_j=0$, $\pi_j$ does not affect any transition flow, and the code assigns a numerical convention.  When $p_{1,j-1}>0$,
$q_{12,j}=\frac{d_{12,j}}{p_{1,j-1}} =1-\frac{p_{1,j}-p_{0,j-1}q_{01,j}}{p_{1,j-1}}.$
If the source state is empty, the code sets $q_{12,j}=0$; this is harmless,
since it multiplies zero state mass, and is a numerical convention rather
than an identified transition probability.

We consider 7 different allocation probabilities that are consistent with \eqref{eq:bounds}.  The same named rule is applied independently to each arm's own bounds $\ell_j,u_j$, so a shared construction name need not imply the same $\pi_j$ trajectory, or even the same source-state exit mass, in both arms:
\begin{align*}
\text{Lower:}\quad&\pi_j=\ell_j, \\
\text{Quarter:}\quad&\pi_j=\ell_j+0.25(u_j-\ell_j), \\
\text{Midpoint:}\quad&\pi_j=\ell_j+0.50(u_j-\ell_j), \\
\text{Three-quarter:}\quad&\pi_j=\ell_j+0.75(u_j-\ell_j), \\
\text{Upper:}\quad&\pi_j=u_j.
\end{align*}
The Smooth construction solves
$\min_{\ell_j\le\pi_j\le u_j}\sum_{j=1}^{J-1}(\pi_{j+1}-\pi_j)^2.$
It is computed by L-BFGS-B, initialized by projecting 0.5 into every interval.  The Constant-projection construction searches $v\in\{0,0.0005,\ldots,1\}$ to minimize
$\sum_{j=1}^{J}\{\operatorname{clip}(v;\ell_j,u_j)-v\}^2, \qquad \pi_j=\operatorname{clip}(v^\star;\ell_j,u_j).$
It is therefore a documented grid search, not an exact continuous optimizer.  Smooth can be nonunique, and both global laws depend on grid extent and resolution.  To prevent post-horizon choices from changing the estimand-horizon mechanism, all seven laws are defined only on $[0,\tau]$, without terminal closure.  The seven laws are illustrative feasible members, not naturally ordered by their efficiency gains and not an exhaustive or conservative range of the true joint law.

\paragraph{Covariates.}
We independently draw a baseline covariate $X\sim N(0,1)$ in each arm;
neither $X$ nor its coefficients is estimated from trial data.  Writing
$s_j=t_j/\tau$ and $z=x^2-1$, the transition-specific log-risk functions
are
\[
\begin{aligned}
g_{01}(x,s)&=0.45\{x+0.25z\}\{0.75+0.25\cos(2\pi s)\},\\
g_{02}(x,s)&=0.45\{-0.55x+0.20z\}\{0.80+0.20\sin(2\pi s)\},\\
g_{12}(x,s)&=0.45\{0.70x-0.30z\}\{0.85+0.15\cos(\pi s)\}.
\end{aligned}
\]
These effects are nonlinear and time-varying, so the data-generating law is
not a Cox proportional-hazards model.

For state 0, let $h_{01,j}(x)=a_{01,j}\exp\{g_{01}(x,s_j)\}$,
$h_{02,j}(x)=a_{02,j}\exp\{g_{02}(x,s_j)\}$, and
$H_j(x)=h_{01,j}(x)+h_{02,j}(x)$.  The conditional, individual-level
transition probabilities are the competing-risks parameterization
\begin{align}
q_{01,j}(x)&=\{1-e^{-H_j(x)}\}\frac{h_{01,j}(x)}{H_j(x)},
\label{eq:xhazards-01} \\
q_{02,j}(x)&=\{1-e^{-H_j(x)}\}\frac{h_{02,j}(x)}{H_j(x)},
\label{eq:xhazards-02} \\
q_{12,j}(x)&=1-\exp[-a_{12,j}\exp\{g_{12}(x,s_j)\}],
\label{eq:xhazards-12}
\end{align}
which enforces $q_{01,j}(x)+q_{02,j}(x)\le1$ by construction.

Let $P_{r,j-1}(x)$ be the conditional probability of state $r$ immediately
before transition $j$, and recall the marginal flows $d_{01,j}$, $d_{02,j}$,
$d_{12,j}$ defined above.  The per-interval log-intercepts
$(a_{01,j},a_{02,j},a_{12,j})$ are chosen so that averaging the conditional
hazards over $X$ reproduces those flows exactly:
\begin{align}
\E_X\{P_{0,j-1}(X)q_{01,j}(X)\}&=d_{01,j},
\label{eq:flow-calibration-01} \\
\E_X\{P_{0,j-1}(X)q_{02,j}(X)\}&=d_{02,j},
\label{eq:flow-calibration-02} \\
\E_X\{P_{1,j-1}(X)q_{12,j}(X)\}&=d_{12,j}.
\label{eq:flow-calibration-12}
\end{align}
The first two equations are solved jointly, on the log-intercept scale, by
bounded nonlinear least squares with bounds $[-32,10]$, tolerance
$10^{-13}$, and at most 200 evaluations; the third is solved by Brent's
method on the same log scale with tolerance $10^{-13}$.  Expectations are
evaluated by 41-node Gauss--Hermite quadrature, $x_h=\sqrt2\,z_h$ with
weights $w_h/\sqrt\pi$.  Because $P_{0,j-1}(x)$ and $P_{1,j-1}(x)$ depend
only on the already-calibrated hazards for intervals $1,\ldots,j-1$, the
intercepts are solved and the state probabilities propagated forward one
interval at a time, from $j=1$ to $j=J$, rather than for all $j$ at once.

\paragraph{Marginal preservation.}
Initialize $P_{0,0}(x)=1$, $P_{1,0}(x)=0$.  After interval $j$,
\[
\begin{aligned}
P_{0,j}(x)&=P_{0,j-1}(x)\{1-q_{01,j}(x)-q_{02,j}(x)\},\\
P_{1,j}(x)&=P_{1,j-1}(x)\{1-q_{12,j}(x)\}+P_{0,j-1}(x)q_{01,j}(x).
\end{aligned}
\]
If $\E_XP_{r,j-1}(X)=p_{r,j-1}$ for $r=0,1$, then
\eqref{eq:flow-calibration-01}--\eqref{eq:flow-calibration-12} give
\[
\E_XP_{0,j}(X)=p_{0,j-1}-d_{01,j}-d_{02,j}=p_{0,j},
\qquad
\E_XP_{1,j}(X)=p_{1,j-1}-d_{12,j}+d_{01,j}=p_{1,j}.
\]
Since the claim holds at $j=0$, induction gives it at every grid point, so
\[
\E_XP_{0,j}(X)=S_{\PFS}(t_j),
\qquad
\E_X\{P_{0,j}(X)+P_{1,j}(X)\}=S_{\OS}(t_j)
\]
exactly, for every $j$: the $X$-dependent law reproduces the target curves,
not merely their pointwise increments.  The implementation records the
numerical flow and curve calibration errors and stops if any flow error
exceeds $10^{-7}$ or any transition or state probability is invalid.

\paragraph{Subject-level path generation.}
Paths are drawn directly in discrete time from the calibrated conditional
probabilities \eqref{eq:xhazards-01}--\eqref{eq:xhazards-12}.  For each arm
and subject, draw $X$, initialize $Z(t_0)=0$, and visit the endpoints
$t_1,\ldots,t_J$ in order.  At each endpoint, draw one
independent $U\sim\mathrm{Uniform}(0,1)$.  If the subject is in state 0
immediately before $t_j$, set
\[
Z(t_j)=
\begin{cases}
1, & U<q_{01,j}(X),\\
2, & q_{01,j}(X)\le U<q_{01,j}(X)+q_{02,j}(X),\\
0, & \text{otherwise};
\end{cases}
\]
if the subject is in state 1, set $Z(t_j)=2$ when $U<q_{12,j}(X)$ and
retain state 1 otherwise; state 2 is absorbing.  The recurrence time $R$
and death time $T$ are each set to the first endpoint at which the
corresponding transition occurs, and to $\infty$ if no such transition
occurs by $\tau$.

\paragraph{Censoring.}
Censoring follows a single two-point law,
\[
C=
\begin{cases}
L, & \text{with probability } m,\\
\infty, & \text{with probability } 1-m,
\end{cases}
\qquad L=6,
\]
independent of $(T,R,X)$ within arm; a transition recorded at $L$ is taken
to occur before censoring, so
\[
G(t-)=P(C\ge t)=
\begin{cases}
1, & t\le L,\\
1-m, & t>L.
\end{cases}
\]
With equal arm allocation, the censoring mass is fixed once, before
simulation, as
\[
m=\frac{0.50}{\{S_{\OS,0}(L)+S_{\OS,1}(L)\}/2},
\]
so that
$\frac12\sum_{a=0}^1P(C<T\mid A=a) =\frac m2\sum_{a=0}^1S_{\OS,a}(L)=0.50$.
The distribution is therefore fixed by the marginal calibration, not tuned
per replicate; the realized censoring fraction is allowed to vary across
replicates.  A subject alive and uncensored through $\tau$ contributes the
truncated outcome $T\wedge\tau$ and is not counted as censored.

\subsubsection{Evaluation and diagnostics}\label{app:mbc-eval}

Before simulating, we first ask how much recurrence can improve efficiency under exact knowledge of the generating law for all of the constructions. We fix the construction and compare two analyses of the same population: one that uses the full recurrence history, as the AJ estimator does, and one that uses only the death information as the KM estimator does. Write $V^{\rm or}_{{\rm AJ},a}$ and $V^{\rm or}_{{\rm KM},a}$ for the smallest asymptotic variance each analysis can attain in arm $a$. With equal arm sizes, the oracle gain for the treatment contrast is
$\Gamma^{\rm or} = 1-\frac{V^{\rm or}_{{\rm AJ},0}+V^{\rm or}_{{\rm AJ},1}} {V^{\rm or}_{{\rm KM},0}+V^{\rm or}_{{\rm KM},1}}.$
We compute this from the true transition law, covariate distribution and censoring distribution so that it measures what recurrence offers under a given law.

We design a diagnostic to measure how much information recurrence carries under each construction we study. Let
$V_k(t,x)=E\{\min(T,\tau)-t \mid Z(t)=k,\ X=x\},\qquad k=0,1,$
be the expected remaining restricted survival in each living state. Their difference $D_R(t,x)=V_1(t,x)-V_0(t,x)$ is how much learning that a patient has recurred changes what we expect of them. Weighting this difference by whether both states are appreciably occupied and integrating over the horizon gives
$D_a=\frac{1}{\tau}\int_0^\tau \left[\,\mathbb{E}_X\!\left\{ \frac{p_{0,a}(t,X)\,p_{1,a}(t,X)}{p_{0,a}(t,X)+p_{1,a}(t,X)}\; D_{R,a}(t,X)^2\right\}\right]^{1/2} dt ,$
where $p_{k,a}(t,x)=P\{Z(t)=k\mid A=a,X=x\}$; we report the average of $D_0$ and $D_1$. The diagnostic is small when nearly everyone alive is in the same state, when recurrence does not change the outlook, or when the difference arises only after censoring has removed the relevant death information.

In the controlled experiment we simulate all five constructions. For the  trial-anchored experiment, we explore three of our seven constructions here; based on the residual-value diagnostic \emph{Quarter}, \emph{Smooth}, and \emph{Upper} are respectively the least, median, and most informative of the seven, so we report these.
For each experiment and law combination we simulate 500 trials with 1,000 patients per arm. Nuisance quantities are estimated by a prespecified smoothed histogram over $X$ and event time, fit on independent folds under two-fold cross-fitting. We compare four matched pairs of estimators: marginal KM against AJ plug-ins, the corresponding one-step estimators, and the covariate-standardized versions of each. Within each pair the two estimators differ only in whether they leverage recurrence history, so difference in sampling variability is dependent solely on recurrence information.

For each input marginal curve, we generate Monte Carlo trial replicates
from the calibrated, curve-compatible law above and evaluate the eight
estimators of \Cref{app:sim-estimators} on them.  Each replicate simulates
$N=1{,}000$ patients per arm, and every scenario below is evaluated from
$B=500$ independent replicates.

The oracle-gain calculation is the same functional of the true law for
every marginal input.  Let $V_{0,j}(x)$ and
$V_{1,j}(x)$ be the expected unnormalized residual survival time after
$t_j$, conditional on being in state $0$ or $1$ immediately after the
endpoint transition; with $V_{0,J}=V_{1,J}=0$, for $j=J-1,\ldots,0$,
\[
V_{0,j}(x)=\Delta+\{1-q_{01,j+1}(x)-q_{02,j+1}(x)\}V_{0,j+1}(x)
+q_{01,j+1}(x)V_{1,j+1}(x),
\]
$V_{1,j}(x)=\Delta+\{1-q_{12,j+1}(x)\}V_{1,j+1}(x),$
so that the arm-$a$ conditional plug-in mean is $\mu_a(x)=V_{0,0}(x)/\tau$.
With $\delta_{01,j}(x)=V_{1,j}(x)-V_{0,j}(x)$, $\delta_{02,j}(x)=-V_{0,j}(x)$, and
$\delta_{12,j}(x)=-V_{1,j}(x)$ (distinct from the hazard-rate building
blocks $h_{01,j},h_{02,j}$ of the Covariates paragraph above), the
transition innovation variances are
\[
\sigma^2_{0,j}(x)=q_{01,j}(x)\delta_{01,j}(x)^2+q_{02,j}(x)\delta_{02,j}(x)^2
-\{q_{01,j}(x)\delta_{01,j}(x)+q_{02,j}(x)\delta_{02,j}(x)\}^2,
\]
$\sigma^2_{1,j}(x)=q_{12,j}(x)\{1-q_{12,j}(x)\}\delta_{12,j}(x)^2,$
and the arm-$a$ within-$X$ AJ variance component is
$\mathcal R_{{\rm AJ},a}=\frac1{\tau^2}\sum_{j=1}^J\E_X\left[ \frac{P_{0,j-1}(X)\sigma^2_{0,j}(X)+P_{1,j-1}(X)\sigma^2_{1,j}(X)}{G(t_j-)} \right].$
For the death-only coarsening, with $A_{j-1}=P_{0,j-1}+P_{1,j-1}$,
$q_{D,j}=\{P_{0,j-1}q_{02,j}+P_{1,j-1}q_{12,j}\}/A_{j-1}$,
$P^+_{0,j}=P_{0,j-1}(1-q_{01,j}-q_{02,j})$,
$P^+_{1,j}=P_{1,j-1}(1-q_{12,j})+P_{0,j-1}q_{01,j}$, and continuation value
$\bar V_j=\{P^+_{0,j}V_{0,j}+P^+_{1,j}V_{1,j}\}/\{P^+_{0,j}+P^+_{1,j}\}$,
\[
\mathcal R_{{\rm KM},a}=\frac1{\tau^2}\sum_{j=1}^J\E_X\left[
\frac{A_{j-1}(X)q_{D,j}(X)\{1-q_{D,j}(X)\}\bar V_j(X)^2}{G(t_j-)}
\right].
\]
With the pooled-standardization term
$B_\Delta=\frac12\Var_X\{\mu_1(X)-\mu_0(X)\}$, the oracle contrast
variances and gain are
\begin{equation}
\mathcal V_b=B_\Delta+\mathcal R_{b,0}+\mathcal R_{b,1},
\quad b\in\{{\rm KM},{\rm AJ}\},
\qquad
\Gamma^{\rm or}=1-\frac{\mathcal V_{\rm AJ}}{\mathcal V_{\rm KM}}.
\label{eq:pooled-oracle}
\end{equation}
This is a deterministic calculation under the true law and censoring,
evaluated once per marginal input; it is not an empirical ratio from the
simulated replicates.

\subsubsection{Results}

All results below report the same quantity: the proportional reduction in the variance of the estimated treatment effect obtained by using the recurrence history. Unless noted otherwise, gains quoted in the text refer to the covariate-adjusted one-step pair.

\begin{figure}[t]
\centering
\includegraphics[width=\linewidth]{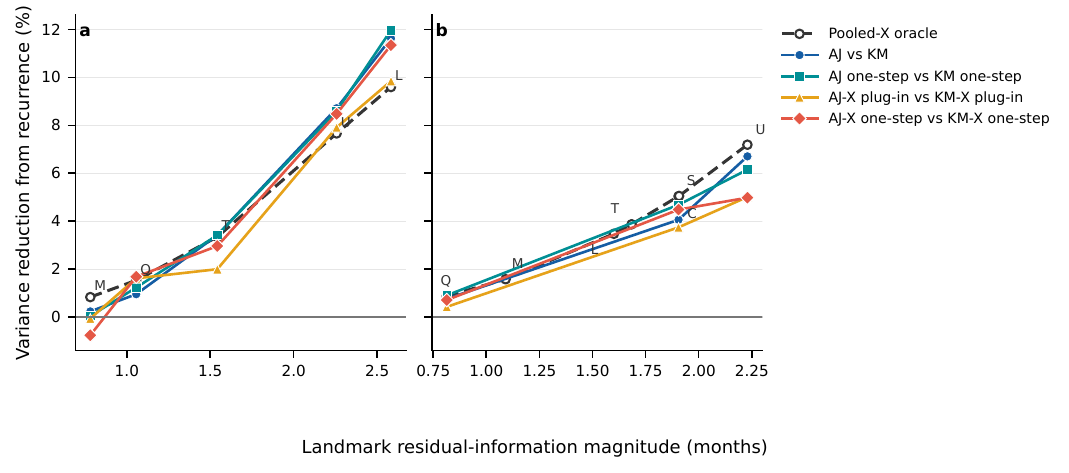}
\caption{Variance gain against the residual-value diagnostic for every
matched estimator pair.  Panel (a): the five pointwise transition
allocations under the controlled Weibull-type margins.  Panel (b): all
seven CONTE0796-compatible allocations; oracle gains are shown for all
seven, and simulated gains for \emph{Quarter}, \emph{Smooth}, and
\emph{Upper}.  Black dashed curves are oracle gains; colored points are
$500$-replicate simulated gains.}
\label{fig:mechanism}
\end{figure}

Figure~\ref{fig:mechanism} plots the variance gain against the residual-value diagnostic for every matched estimator pair, with the results from the controlled experiment in panel~(a) and the trial-anchored experiment results in panel~(b). The black dashed curves plot the oracle gains. The relationship is monotone: the rank correlation between the diagnostic and oracle gain is $1.00$ across the controlled experiment constructions and $1.00$ across the trial-anchored constructions, with no inversions in either family. All four matched pairs of estimators trace the same relationship, so the diagnostic is predictive of the gain under any of them.

Holding a single pair of synthetic curves fixed in the controlled experiment, the oracle gain runs from $0.84\%$ to $9.60\%$ across the five constructions and the simulated gain from $-0.76\%$ to $11.35\%$. Observing recurrence cannot reduce efficiency in population: the oracle gain is positive in every law and censoring configuration we computed. The slightly negative simulated value reflects Monte Carlo variability at $500$ replicates. Full results appear in \Cref{app:mbc-eval}.

Under the trial-anchored curves the same pattern appears at a more modest scale. Table~\ref{tab:gain} reports the results from the three representative trial-anchored constructions. All three constructions reproduce the PFS and OS curves from \cite{conte1996chemotherapy} exactly, yet the oracle gain rises from $0.88\%$ under \emph{Quarter} to $5.04\%$ under \emph{Smooth} and $7.19\%$ under \emph{Upper}, tracking residual-value diagnostics of $0.81$, $1.91$, and $2.23$ months. The simulated gains follow at $0.71\%$, $4.49\%$, and $4.98\%$: close to the oracle for \emph{Quarter} and \emph{Smooth}, while \emph{Upper} shows the widest gap, consistent with $500$-replicate Monte Carlo noise rather than a systematic shortfall.

\begin{figure}[t]
\centering
\includegraphics[width=\linewidth]{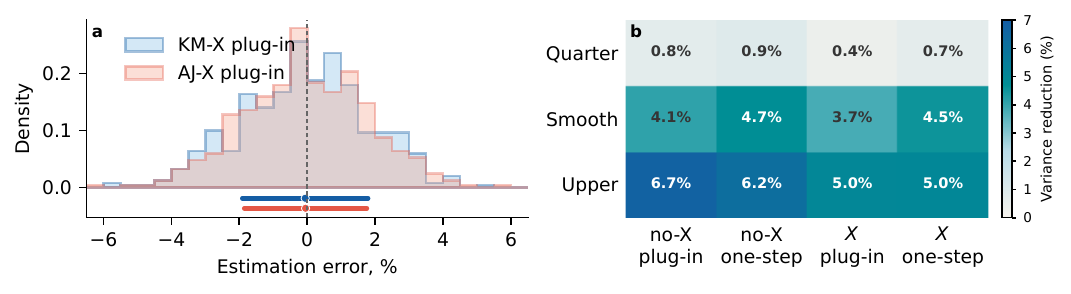}
\caption{Finite-sample comparisons under the CONTE0796-anchored
constructions.  Panel (a): sampling distributions of the marginal plug-in
estimators under \emph{Upper}, KM versus AJ.  Panel (b): variance
reduction for all twelve matched estimator-by-scenario combinations
across \emph{Quarter}, \emph{Smooth}, and \emph{Upper}.}
\label{fig:mc}
\end{figure}

Figure~\ref{fig:mc} shows an overview of the finite-sample comparisons for the trial-anchored constructions.
Panel~(a) enlarges a single representative pair, the marginal plug-in estimators under \emph{Upper}, to show what a reduction of this size looks like.
The two sampling distributions are centered together but differ in spread, a $6.7\%$ variance reduction.
Panel~(b) reports the exact reduction in all twelve scenario-by-estimator combinations.
All twelve are positive.
The ordering \emph{Quarter} $<$ \emph{Smooth} $<$ \emph{Upper} holds in every estimator family, so the ranking is a property of the generating law rather than of the analysis.
The families differ in magnitude but not in order: under \emph{Upper} the marginal estimators gain $6.7\%$ and $6.2\%$ while their covariate-adjusted counterparts gain $5.0\%$, consistent with the baseline covariate already supplying part of the same prognostic information.
Differences between plug-in and one-step estimators are smaller and run in both directions.

Collectively, these results show that observing and incorporating recurrence can never reduce efficiency, but how much it helps depends on the joint law connecting recurrence and death. This cannot simply be determined from effect sizes and PFS and OS curves.
Importantly, the estimator exploits recurrence when it is informative without requiring the joint law to be known in advance: in the worst case the benefit is small, and in the best case it is substantial.

These gains translate directly into trial size. A variance reduction of $\Gamma$ means a Kaplan--Meier analysis requires $1/(1-\Gamma)$ times as many patients to reach the same precision. Under the most informative trial-compatible construction, where $\Gamma\approx5\%$, a trial that would otherwise enroll $600$ patients reaches the same precision with roughly $570$---about $30$ fewer enrollments. At \$50{,}000 per participant \citep{moore2018estimated}, that is a saving of approximately \$1.5 million; under the intermediate construction ($\Gamma\approx4.5\%$) it is closer to \$1.3 million. In both cases the saving comes from information already collected in the course of the trial.

One might naturally ask whether the measures conventionally used to validate surrogate endpoints, individual-level and trial-level correlation \citep{prentice1989surrogate, buyse2000validation}, would tell us in advance how much incorporating recurrence might reduce variance.
However, they cannot because while they are properties of the joint law, the gain also depends on when recurrence occurs relative to when follow-up ends.
These measures were developed to assess whether a surrogate can substitute for the true endpoint.
Our question is different: overall survival remains the estimand and recurrence enters as auxiliary information.

\paragraph{Controlled Weibull-type margins.}
As a first, fully specified target we fix prespecified Weibull-type
margins,
\[
S_{\OS,0}(t)=\exp\{-(t/22)^{1.25}\},
\qquad
S_{\PFS,0}(t)=\exp\{-(t/10.5)^{1.12}\},
\]
\[
S_{\OS,1}(t)=\exp\{-(t/25)^{1.25}\},
\qquad
S_{\PFS,1}(t)=\exp\{-(t/12)^{1.12}\},
\]
with PFS replaced by its pointwise minimum with OS as a compatibility
safeguard (inactive on $[0,\tau]$).  These margins are design constants,
not fit to any trial; they resemble the MBC time scale, but their
24-month RMST contrast has the opposite sign from the CONTE0796 contrast
used below.  Under this margin we simulate all five pointwise allocations
Lower, Quarter, Midpoint, Three-quarter, and Upper.

\paragraph{CONTE0796-anchored representative constructions.}
Under the CONTE0796-anchored margins, all seven constructions of
\Cref{app:datagen} receive the oracle calculation of
\eqref{eq:pooled-oracle}, but only \emph{Quarter}, \emph{Smooth}, and
\emph{Upper} -- a low-signal interior allocation, a globally regularized
allocation, and the boundary allocation with the largest landmark signal,
respectively -- are carried to Monte Carlo evaluation; \emph{Lower},
\emph{Midpoint}, \emph{Three-quarter}, and \emph{Constant projection} have
oracle results only.  Table~\ref{tab:gain} reports the oracle and empirical,
covariate-adjusted one-step gains for the three representative
constructions of \Cref{sec:application}.

\begin{table}[H]
\centering\footnotesize
\setlength{\tabcolsep}{4pt}
\begin{tabular}{lccccc}
\toprule
Construction & Diagnostic (mo.) & Oracle gain (\%) & Empirical gain (\%) & MCSE (\%) & 95\% MC interval \\
\midrule
Quarter & 0.81 & 0.88 & 0.71 & 0.91 & [$-1.08$, $2.50$] \\
Smooth  & 1.91 & 5.04 & 4.49 & 2.11 & [$0.36$, $8.62$] \\
Upper   & 2.23 & 7.19 & 4.98 & 2.29 & [$0.49$, $9.47$] \\
\bottomrule
\end{tabular}
\caption{Oracle and empirical variance gains for the three representative
CONTE0796-anchored constructions, covariate-adjusted one-step pair.  The
diagnostic is $\{[J_0(L)+J_1(L)]/2\}^{1/2}$ at the censoring landmark
$L=6$.  Empirical gain is reported with its Monte Carlo standard error
(MCSE) and 95\% interval.}
\label{tab:gain}
\end{table}

\subsubsection{Results across trials}\label{app:allTrials}

To assess whether the CONTE0796 pattern generalizes, we repeat the same
construction across all 40 eligible comparisons with covariates, holding the estimators
and the seven transition constructions fixed and varying only the trial.
Figure~\ref{fig:alltrials} reports the resulting empirical variance gain
for every trial-by-construction combination.

\begin{figure}[t]
\centering
\includegraphics[width=\linewidth,height=0.85\textheight,keepaspectratio]{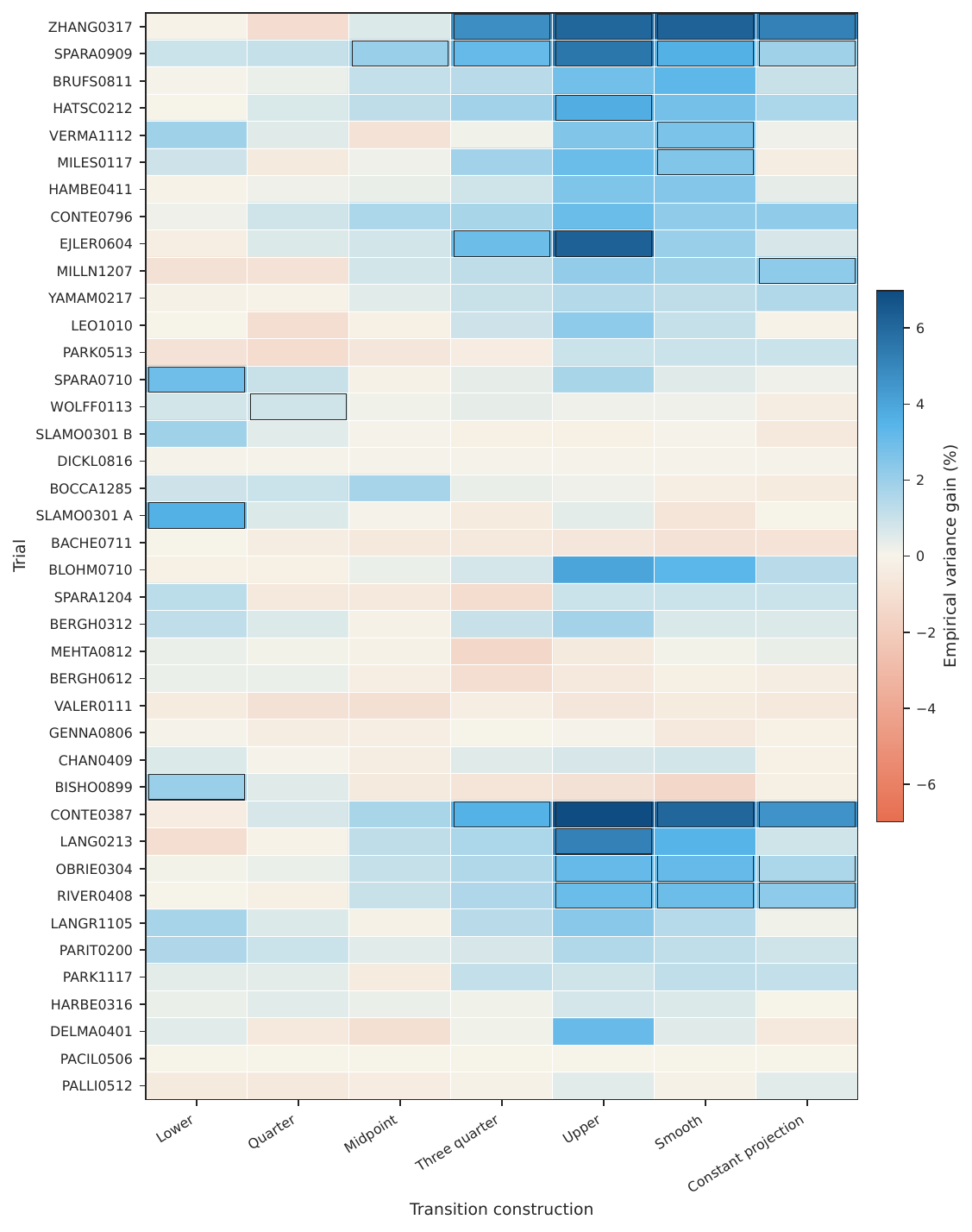}
\caption{Empirical variance gain (\%) from using recurrence, for each of
the 40 eligible trials (rows) under each of the seven transition
constructions (columns).}
\label{fig:alltrials}
\end{figure}

Figure~\ref{fig:alltrials-effects} plots the same gain (\emph{Smooth}
construction) against each trial's own 12-month PFS and OS treatment
effects.  There is no discernible relationship between the two treatment
effects and the color: trials with similar joint PFS/OS effects achieve
markedly different gains, and trials achieving similarly large gains can
have very different, even oppositely signed, PFS and OS effects.  The
correlation between the PFS and OS treatment effects is therefore not
informative about the value of recurrence; as in \Cref{sec:application},
that value is a property of the underlying joint law connecting recurrence
and death, not of the marginal effect sizes.

\begin{figure}[t]
\centering
\includegraphics[width=\linewidth]{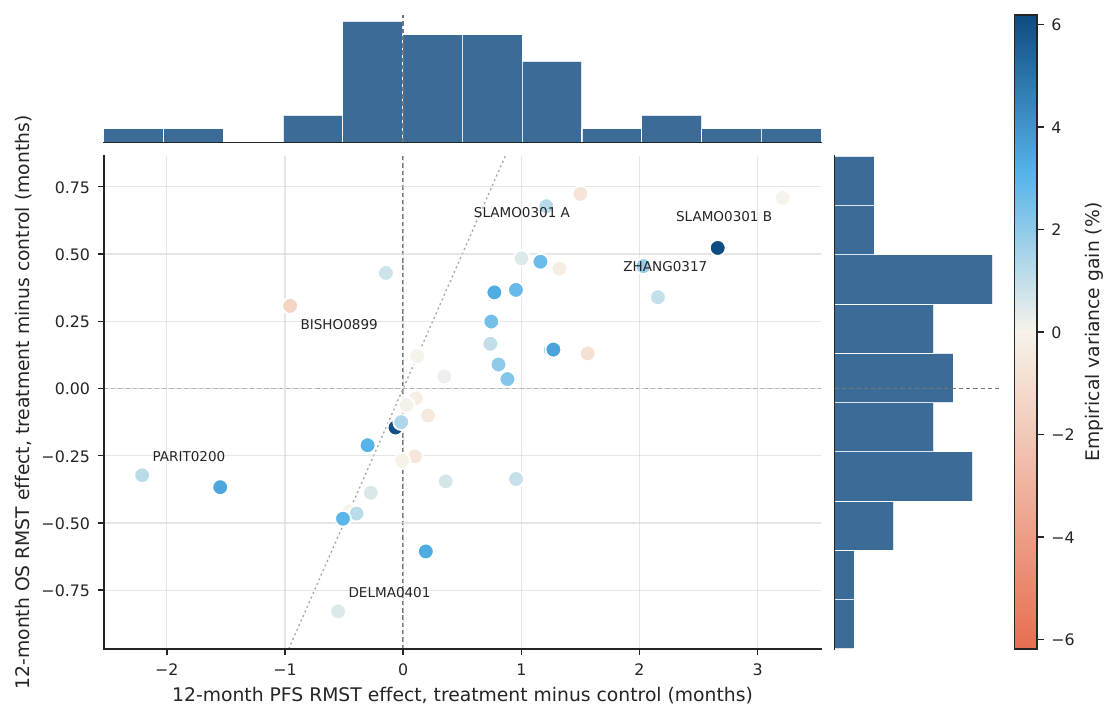}
\caption{Empirical variance gain (\%, color) from using recurrence,
\emph{Smooth} construction, plotted against each trial's own 12-month PFS
and OS treatment effects (treatment minus control, months).  Marginal
histograms show the distribution of each effect across the 40 trials.}
\label{fig:alltrials-effects}
\end{figure}

Across the resulting $40\times 7=280$ trial-by-construction cells, the
covariate-adjusted one-step empirical gain has median $0.44\%$ (from
$-0.12\%$ to $1.33\%$ over the interquartile range, and from $-1.47\%$ to
$6.98\%$ overall), and the continuous-covariate oracle gain has median
$0.70\%$ (from $0.00\%$ to $5.77\%$); the two correlate at $0.78$ across
cells, and the paired Monte Carlo interval covers the oracle value in
$92\%$ of cells, with the remaining negative empirical gains attributable
to finite-$B$ noise around a population gain that is itself never
negative. Restricted to the \emph{Smooth} construction plotted in
Figure~\ref{fig:alltrials-effects}, the 40 trial-level gains average
$1.37\%$ (median $1.01\%$), matching the oracle median of $1.38\%$ to
within Monte Carlo error. Holding each trial's marginal PFS and OS
curves -- and hence its coordinates in Figure~\ref{fig:alltrials-effects}
-- fixed while varying only the joint recurrence/death construction still
moves the gain substantially, which is the sharpest demonstration that
the PFS/OS treatment-effect pair cannot determine the value of recurrence:
the same marginal effects are compatible with many different joint laws,
and it is the joint law alone, through how the two living states are
populated at the censoring landmark and how different their remaining
survival prospects are, that sets the gain.

\clearpage
\bibliographystyle{agsm}
\bibliography{surr4surv}

\end{spacing} 
\end{document}